\documentclass{stylefile}

\usepackage{calrsfs}
\DeclareMathAlphabet{\pazocal}{OMS}{zplm}{m}{n}

\usepackage{scalerel}
\usepackage{enumerate}
\usepackage[shortlabels]{enumitem}

\allowdisplaybreaks

\usepackage{diagrams}

\usepackage[version=3]{mhchem}
\usepackage{accents}
\usepackage{bbm}
\usepackage[outline]{contour}
\usepackage{color}
\definecolor{olaj}{RGB}{70,88,104}
\definecolor{zold}{RGB}{123,79,76}
\definecolor{zold}{RGB}{151,151,49}
\definecolor{zold}{RGB}{70,123,148}
\definecolor{midgrey}{RGB}{150,173,180}
\definecolor{grey}{RGB}{232,232,232}

\usepackage{comment}

\definecolor{mydarkorange}{RGB}{0,0,0}
\definecolor{gold}{rgb}{0,0,0}
\definecolor{grey}{RGB}{0,0,0}
\definecolor{myorange}{RGB}{0,0,0}
\definecolor{mydarkorange}{RGB}{0,0,0}
\definecolor{mylightblue}{RGB}{0,0,0}
\definecolor{myyellow}{RGB}{0,0,0}
\definecolor{purple}{RGB}{0,0,0}
\definecolor{myblue}{RGB}{0,0,0}
\definecolor{mygreen}{RGB}{0,0,0}
\definecolor{brown}{RGB}{153,88,43}
\usepackage{graphicx}
\usepackage[export]{adjustbox}

\usepackage{amsmath}
\usepackage{gensymb}
\allowdisplaybreaks

\usepackage{mathtools}
\usepackage{amssymb}

\newtheorem{theorem}{Theorem}[section]
\newtheorem{lemma}[theorem]{Lemma}

\theoremstyle{definition}
\newtheorem{definition}[theorem]{Definition}

\newtheorem{proposition}[theorem]{Proposition}

\theoremstyle{remark}
\newtheorem{remark}[theorem]{Remark}

\numberwithin{equation}{section}

\newcommand*\cleartoleftpage{%
  \clearpage
  \ifodd\value{page}\hbox{}\newpage\fi
}

\newcommand{\iksz}{^{\scaleto{(\mathbf X)}{4pt}}}
\newcommand{\ipsz}{^{\scaleto{(\mathbf Y)}{4pt}}}

\newcommand{\teiksz}{\mathbin{\cdot^{\scaleto{(\mathbf X)}{4pt}}}}
\newcommand{\teipsz}{\mathbin{\cdot^{\scaleto{(\mathbf Y)}{4pt}}}}

\newcommand{\giksz}        [2] {{#1} \teiksz {#2}}
\newcommand{\gipsz}        [2] {{#1} \teipsz {#2}}

\newcommand{\negaiksz}      [1] {\negaM{{\scaleto{(\mathbf X)}{3,5pt}}}{#1}}
\newcommand{\negaipsz}      [1] {\negaM{{\scaleto{(\mathbf Y)}{3,5pt}}}{#1}}

\newcommand{\iteiksz}        {\ite{}^{\scaleto{(\mathbf X)}{4pt}}}
\newcommand{\iteipsz}        {\ite{}^{\scaleto{(\mathbf Y)}{4pt}}}

\newcommand{\resiksz}               [2] {{#1}\mathbin{\iteiksz}{#2}}
\newcommand{\resipsz}               [2] {{#1}\mathbin{\iteipsz}{#2}}

\newcommand{\komp}      {{^\prime}}

\newcommand{\kompM}[1]{^{{\prime^{\mkern-4mu^{_{#1}}}}}}
\newcommand{\nega}      [1] {{#1}\komp}
\newcommand{\negaM}      [2] {{#2}\kompM{#1}}

\newcommand{\te}{{\mathbin{*\mkern-9mu \circ}}}

\newcommand{\ite}[1]{\mathbin{\rightarrow_{#1}}}

\newcommand{\gdotu}                 [2] {{#1} \mathbin{\cdot^{\scaleto{(\mathbf X)}{4pt}}_u} {#2}}
\newcommand{\gdiamondw}                 [2] {{#1} \mathbin{\cdot^{\scaleto{(\mathbf Y)}{4pt}}_w} {#2}}
\newcommand{\gikszu}          [2] {{#1} \mathbin{\teiksz_u} {#2}}

\newcommand{\gteu}         [2] {{#1} \mathbin{\te_u} {#2}}

\newcommand{\gteM}       [3] {{#2} \mathbin{\te_{#1}} {#3}}
\newcommand{\gteMX}       [3] {{#2} \mathbin{\te_{#1}^{\scaleto{(\mathbf X)}{4pt}}} {#3}}
\newcommand{\gteMY}       [3] {{#2} \mathbin{\te_{#1}^{\scaleto{(\mathbf Y)}{4pt}}} {#3}}

\newcommand{\res}               [3] {{#2}\mathbin{\ite{#1}}{#3}}

\begin{document}
%\pagecolor{green}
\large

\title{A categorical equivalence for odd or even involutive FL$_e$-chains}

%\subtitle{Do you have a subtitle?\\ If so, write it here}

%\titlerunning{}
%\small
%\large

\begin{abstract}
We lift the representation of the class $\mathfrak I^{\mathfrak c}_{\mathfrak 0\mathfrak 1}$ of odd or even involutive FL$_e$-chains by the class $\mathfrak B_{\mathfrak G}$ of bunches of layer groups (a class of direct systems of abelian $o$-groups) published in \cite{JenRepr2020}, to a categorical equivalence between the category $\mathcal I^{\mathfrak c}_{\mathfrak 0\mathfrak 1}$ of odd or even involutive FL$_e$-chains equipped with their normal homomorphisms and the category $\mathcal B_{\mathcal G}$ of bunches of layer groups equipped with bunch homomorphisms, defined here.
Restricting this equivalence only to odd or only to even involutive FL$_e$-chains or to further subclasses thereof (e.g. to Sugihara chains) yields further categorical equivalences.

%It is shown that the representation theorem of odd or even involutive FL$_e$-chains by bunches of layer groups \cite{JenRepr2020} can be lifted to a categorical equivalence.
% by properly defining bunch homomorphisms. 
\end{abstract}

\author{S\'andor Jenei}

%\authorrunning{Short form of author list} % if too long for running head

\address{
Institute of Mathematics and Informatics, Eszterh\'azi K\'aroly Catholic University, Hungary,
and 
Institute of Mathematics and Informatics, University of P\'ecs, Hungary}



\date{}
% The correct dates will be entered by the editor

\maketitle

%\Large
%\small
%\footnotesize

\section{Introduction}

The notion of categorical equivalence plays a fundamental role in category theory.
Proving categorical equivalences is relevant, besides category theory, in several other fields of mathematics, too.
If two categories are equivalent then \lq\lq essentially the same theorems hold\rq\rq\ in the two categories;
an equivalence of categories preserves all \lq\lq categorical\rq\rq\ concepts and properties.
The functor which shows the equivalence of the categories
creates the opportunity to "translate" theorems between different kinds of mathematical structures, such that the essential meaning of those theorems is preserved under the translation.
A categorical equivalence is often used as a \lq\lq bridge theorem\rq\rq: if the objects and morphisms are better understood in one of the categories then problems in the other category can be settled by solving the \lq\lq translated\rq\rq\ version of the problem.
We present a few examples in the realm of non-classical logics, taking into consideration that topology, algebra, categories, and logic are intertwined areas.
In \cite{GR2015}, N.\,Galatos and J.\,G.\,Raftery established a category equivalence between (semilinear) generalized Sugihara monoids and nuclear relative Stone algebras. This result has been generalized by W.\,Chen in \cite{Wei2020} to a category equivalence between semiconic generalized Sugihara monoids and strong nuclear Brouwerian algebras.
A whole lot of fundamental categorical equivalences under the name Stone equivalence are between topological spaces and partially ordered sets.
A well-known categorical equivalence is the $\Gamma$ functor by D.\,Mundici between the category of MV-algebras and the category of abelian $\ell$-groups with strong unit \cite{Gamma}. This result has a precursor by C.\,C.\,Chang in \cite{chang} and
has been generalized by A.\,Dvure\v censkij in \cite{GammaArb} and by
N.\,Galatos and C.\,Tsinakis in \cite{GenMV}.
Several categorical equivalences concerning subcategories of those odd involutive residuated lattices where the unit element has a Boolean complement
have been established in \cite{CMS2008}.
In the present paper we prove a categorical equivalence between all odd or even involutive (pointed) residuated chains and a category of direct systems of abelian $o$-groups.
More precisely, we lift the representation of the class $\mathfrak I^{\mathfrak c}_{\mathfrak 0\mathfrak 1}$ of odd or even involutive FL$_e$-chains by the class $\mathfrak B_{\mathfrak G}$ of bunches of layer groups, published in \cite{JenRepr2020}, to a categorical equivalence between the category $\mathcal I^{\mathfrak c}_{\mathfrak 0\mathfrak 1}$ of odd or even involutive FL$_e$-chains equipped with their normal homomorphisms and the category $\mathcal B_{\mathcal G}$ of bunches of layer groups equipped with bunch homomorphisms, defined here.
A consequence of our results is that a rich segment of weakening-free substructural logics can be investigated by using age-old knowledge and techniques in the field of abelian $o$-groups, see \cite{JenAMALG} for example.

FL$_e$-algebras are algebraic semantics for substructural logics with the exchange property \cite{OnoBook}. In (substructural) fuzzy logics \cite{Metc Mont 2007} the chains of the related variety play a fundamental role, this makes the algebraic investigation of FL$_e$-chains a cutting edge of the research in fuzzy logics. 
%\begin{definition}
An {\em FL$_e$-algebra}\footnote{Other terminologies for FL$_e$-algebras are: pointed commutative residuated lattices or pointed commutative residuated lattice-ordered monoids.} 
is a structure $\mathbf X=( X, \leq, \cdot,\ite{}, t, f )$ such that 
$(X, \leq )$ is a lattice, $( X, \leq,\cdot,t)$ is a commutative monoid, %(the unit element $t$ is also referred to as the {\em truth} constant), 
$f$ is an arbitrary constant\footnote{Called the {\em falsum} constant.},
and
${x}{y}\leq z$ \footnote{If the multiplication operation is clear from the context, we write ${x}{y}$ for ${x}\cdot{y}$, as usual.} iff $\res{}{x}{z}\geq y$.
The latter condition, called residuation condition, is equivalent to the following: for any $x,z$, the set $\{v\ | \ {x}{v}\leq z\}$ has its greatest element, and $\res{}{x}{z}$, the residuum of $x$ and $z$, is defined as this element: $\res{}{x}{z}:=\max\{v\ | \ {x}{v}\leq z\}$. {\em Commutative residuated lattices} are the $f$-free reducts of FL$_e$-algebras. 
%Sometimes the lattice operators will be replaced by their induced ordering $\leq$ in the signature, in particular, if an FL$_e$-{\em chain} is considered, that is, when the order is total.
Being residuated implies that $\cdot$ is lattice ordered, that is $\cdot$ distributes over join.
One defines the {\em residual complement operation} by $\nega{x}=\res{}{x}{f}$ and calls an FL$_e$-algebra {\em involutive} if $\nega{(\nega{x})}=x$ holds.
In the involutive case $\res{}{x}{y}=\nega{({x}{\nega{y}})}$ holds. 
Call an element $x\geq t$ {\em positive}. 
An involutive FL$_e$-algebra is called {\em odd} if the residual complement operation leaves the unit element fixed, that is, $\nega{t}=t$, and {\em even} if the following (two) quasi-identities hold: $x<t$ $\Leftrightarrow$ $x\leq f$. 
The former condition is equivalent to $f=t$, while 
the latter quasi-identities are equivalent to assuming that 
$f$ 
is the lower cover of 
$t$ (and $t$ 
is the upper cover of 
$f$)
if chains are considered, that is, when the order is total. 
%\end{definition}

An original decomposition method along with the related construction method has been introduced in \cite{JenRepr2020} for the class of odd or even involutive FL$_e$-chains.
The main idea was to partition the algebra with the help of its local unit function $x\mapsto\res{}{x}{x}$ 
%(used slightly differently in \cite{FUZZ2017,JS_Hahn}) 
into a direct system of (hopefully simpler, \lq\lq nicer\rq\rq) algebras, indexed by the positive idempotent elements of the original algebra, with transitions of the direct system defined by multiplication with a positive idempotent element, 
and to rebuild the algebra from the direct system using a construction which partly coincides with the construction of P\l{}onka sums.
It is called layer algebra decomposition.
This idea has successfully been applied recently for other classes of residuated lattices including 
finite commutative idempotent involutive residuated lattices in \cite{JTV2021},
locally integral involutive po-monoids and semirings in \cite{GFoth},
and very recently for Bochvar algebras in \cite{BonBal}.
In these classes the layer algebras are \lq\lq nice\rq\rq.
In \cite{JenRepr2020}, however, the layer algebras are only somewhat nicer than the original algebra, therefore a second step, a layer algebras to layer groups construction (and reconstruction) phase had to be included.
The layer algebra decomposition together with this second phase shall be presented as a single step in Theorem~\ref{KATEGOR_BUNCH_X} below.
The main result of the present paper is to lift the correspondence between odd or even involutive FL$_e$-chains and bunches of layer groups, described in Theorem~\ref{KATEGOR_BUNCH_X}, to a categorical equivalence between the category 
$\mathcal I^{\mathfrak c}_{\mathfrak 0\mathfrak 1}$
of odd or even involutive FL$_e$-chains with FL$_e$-algebra homomorphisms
and 
the category 
$\mathcal B_\mathcal G$
of bunches of layer groups with bunch homomorphisms.
Restricting the related functor to subcategories yields further categorical equivalences.

%Definitions and results in this section are from \cite{JS_Hahn}, with the exception of claims~(\ref{KATEGOR_diagonalCOSinverse}), (\ref{KATEGOR_Xfcsoport}), (\ref{KATEGOR_diagonalSZIGORU}), and (\ref{KATEGOR_tau_lemma}) in Lemma~\ref{KATEGOR_tuttiINVOLUTIVE}. 

%We start with a few notations. 
%\begin{definition}\label{KATEGOR_nyilak}\rm

%Algebras will be denoted by bold capital letters, their underlying sets by the same regular letter unless otherwise stated.

\section{A bijective correspondence between $\mathfrak I^\mathfrak c_{\mathfrak 0\mathfrak1}$ and $\mathfrak B_{\mathfrak G}$}%\label{KATEGOR_SEClinkkk}

For a partially ordered set $\mathbf X=(X, \leq)$ and for $x\in X$ define the upper neighbor $x_\uparrow$ of $x$ to be the unique cover of $x$ if such exists, and $x$ otherwise.
Define $x_\downarrow$ dually.
A partially ordered algebra with a poset reduct will be called {\em discretely ordered} if for any element $x$, $x_\downarrow<x<x_\uparrow$ holds.
A directed partially ordered set is a partially ordered set such that every pair of elements has an upper bound.

\begin{definition}
Let $\langle \kappa,\leq \rangle$ be a directed partially ordered set.
Let $\{\mathbf A_i\in\mathfrak U : i\in\kappa\}$ be a family of algebras of the same type and $\varsigma_{i\to j}$ be a homomorphism\footnote{Homomorphisms are understood according to the corresponding setting. We shall call them the transitions of the direct system.} for every $i,j\in\kappa$, $i\leq j$ with the following properties:
\begin{enumerate}[(D1)]
\item\label{IDes}
$\varsigma_{i\to i}$ is the identity of $\mathbf A_i$, and
\item%\label{Kompooot}
$\varsigma_{i\to k}=\varsigma_{j\to k}\circ \varsigma_{i\to j}$ for all $i\leq j\leq k$.
\end{enumerate}
\noindent
Then $\langle \mathbf A_i,\varsigma_{i\to j} \rangle$ is called a direct system of algebras in $\mathfrak U$ over $\kappa$. 
\end{definition}
\begin{definition}\label{homodirsyst} 
Let $\pazocal A=\langle \mathbf A_i,f_{i\to j} \rangle_{\langle\alpha,\leq_\alpha\rangle}$ 
and 
$\pazocal B=\langle \mathbf B_i,g_{i\to j} \rangle_{\langle\beta,\leq_\beta\rangle}$
be two direct systems from the same class $\mathfrak U$ of algebraic systems. % and $\alpha\subseteq\beta$.
By a (direct system) homomorphism $\Phi:\pazocal A\to\pazocal B$ 
%from $\mathcal A$ to $\mathcal B$
we mean
a system of $\mathfrak U$-homomorphisms $\Phi=\{\Phi_i:A_i\to B_{\iota_o(i)} \ | \  i\in\alpha\}$
such that $\iota_o:\alpha\to\beta$ is an $o$-embedding %$\iota_o$ of $\alpha$ into $\beta$ 
and
for every $i,j\in\alpha$, $i\leq j$ the diagram in Fig.\,\ref{HomoM}
\begin{figure}[ht]
%\color{myfontcolor}
\begin{diagram}
\textbf{\textit{A$_i$}} & \rTo_{\Phi_i} & \textbf{\textit{B$_{\iota_o(i)}$}} \\
\dTo^{f_{i\to j}} & & \dTo_{g_{{\iota_o(i)}\to{\iota_o(j)}}} \\
\textbf{\textit{A$_j$}} & \rTo_{\Phi_j} & \textbf{\textit{B$_{\iota_o(j)}$}} \\
\end{diagram}
\caption{}
\label{HomoM}
\end{figure}
commutes.
%We say that $\Phi$ is an embedding if for every $i\in\alpha$, $\Phi_i$ is an embedding. Nota bene, if  $\alpha\subseteq\beta$ and $\leq_\alpha\,\subseteq\,\leq_\beta$ then throughout the paper we shall mean that $\iota_o$ is the inclusion of $\alpha$ into $\beta$.
\end{definition}
As said in the introduction, every odd or even involutive FL$_e$-chain will be represented by a bunch of layer groups in Theorem~\ref{KATEGOR_BUNCH_X}.
To this end we need the following definition.
%In a bunch of layer groups there are three pairwise disjoint sets $\kappa_o$, $\kappa_J$, and $\kappa_I$, their union $\kappa$ is totally ordered by $\leq_\kappa$ such that $\kappa$ has a least element $t$. It can be seen from the bunch of layer groups of $\mathbf X$ whether $\mathbf X$ is odd, or even with a non-idempotent falsum constant, or even with an idempotent falsum constant, by looking at whether $t$ is in $\kappa_o$, in $\kappa_J$, or in $\kappa_I$, respectively. This explains the role of Table~\ref{KATEGOR_EzAzAmaz} in the following definition.
\begin{definition}\label{KATEGOR_DEFbunch}
\cite[Definition~7.1]{JenRepr2020}
A bunch of layer groups
$${\pazocal G}=\langle \textbf{\textit{G$_u$}},\textbf{\textit{H$_u$}}, \varsigma_{u\to v} \rangle_{\langle \kappa_o, \kappa_J, \kappa_I, \leq_\kappa\rangle}$$
is a direct system 
$\langle \textbf{\textit{G$_u$}},\varsigma_{u\to v}\rangle_{\langle\kappa, \leq_\kappa\rangle}$ of abelian $o$-groups (totally ordered abelian groups)
over the totally ordered set
$\kappa=
\kappa_o\,\cup\,\kappa_J\,\cup\,\kappa_I$
($\kappa_o$, $\kappa_J$, $\kappa_I$ are pairwise disjoint)
with\begin{equation}\label{tLEAST}
\mbox{
least element $t$, 
}
\end{equation}
such that
\begin{enumerate}[start=1,label={(G\arabic*)}]
%\item 
%%$\kappa_o\cup\kappa_J\cup\kappa_I$ 
%$\kappa$ 
%has a least element $t$,
\item %\label{CsakT} 
$\kappa_o\subseteq\{t\}$,
\item\label{KATEGOR_G2}
for $v\in\kappa_I$, \textbf{\textit{H$_v$}}$\leq$\textbf{\textit{G$_v$}}
\footnote{$\textbf{\textit{H$_u$}}$'s are indexed by $\kappa_I$ only.}
and for
$u<_\kappa v\in\kappa_I$,
$\varsigma_{u\to v}$ maps into $H_v$,
\item \label{KATEGOR_DiSCRetE}
for $u\in\kappa_J$, \textbf{\textit{G$_u$}} is discrete
and for $v>_\kappa u\in\kappa_J$,
$\varsigma_{u\to v}(u)=\varsigma_{u\to v}(u_{\downarrow_u})$.
\end{enumerate}
In \ref{KATEGOR_DiSCRetE}, ${ }_{\downarrow_u}$ denotes the neighborhood operation in $\textbf{\textit{G$_u$}}$. 
Call the \textbf{\textit{G$_u$}}'s and the \textbf{\textit{H$_u$}}'s the layer groups and layer subgroups of $\pazocal G$, respectively, call $\langle\kappa,\leq_\kappa\rangle$ the {\em skeleton} of $\pazocal G$, call $\langle \kappa_o, \kappa_J, \kappa_I \rangle$ the {\em partition}\footnote{Here we (ill-)use the noun partition for the relaxed version of the familiar notion of partition, where the involved subsets are not necessarily nonempty.} of the skeleton, and call $\langle \textbf{\textit{G$_u$}}, \varsigma_{u\to v} \rangle_{\langle \kappa, \leq_\kappa\rangle}$ the direct system of $\pazocal G$.
\end{definition}
In our subsequent discussion 
we shall often refer to the type of $\varsigma_{u\to v}$
\begin{equation}\label{tipusSIGMA}
\varsigma_{u\to v} : G_u\to G_v
\end{equation}
and there will be a few classes of algebras and categories which will play a significant role. These are given distinguished notation as listed below.
$$
\begin{array}{ll}
%\mathfrak I^{\mathfrak c}_{\mathfrak 0},	\mathcal I^{\mathfrak c}_{\mathfrak 0} 	& \mbox{the class of odd involutive FL$_e$-chains and its category (with FL$_e$-algebra homomorphisms)},\\
%\mathfrak I^{\mathfrak c}_{\mathfrak 1},	\mathcal I^{\mathfrak c}_{\mathfrak 1} 	& \mbox{the class of even involutive FL$_e$-chains and its category},\\
\mathfrak I^{\mathfrak c}_{\mathfrak 0\mathfrak 1},	\mathcal I^{\mathfrak c}_{\mathfrak 0\mathfrak 1} 	& \mbox{the class of odd or even involutive FL$_e$-chains and its category}\\
& \mbox{equipped with FL$_e$-algebra homomorphisms},\\
\mathfrak B_{\mathfrak G},\mathcal B_{\mathcal G}	& \mbox{the class of bunches of layer groups and its category}\\
& \mbox{equipped with bunch homomorphisms, see Definition~\ref{DEFbunchHom}}
\end{array}
$$
%Adjunct to $\mathfrak I$, 
%\begin{itemize}[-]
%\item the superscript $\mathfrak c$ means restriction to totally ordered algebras,
%\item the superscript $\mathfrak{sl}$ (semilinear) means restriction to algebras which are subdirect products of totally ordered ones,
%\item the subscript $\mathfrak 1$ means restriction to even (that is, rank $1$) algebras,
%\item the subscript $\mathfrak 0$ means restriction to odd (that is, rank $0$) algebras.
%\item the subscript $\mathfrak0\mathfrak1$ means restriction to algebras which are odd or %even.
%\end{itemize}
%Let $\mathfrak I^{\mathfrak c}_{\mathfrak 0\mathfrak1}$ denote the class of odd or even involutive FL$_e$-chains, and let $\mathfrak B_{\mathfrak G}$ denote the class of bunches of layer groups.
Theorem~\ref{KATEGOR_BUNCH_X} demonstrates a bridge, in a constructive manner, between the classes $\mathfrak I^{\mathfrak c}_{\mathfrak 0\mathfrak1}$ and $\mathfrak B_{\mathfrak G}$.
Because of this, if $\mathcal X$ denotes the bunch constructed from the  odd or even involutive FL$_e$-algebra $\mathbf X$  then we also say that its \textbf{\textit{G$_u$}}'s and the \textbf{\textit{H$_u$}}'s are the layer groups and layer subgroups of $\mathbf X$, we call $\langle\kappa,\leq_\kappa\rangle$ the {\em skeleton} of $\mathbf X$, $\langle \kappa_o, \kappa_J, \kappa_I \rangle$ the {\em partition} of the skeleton, and call $\langle \textbf{\textit{G$_u$}}, \varsigma_{u\to v} \rangle_\kappa$ the direct system of $\mathbf X$.

%Summing up, bunches of layer groups are direct systems of abelian $o$-groups over a totally ordered set, equipped with a few extra properties. 

%In other words, each algebra in $\mathfrak I^\mathfrak c_0\cup\mathfrak I^\mathfrak c_1$ has a unique representation by its (so called) layer groups.

%In order to provide an insight on why all algebras in $\mathfrak I^\mathfrak c_0\cup\mathfrak I^\mathfrak c_1$ can be represented by bunches of layer groups, in addition to the notion of bunches of layer groups, also the notion of bunches of layer algebras has been introduced in \cite{JenRepr2020}. There, algebras in $\mathfrak I^\mathfrak c_0\cup\mathfrak I^\mathfrak c_1$ are first decomposed into their layer algebras and then the layer algebras are deformed to become layer groups.Also contrary, layer groups are deformed to become layer algebras, and then the layer algebras are put together to form an algebra in $\mathfrak I^\mathfrak c_0\cup\mathfrak I^\mathfrak c_1$. This intermediate clarifying step will not be needed in the present paper. Therefore, the two steps in the main representation theorem of \cite{JenRepr2020} are reformulated here in Theorem~\ref{KATEGOR_BUNCH_X} in a more condensed one-step form (i.e., without referring to layer algebras) to make the proof of Lemma~\ref{KATEGOR_HogyanLatszik} much shorter.

\smallskip
We shall need the following lemma.
\begin{lemma}%\label{alaszozokLEMMA}
For any bunch of layer groups
$\langle \textbf{\textit{G$_u$}},\textbf{\textit{H$_u$}}, \varsigma_{u\to v} \rangle_{\langle \kappa_o, \kappa_J, \kappa_I, \leq_\kappa\rangle}$
with
$\textbf{\textit{G$_u$}} = (G_u,\preceq_u,\cdot_u,\ { }^{-1_u},u)$, if $u\in\kappa_J$ then
\begin{equation}\label{87JHxdhFldD}
x
\cdot_u
u_{\downarrow_u}
=
x_{\downarrow_u}
,
\end{equation}
\end{lemma}
\begin{proof}
Since $u\in\kappa_J$, 
the lower cover $x_{\downarrow_u}$ of $x$ exists 
in $G_u$
for any $x\in G_u$,
see \ref{KATEGOR_DiSCRetE}.
Since $\cdot_u$
is cancellative, 
$
x
\cdot_u
u_{\downarrow_u}
\prec_u
x
\cdot_u
u
=
x
$
holds on the one hand, while on the other, 
$
x
\cdot_u
u_{\downarrow_u}
\prec_u
y
\prec_u
x
$
for some $y\in G_u$
would lead, by multiplying with $x^{-1_u}$, to an element in $G_u$ which is strictly in between $u_{\downarrow_u}$ and $u$, a contradiction.
\end{proof}

\begin{theorem}\label{KATEGOR_BUNCH_X}
%\footnotesize
{\cite[Theorem~8.1]{JenRepr2020}}
\begin{enumerate}[label={(A)}]
\item\label{KATEGOR_errefere}
Given an odd or an even involutive FL$_e$-chain $\mathbf X=(X,\leq,\cdot,\ite{},t,f)$ with residual complement operation $\komp$,
\begin{equation}\label{KATEGOR_IgyNeznekKi}
\pazocal G_{\mathbf X}=\langle \textbf{\textit{G$_u$}},\textbf{\textit{H$_u$}}, \varsigma_{u\to v} \rangle_{\langle \kappa_o, \kappa_J, \kappa_I,\leq_\kappa\rangle}
\ \  
{\mbox{with}}
\ \
\textbf{\textit{G$_u$}} = (G_u,\preceq_u,\cdot_u,\ { }^{-1_u},u)
\ \ \ 
(u\in\kappa)
\end{equation}
is bunch of layer groups, called the {\em bunch of layer groups of $\mathbf X$},
where
\begin{equation}\label{KATEGOR_IGYleszSKELETON}
\kappa=\{\res{}{x}{x} : x\in X\}=
\{u\geq t : u \mbox{ is idempotent} \} 
\mbox{ is ordered by $\leq$,}
\end{equation}
\begin{equation}\label{KATEGOR_kappaIJ}
\begin{array}{lll}
\bar\kappa_I
&=&
\{u\in \kappa\setminus\{t\} : \nega{u} \mbox{ is idempotent}\},\\
\bar\kappa_J&= & 
\{u\in \kappa\setminus\{t\} : \nega{u} \mbox{ is not idempotent}\},
\\
\end{array}
\end{equation}
$\kappa_o$, $\kappa_J$, $\kappa_I$ are defined by Table~\ref{KATEGOR_ThetaPsiOmegaAGAIN},
\begin{table}[h]
\begin{center}
\begin{tabular}{l|l|l|lll}
$\kappa_o$ \ \ \ \ \ \ \ \ & $\kappa_J$ & $\kappa_I$ & \\
\hline
\{t\} & $\bar\kappa_J$ & $\bar\kappa_I$ & if $\mathbf X$ is odd\\
\hline
$\emptyset$ & $\bar\kappa_J\cup\{t\}$ & $\bar\kappa_I$ & if $\mathbf X$ is even and $f$ is not idempotent \\
\hline
$\emptyset$ & $\bar\kappa_J$ & $\bar\kappa_I\cup\{t\}$ & if $\mathbf X$ is even and $f$ is idempotent\\
\hline
\end{tabular}
\caption{}
\label{KATEGOR_ThetaPsiOmegaAGAIN}
\end{center}
\end{table}
\\
for $u\in\kappa$,
\begin{equation}\label{KATEGOR_XHiGYkESZUL}
\begin{array}{llll}
L_u&=&\{x\in X : \res{}{x}{x}=u\},\\
H_u&=&\{x\in L_u : {x}{\nega{u}}<x\}=\{x\in L_u : {\mbox{$x$ is $u$-invertible}} \},\footnotemark\\
\accentset{\bullet}H_u&=&\{ \accentset{\bullet}x : x\in H_u\}
\mbox{
where $\accentset{\bullet}x={x}{\nega{u}}$,
}
\end{array}
\end{equation}
\footnotetext{Here and in the next row, $u\in\kappa_I$. We say that $x\in L_u$ is $u$-invertible if there is $y\in L_u$ such that $xy=u$.}
\begin{equation}\label{KATEGOR_DEFcsopi}
G_u=\left\{
\begin{array}{ll}
L_u & \mbox{if $u\notin\kappa_I$}\\
%L_u\setminus\{\gipsz{x}{\nega{u}} : x\in H_u\} & \mbox{if $u\in\kappa_I$}\\
L_u\setminus\accentset{\bullet}H_u & \mbox{if $u\in\kappa_I$}\\
\end{array}
\right.
,
\end{equation}
%$G_u=L_u$ if $u\notin\kappa_I$, $G_u=L_u\setminus\{\gipsz{x}{\nega{u}} : x\in H_u\}$ if $u\in\kappa_I$,
$$
%\begin{equation}%\label{}
\preceq_u\ =\ \leq\,\cap\ (G_u\times G_u)
%\end{equation}
$$
 \begin{equation}\label{KATEGOR_IgyTorzulaSzorzat}
x\cdot_u y=
\left\{
\begin{array}{ll}
xy & \mbox{if $u\notin\kappa_I$}\\
\res{}{(\res{}{{x}{y}}{u})}{u} & \mbox{if $u\in\kappa_I$}\\
\end{array}
\right.
,
\end{equation}
for $x\in G_u$,
\begin{equation}\label{KATEGOR_EzLeSzainVerZ}
x^{-1_u}=\res{}{x}{u},
\end{equation} 
%where 
%\\
%$L_u=\{x\in X : \tau(x)=u\}$, $\leq_u\,=\,\leq\cap \ (L_u\times L_u)$, $\teu = \te_{|L_u\times L_u}$, $\ite{\teu} = \ite{\te}_{|L_u\times L_u}$, $x^{-1_u}=\res{\teu}{x}{u}$
%\\
%$H_u=\{x\in L_u : \gipsz{x}{\nega{u}}<x\}$,
%$G_u=L_u\setminus\{\gipsz{x}{\nega{u}} : x\in H_u\}$,
%$\preceq_u\,=\,\leq\cap \ (G_u\times G_u)$, $\cdot_u = \te_{|G_u\times G_u}$, 
%\\
%$\ite{\teu} = \ite{\te}_{|L_u\times L_u}$, $x^{-1_u}=\res{\teu}{x}{u}$
%\\
%$$
%\begin{array}{lllll}
%\kappa_o=\{t\}, & \kappa_J=\bar\kappa_J, & \kappa_I=\bar\kappa_I & \mbox{if $\mathbf X$ is odd}\\
%\kappa_o=\emptyset, & \kappa_J=\bar\kappa_J\cup\{t\}, & \kappa_I=\bar\kappa_I & \mbox{if $\mathbf X$ is even and $f$ is not idempotent}\\
%\kappa_o=\emptyset, & \kappa_J=\bar\kappa_J, & \kappa_I=\bar\kappa_I\cup\{t\} & \mbox{if $\mathbf X$ is even and $f$ is idempotent}\\
%\end{array}
%$$
and for $u,v\in\kappa$ such that $u\leq_\kappa v$, $\varsigma_{u\to v} : G_u\to G_v$ is defined by
\begin{equation}\label{KATEGOR_MapAzSzorzas}
\varsigma_{u\to v}(x)={v}{x}
.
\end{equation}
\end{enumerate}

\bigskip
\begin{enumerate}[label={(B)}]
\item\label{KATEGOR_arrafere}
Given a bunch of layer groups
$$
\mbox{
$
\pazocal G=\langle \textbf{\textit{G$_u$}},\textbf{\textit{H$_u$}}, \varsigma_{u\to v} \rangle_{{\langle \kappa_o, \kappa_J, \kappa_I, \leq_\kappa\rangle}}
$
with
$
\textbf{\textit{G$_u$}}=(G_u,\preceq_u,\cdot_u,\ { }^{-1_u},u)
\ \ 
(u\in\kappa)
$,
}
$$
$$
\mathbf X_{\pazocal G}=(X,\leq,\cdot,\ite{},t,\nega{t})
$$
is an involutive FL$_e$-chain with residual complement $\komp$,
called the {\em involutive FL$_e$-chain of $\pazocal X$}
with
\begin{equation}\label{KATEGOR_EZazX}
X=\displaystyle\dot\bigcup_{u\in \kappa}L_u
,
\end{equation}
\indent where (according to Definition~\ref{KATEGOR_DEFbunch})
$$
%\begin{equation}\label{kappaMIBOLall}
\kappa=\kappa_o\cup\kappa_J\cup\kappa_I,
%\end{equation}
$$
\indent
for $u\in\kappa_I$,
%$$
\begin{equation}\label{KATEGOR_Hukeszul}
\mbox{
$\accentset{\bullet}{H}_u=\{\accentset{\bullet}{x} : x\in H_u\}$ is a copy of $H_u$
}
%\footnotemark
%\footnotetext{\ It is tacitly understood (NEM! EZ BIZONYÍTHATÓ!!!) that if $a\in A\subseteq B$ %, and hence $a\in B$, then $\accentset{\bullet}{a}$  is the same for $a\in A$ and for $a\in B$. Hence, for $A\subseteq B$, $\accentset{\bullet}{A} \subseteq \accentset{\bullet}{B}$.}
\end{equation}
%$$
%, each element $h$ of the disjointly adjuncted $H_u$ will be referred to as $\accentset{\bullet}h$ to make it apart from the element $h\in H_u\subseteq G_u$)
\indent
and for $u\in\kappa$,
\begin{equation}\label{KATEGOR_IkszU}
L_u=\left\{
\begin{array}{ll}
G_u & \mbox{ if $u\not\in\kappa_I$}\\
G_u\,%\dot
\overset{.}{\cup}\, \accentset{\bullet}H_u & \mbox{ if $u\in\kappa_I$}\\
\end{array}
\right. ,
\end{equation}

\medskip
\noindent
for $u,v\in\kappa$, $x\in L_u$ and $y\in L_v$,
\begin{equation}\label{KATEGOR_RendeZesINNOVATIVAN}
\mbox{$x<y$ iff \footnotemark \ $\rho_{uv}(x)<_{uv}\rho_{uv}(y)$ or 
($\rho_{uv}(x)=\rho_{uv}(y)$ and $u<_\kappa v$),}\\
\end{equation}\footnotetext{\ Here and also in (\ref{KATEGOR_EgySzeruTe}) $uv$ stands for $\max_{\leq_\kappa}(u,v)$. We remark that this notation does not cause any inconsistency with the notation of Theorem~\ref{KATEGOR_BUNCH_X}/\ref{KATEGOR_errefere} 
since for any two positive idempotent elements $u,v$ of an odd or even involutive FL$_e$-chain $(X,\leq,\cdot,\ite{},t,f)$ it holds true that $uv=\max_{\leq}(u,v)$ which is further equal to $\max_{\leq_\kappa}(u,v)$ by (\ref{KATEGOR_IGYleszSKELETON}).}
%\begin{equation*}
%\color{blue}
%nemjo
%\mbox{$x<y$
%iff 
%$\varsigma_{u\to uv}(\gamma_u(x))<_{uv}\varsigma_{v\to uv}(\gamma_v(y))$ or 
%$\varsigma_{u\to uv}(\gamma_u(x))=\varsigma_{v\to uv}(\gamma_v(y))$
% and $u<_\kappa v$,}
%\end{equation*}
%\begin{equation}
%x\leq y \mbox{ iff } 
%\mbox{$\rho_{u\vee v}(x)\leq_{u\vee v}\rho_{u\vee v}(y)$ except if $u>_\kappa v$ and $\rho_{u\vee v}(x)=\rho_{u\vee v}(y)$},\\
%\end{equation}
\indent
where
for $u\in\kappa$, $\gamma_u : L_u \to G_u$ is defined by
\begin{equation}\label{DEFgamma}
\gamma_u(x)=
\left\{
\begin{array}{ll}
x & \mbox{ if $x\in G_u$,}\\
a & \mbox{ if $x=\accentset{\bullet}a\in \accentset{\bullet}H_u$ (for $u\in\kappa_I$),}\\
\end{array}
\right.
\end{equation}

for $v\in\kappa$, $\rho_v : X\to X$ is defined by
\begin{equation}\label{KATEGOR_P5}
\rho_v(x)
=
\left\{
\begin{array}{ll}
x & \mbox{ if $x\in L_u$ and $u\geq_\kappa v$}\\
\varsigma_{u\to v}(\gamma_u(x)) & \mbox{ if $x\in L_u$ and $u<_\kappa v$}\\
\end{array}
\right. 
\end{equation}

and the ordering $<_u$ of $L_u$ is given by
\begin{equation}\label{KATEGOR_KibovitettRendezesITTIS}
\begin{array}{ll}
\mbox{ $\leq_u\,=\,\preceq_u$ if $u\notin\kappa_I$, whereas 
if $u\in\kappa_I$ then 
$\leq_u$ extends $\preceq_u$ %to $L_u$ 
by letting 
}\\
\mbox{
$\accentset{\bullet} a<_u\accentset{\bullet} b$ and $x<_u\accentset{\bullet} a<_uy$
for $a,b\in H_u$, $x,y\in G_u$ with $a\prec_u b$, $x\prec_u a\preceq_u y$,
}
\end{array}
\end{equation}

\medskip
\noindent
for $u,v\in\kappa$, $x\in L_u$ and $y\in L_v$,
\begin{equation}\label{KATEGOR_EgySzeruTe}
{x}{y}=\gteM{uv}{\rho_{uv}(x)}{\rho_{uv}(y)},
\end{equation}

where the
multiplication $\te_u$ on $L_u$ is defined,
for $x,y\in L_u$, by
\begin{equation}\label{KATEGOR_uPRODigy}
\gteu{x}{y}=
\left\{
\begin{array}{ll}
\left({\gamma_u(x)}\cdot_u{\gamma_u(y)}\right)^\bullet	
& \mbox{ if 
$u\in\kappa_I$, 
${\gamma_u(x)}\cdot_u{\gamma_u(y)}\in H_u$,
$\neg(x,y\in H_u)$
}\\
{\gamma_u(x)}\cdot_u{\gamma_u(y)}		& \mbox{ if $u\in\kappa_I$, ${\gamma_u(x)}\cdot_u{\gamma_u(y)}\notin H_u$ or $x,y\in H_u$}\\
x\cdot_u y& \mbox{ if $u\notin\kappa_I$}\\
\end{array}
\right. ,
\end{equation}

\medskip\noindent
for $x,y\in X$,
\begin{equation}\label{KATEGOR_IgYaReSi}
\res{}{x}{y}=\nega{({x}{\nega{y}})},
\end{equation}

where
for $x\in X$ the residual complement $\komp$ is defined by
\begin{equation}\label{KATEGOR_SplitNega}
\nega{x}
=
\left\{
\begin{array}{ll}
a^{-1_u}		& \mbox{ if $u\in\kappa_I$ and $x=\accentset{\bullet}a\in \accentset{\bullet}H_u$}\\
\left(x^{-1_u}\right)^\bullet	& \mbox{ if $u\in\kappa_I$ and $x\in H_u$}\\
x^{-1_u}	& \mbox{ if $u\in\kappa_I$ and $x\in G_u\setminus H_u$}\\
{x^{-1_u}}_{\downarrow_u} & \mbox{ if $u\in\kappa_J$ and $x\in G_u$}\\
x^{-1_u}		& \mbox{ if $u\in\kappa_o$ and $x\in G_u$}\\
\end{array}
\right. ,
\end{equation}
\begin{equation}\label{KATEGOR_tLESZez}
\mbox{
$t$ is the least element of $\kappa$,
}
\end{equation}
%$$
\begin{equation}\label{KATEGOR_tLESZaz}
\mbox{
$f$ is the residual complement of $t$.
}
\end{equation}
%\color{red} kell ez?
%and is given by
%$$
%\begin{array}{lll}
%\nega{t}&=&\left\{
%\begin{array}{ll}
%\left(t^{-1}\right)^\bullet	& \mbox{ if $u\in\kappa_I$}\\
%t^{-1}		& \mbox{ if $u\in\kappa_o$}\\
%{t^{-1}}_\downarrow		& \mbox{ if $u\in\kappa_J$}\\
%\end{array}
%\right. .
%\end{array}
%$$
In addition, 
$\mathbf X_{\pazocal G}$ is odd if $t\in\kappa_o$, 
even with a non-idempotent falsum if $t\in\kappa_J$, and 
even with an idempotent falsum if $t\in\kappa_I$.

\end{enumerate}

\begin{enumerate}[start=1,label={(C)}]
\item%\label{KATEGOR_INVeRSeS} 
%Items~\ref{KATEGOR_errefere} and \ref{KATEGOR_arrafere} describe a one-to-one correspondence between the class containing all odd and all even involutive FL$_e$-chains and the class of bunches of layer groups:
Given a bunch of layer groups $\pazocal G$ it holds true that 
$\pazocal G_{({\mathbf X}_\pazocal G)}=\pazocal G$, and
%:
given an odd or even involutive FL$_e$-chain $\mathbf X$ it holds true that $\mathbf X_{(\pazocal G_\mathbf X)}\simeq\mathbf X$.
\qed
\end{enumerate}
\end{theorem}
\begin{remark}\label{ModifiCaTO}
If Definition~\ref{KATEGOR_DEFbunch} is slightly modified in such a way that the
$\accentset{\bullet}{\textbf{\textit{H$_u$}}}$'s
are \lq\lq stored\rq\rq\ in the definition of a bunch
(like ${\pazocal G}=\langle \textbf{\textit{G$_u$}},\textbf{\textit{H$_u$}},\accentset{\bullet}{\textbf{\textit{H$_u$}}}, \varsigma_{u\to v} \rangle_{\langle \kappa_o, \kappa_J, \kappa_I, \leq_\kappa\rangle}$),
and 
instead of taking a copy $\accentset{\bullet}{H}_u$ of $H_u$ in (\ref{KATEGOR_Hukeszul}), that stored copy is used in the construction of Theorem~\ref{KATEGOR_BUNCH_X}/\ref{KATEGOR_errefere},
then also
$\mathbf X_{(\pazocal G_\mathbf X)}=\mathbf X$
holds.
Then, Theorem~\ref{KATEGOR_BUNCH_X} describes a bijection, in a constructive manner, between the classes $\mathfrak I^{\mathfrak c}_{\mathfrak 0\mathfrak1}$ and $\mathfrak B_{\mathfrak G}$.
Stating only the isomorphism will be sufficient for our purposes in the present paper.
See also Remark~\ref{ISOis}.
\end{remark}

\begin{comment}
\begin{definition}
Let 
$\mathcal I^{\mathfrak c}_{\mathfrak 0\mathfrak 1}$,
$\mathcal I^{\mathfrak c}_{\mathfrak 0}$,
$\mathcal I^{\mathfrak c}_{\mathfrak 1}$
be the categories with objects
odd or even,
odd,
even
involutive FL$_e$-chains, respectively,
 and morphisms be (FL-algebra) homomorphisms.
Also, let 
$\mathcal B$,
$\mathcal B_{\mathcal G_{\mathfrak 0}}$,
$\mathcal B_{\mathcal G_{\mathfrak 1}}$
be the category with objects 
bunches of layer groups,
odd bunches of layer groups, 
and
even bunches of layer groups, 
and morphisms bunch homomorphisms.
\end{definition}
\end{comment}

\section{Categorical equivalence between $\mathcal I^{\mathfrak c}_{\mathfrak 0\mathfrak 1}$ and $\mathcal B_\mathcal G$}

We start with a few lemmata to ease the subsequent computation.

\begin{lemma}
For $x\in X$,
(\ref{KATEGOR_SplitNega}) can equivalently be written as
\begin{equation}\label{KATEGOR_SplitNegaSIMP}
\nega{x}
=
\left\{
\begin{array}{ll}
\left(\gamma_u(x)^{-1_u}\right)^\bullet	& \mbox{ if $u\in\kappa_I$ and $x\in H_u$}\\
{\gamma_u(x)^{-1_u}}_{\downarrow_u} & \mbox{ if $u\in\kappa_J$ and $x\in G_u$}\\
%\gamma_u(x)^{-1_u}		& \mbox{ if $u\in\kappa_o$ or ($u\in\kappa_I$ and $x\in L_u\setminus H_u%(G_u\setminus H_u)\cup \accentset{\bullet}H_u 
%$)}\\
\gamma_u(x)^{-1_u}		& \mbox{ otherwise}\\
\end{array}
\right.
.
\end{equation}
For $x,y\in X$,
(\ref{KATEGOR_uPRODigy}) can equivalently be written as
\begin{equation}\label{KATEGOR_uPRODigySHORT}
\gteu{x}{y}=
\left\{
\begin{array}{ll}
\left({\gamma_u(x)}\cdot_u{\gamma_u(y)}\right)^\bullet	
& \mbox{ if 
$u\in\kappa_I$, 
${\gamma_u(x)}\cdot_u{\gamma_u(y)}\in H_u$,
$\neg(x,y\in H_u)$
}\\
{\gamma_u(x)}\cdot_u{\gamma_u(y)}		& \mbox{ otherwise}\\
\end{array}
\right.
.
\end{equation}
\end{lemma}
\begin{proof}
C.f.\,(\ref{DEFgamma}).
\end{proof}

\begin{lemma}
For $u\in\kappa_I$, the definition of the ordering in (\ref{KATEGOR_KibovitettRendezesITTIS}) can equivalently be given by any of the following ones.
\begin{equation}\label{KATEGOR_KibovitettRendezesITTIS55}
\begin{array}{ll}
\mbox{ 
$x\leq_u y$
iff
$\gamma_u(x)\prec_u \gamma_u(y)$
or 
$\left(
\mbox{
$\gamma_u(x)=\gamma_u(y)$
and 
$(x\in\accentset{\bullet}H_u$ or $y\in G_u)$
}
\right)
$
}
\end{array}
\end{equation}
\begin{equation}\label{KATEGOR_KibovitettRendezesITTIS66}
\begin{array}{ll}
\mbox{ 
$x<_u y$
iff 
$\gamma_u(x)\prec_u \gamma_u(y)$
or 
$\left(
\mbox{
$\gamma_u(x)=\gamma_u(y)$, 
$x\in\accentset{\bullet}H_u$,
$y\in G_u$
}
\right)
$
}
\end{array}
\end{equation}
\end{lemma}
\begin{proof}
For $u\in\kappa_I$, the meaning of the definition of the ordering in (\ref{KATEGOR_KibovitettRendezesITTIS}) is that for any element $a$ in a subgroup, its dotted copy $\accentset{\bullet}a$ is inserted just below $a$.
Any of (\ref{KATEGOR_KibovitettRendezesITTIS55}) and (\ref{KATEGOR_KibovitettRendezesITTIS66}) expresses the same.
\end{proof}

Throughout the paper we adopt the notation of Definition~\ref{KATEGOR_DEFbunch} and Theorem~\ref{KATEGOR_BUNCH_X}, with appropriate superscripts $^{\scaleto{(\mathbf X)}{4pt}}$ or $^{\scaleto{(\mathbf Y)}{4pt}}$ if more than one algebra is involved in the discussion.
\begin{definition}\label{DEFbunchHom}
Let 
%\begin{equation}\label{EQrend875skjdhJG}
$$
\begin{array}{l}
{\pazocal X}=\langle 
\textbf{\textit{G$^{\scaleto{(\mathbf X)}{4pt}}_u$}},
\textbf{\textit{H$^{\scaleto{(\mathbf X)}{4pt}}_u$}}, 
\varsigma_{u\to v}^{\scaleto{(\mathbf X)}{4pt}}  \rangle_{\langle \kappa_o^{\scaleto{(\mathbf X)}{4pt}}, \kappa_J^{\scaleto{(\mathbf X)}{4pt}}, \kappa_I^{\scaleto{(\mathbf X)}{4pt}},\leq_{\kappa^{\scaleto{(\mathbf X)}{3pt}}}\rangle}
\mbox{ with $\textbf{\textit{G$^{\scaleto{(\mathbf X)}{4pt}}_u$}}=(G^{\scaleto{(\mathbf X)}{4pt}}_u,\preceq^{\scaleto{(\mathbf X)}{4pt}}_u,\cdot^{\scaleto{(\mathbf X)}{4pt}}_u,\ { }^{-1^{\scaleto{(\mathbf X)}{3pt}}_u},u),%\ \ \ (u\in\kappa^{\scaleto{(\mathbf X)}{4pt}}),
$}
\\
{\pazocal Y}=\langle \textbf{\textit{G$^{\scaleto{(\mathbf Y)}{4pt}}_u$}},\textbf{\textit{H$^{\scaleto{(\mathbf Y)}{4pt}}_u$}}, \varsigma_{u\to v}^{\scaleto{(\mathbf Y)}{4pt}}  \rangle_{\langle \kappa_o^{\scaleto{(\mathbf Y)}{4pt}}, \kappa_J^{\scaleto{(\mathbf Y)}{4pt}}, \kappa_I^{\scaleto{(\mathbf Y)}{4pt}},\leq_{\kappa^{\scaleto{(\mathbf Y)}{3pt}}}\rangle}
\mbox{ with $ \textbf{\textit{G$^{\scaleto{(\mathbf Y)}{4pt}}_u$}}=(G^{\scaleto{(\mathbf Y)}{4pt}}_u,\preceq^{\scaleto{(\mathbf Y)}{4pt}}_u,\cdot^{\scaleto{(\mathbf Y)}{4pt}}_u,\ { }^{-1^{\scaleto{(\mathbf Y)}{3pt}}_u},u)%\ \ \ (u\in\kappa^{\scaleto{(\mathbf Y)}{4pt}})
$}
\end{array}
%\end{equation}
$$
\begin{comment}
$$
{\pazocal X}=\langle \textbf{\textit{G$^{\scaleto{(\mathbf X)}{4pt}}_u$}},\textbf{\textit{H$^{\scaleto{(\mathbf X)}{4pt}}_u$}}, \varsigma_{u\to v}^{\scaleto{(\mathbf X)}{4pt}}  \rangle_{\langle \kappa_o^{\scaleto{(\mathbf X)}{4pt}}, \kappa_J^{\scaleto{(\mathbf X)}{4pt}}, \kappa_I^{\scaleto{(\mathbf X)}{4pt}},\leq_{\kappa^{\scaleto{(\mathbf X)}{3pt}}}\rangle},
$$
$$
{\pazocal Y}=\langle \textbf{\textit{G$^{\scaleto{(\mathbf Y)}{4pt}}_u$}},\textbf{\textit{H$^{\scaleto{(\mathbf Y)}{4pt}}_u$}}, \varsigma_{u\to v}^{\scaleto{(\mathbf Y)}{4pt}}  \rangle_{\langle \kappa_o^{\scaleto{(\mathbf Y)}{4pt}}, \kappa_J^{\scaleto{(\mathbf Y)}{4pt}}, \kappa_I^{\scaleto{(\mathbf Y)}{4pt}},\leq_{\kappa^{\scaleto{(\mathbf Y)}{3pt}}}\rangle}
$$
be bunches of layer groups with
$$
\begin{array}{l}
\mbox{ $\textbf{\textit{G$^{\scaleto{(\mathbf X)}{4pt}}_u$}}=(G^{\scaleto{(\mathbf X)}{4pt}}_u,\preceq^{\scaleto{(\mathbf X)}{4pt}}_u,\cdot^{\scaleto{(\mathbf X)}{4pt}}_u,\ { }^{-1^{\scaleto{(\mathbf X)}{3pt}}_u},u)\ \ \ \ \ (u\in\kappa^{\scaleto{(\mathbf X)}{4pt}}),
$}
\\
\mbox{ $ \textbf{\textit{G$^{\scaleto{(\mathbf Y)}{4pt}}_u$}}=(G^{\scaleto{(\mathbf Y)}{4pt}}_u,\preceq^{\scaleto{(\mathbf Y)}{4pt}}_u,\cdot^{\scaleto{(\mathbf Y)}{4pt}}_u,\ { }^{-1^{\scaleto{(\mathbf Y)}{3pt}}_u},u)\ \ \ \ \ (u\in\kappa^{\scaleto{(\mathbf Y)}{4pt}}).
$}
\end{array}
$$
respectively.
\end{comment}
be bunches of layer groups.
Let
\begin{equation}\label{PHIdomain}
\mbox{
$
G^{\scaleto{(\mathbf X)}{4pt}}
=
\displaystyle\dot \bigcup_{u\in \kappa^{\scaleto{(\mathbf X)}{4pt}}}G^{\scaleto{(\mathbf X)}{4pt}}_u
$
and
$
G^{\scaleto{(\mathbf Y)}{4pt}}
=
\displaystyle\dot \bigcup_{s\in \kappa^{\scaleto{(\mathbf Y)}{3pt}}}G^{\scaleto{(\mathbf Y)}{4pt}}_s
$.
}
\end{equation}
We say that 
\begin{equation}\label{KATEGOR_}
\Phi : G^{\scaleto{(\mathbf X)}{4pt}} \to G^{\scaleto{(\mathbf Y)}{4pt}}
\end{equation}
%if $\Phi$ is a mapping of type $G^{\scaleto{(\mathbf X)}{4pt}}\toG^{\scaleto{(\mathbf Y)}{4pt}}$
%$$\displaystyle\dot \bigcup_{u\in \kappa^{\scaleto{(\mathbf X)}{4pt}}}G^{\scaleto{(\mathbf X)}{4pt}}_u\to\displaystyle\dot \bigcup_{u\in \kappa^{\scaleto{(\mathbf Y)}{3pt}}}G^{\scaleto{(\mathbf Y)}{4pt}}_u$$
is a {\em bunch homomorphism} from $\pazocal X$ to $\pazocal Y$
if
%its $u$-slice
%the mapping $\Phi_\kappa : \kappa^{\scaleto{(\mathbf X)}{4pt}} \to \kappa^{\scaleto{(\mathbf Y)}{4pt}}: \Phi_\kappa(u)=\Phi_u(u)$
\begin{enumerate}[start=1,label={(S\arabic*)}]
%\item\label{KATEGOR_LABs0}
%$\Phi$ respects the layers:
%\hfill
%for $u\in\kappa_I^{\scaleto{(\mathbf X)}{4pt}}$,
%$\Phi(G^{\scaleto{(\mathbf X)}{4pt}}_u)\subseteq G^{\scaleto{(\mathbf Y)}{4pt}}_{\Phi(u)}$,
\item\label{KATEGOR_embedding}
$\Phi |_{\kappa^{\scaleto{(\mathbf X)}{4pt}}}$ maps into $\kappa^{\scaleto{(\mathbf Y)}{4pt}}$ and is isotone:\hfill 
if $u,v\in{\kappa^{\scaleto{(\mathbf X)}{4pt}}}$, 
$u\leq_{\kappa^{\scaleto{(\mathbf X)}{3pt}}} v$ then $\Phi(u)\leq_{\kappa^{\scaleto{(\mathbf Y)}{3pt}}} \Phi(v)$,
\item\label{KATEGOR_homomorfizmus} 
for $u\in \kappa^{\scaleto{(\mathbf X)}{4pt}}$, 
\begin{equation}\label{PhiSlice}
\Phi_u:=\Phi|_{G^{\scaleto{(\mathbf X)}{4pt}}_u}
\end{equation}
\hfill
is an $o$-group homomorphism from 
\textbf{\textit{G$^{\scaleto{(\mathbf X)}{4pt}}_u$}} into 
%the $u^{\rm th}$ layer group \textbf{\textit{G$^{\scaleto{(\mathbf X)}{4pt}}_u$}} of $\pazocal X$ to the $\Phi(u)^{\rm th}$ layer group 
\textbf{\textit{G$^{\scaleto{(\mathbf Y)}{4pt}}_{\Phi(u)}$}},
%of $\pazocal Y$,
\item\label{KATEGOR_commutes}
%the transitions commute with $\Phi$ in a strict way\\
$\Phi$ commutes with the transitions:

for $u,v\in{\kappa^{\scaleto{(\mathbf X)}{4pt}}}$, $u\leq_{\kappa^{\scaleto{(\mathbf X)}{3pt}}}v$,
%$x\in G^{\scaleto{(\mathbf X)}{4pt}}_u$, $y\in G^{\scaleto{(\mathbf X)}{4pt}}_v$,
%\begin{enumerate}\item\label{KATEGOR_commutes_a} 
%the diagram in \ref{KATEGOR_HOMO_Kommutal} commutes.
the following diagram commutes.
\begin{figure}[ht]
\begin{diagram}
G^{\scaleto{(\mathbf X)}{4pt}}_u & \rEmbed_{\Phi_u} & G^{\scaleto{(\mathbf Y)}{4pt}}_{\Phi(u)}
\\
\dTo^{\varsigma_{u\to v}^{\scaleto{(\mathbf X)}{4pt}}} & & \dTo_{\varsigma_{{\Phi(u)}\to {\Phi(v)}}^{\scaleto{(\mathbf Y)}{4pt}}} & \\
G^{\scaleto{(\mathbf X)}{4pt}}_v & \rEmbed_{\Phi_v} & G^{\scaleto{(\mathbf Y)}{4pt}}_{\Phi(v)} \\
\end{diagram}
%\caption{}
%\label{KATEGOR_HOMO_Kommutal}
\end{figure}
%\end{enumerate}

\item\label{KATEGOR_legkisebbelem}
$\Phi$ preserves the least element of the skeleton:\ \hfill $\Phi(\min_{\leq_{\kappa^{\scaleto{(\mathbf X)}{3pt}}}} \kappa^{\scaleto{(\mathbf X)}{4pt}})=\min_{\leq_{\kappa^{\scaleto{(\mathbf Y)}{3pt}}}} \kappa^{\scaleto{(\mathbf Y)}{4pt}}$,
\item\label{KATEGOR_particio}
$\Phi$ respects the partition:
\hfill
$\Phi(\kappa_o^{\scaleto{(\mathbf X)}{4pt}})\subseteq\kappa_o^{\scaleto{(\mathbf Y)}{4pt}}$, $\Phi(\kappa_J^{\scaleto{(\mathbf X)}{4pt}})\subseteq\kappa_J^{\scaleto{(\mathbf Y)}{4pt}}\cup\kappa_o^{\scaleto{(\mathbf Y)}{4pt}}$, 
$\Phi(\kappa_I^{\scaleto{(\mathbf X)}{4pt}})\subseteq\kappa_I^{\scaleto{(\mathbf Y)}{4pt}}\cup\kappa_o^{\scaleto{(\mathbf Y)}{4pt}}$,
\item\label{KATEGOR_subgroups_and_complements}
$\Phi$ preserves the subgroups and their complements:

\hfill
if 
$u\in\kappa_I^{\scaleto{(\mathbf X)}{4pt}}$
and
$\Phi(u)\in\kappa_I^{\scaleto{(\mathbf Y)}{4pt}}$
then
$\Phi_u(H^{\scaleto{(\mathbf X)}{4pt}}_u)\subseteq H^{\scaleto{(\mathbf Y)}{4pt}}_{\Phi(u)}$
and
$\Phi_u(G^{\scaleto{(\mathbf X)}{4pt}}_u\setminus H^{\scaleto{(\mathbf X)}{4pt}}_u)\subseteq G^{\scaleto{(\mathbf Y)}{4pt}}_{\Phi(u)}\setminus H^{\scaleto{(\mathbf Y)}{4pt}}_{\Phi(u)}$,
%\hfill if $\Phi(v)\in\kappa_I$ and $x\notin H_v$ then $\Phi(x)\notin H_{\Phi(v)}$,
\item\label{KATEGOR_szomszed}
$\Phi$ respects the neighborhood operations\footnote{Since it is clear from the context, and to make the notation less convoluted, the superscripts $\iksz$ and $\ipsz$ will not be indicated in the neighborhood operations. For example, here ${ }_{\downarrow_u}$ denotes the neighborhood operation in $\textbf{\textit{G$^{\scaleto{(\mathbf X)}{3pt}}_u$}}$
and 
${ }_{\downarrow_{\Phi(u)}}$ denotes the neighborhood operation in $\textbf{\textit{G$^{\scaleto{(\mathbf Y)}{3pt}}_{\Phi(u)}$}}$.}
\begin{enumerate}
\item\label{S6a} 
if $u\in\kappa_J^{\scaleto{(\mathbf X)}{4pt}}$ and 
$\Phi(u)\in\kappa_J^{\scaleto{(\mathbf Y)}{4pt}}$
then for $x\in G^{\scaleto{(\mathbf X)}{4pt}}_u$,
$\Phi(x_{\downarrow_u})={\Phi(x)}_{\downarrow_{\Phi(u)}}$,

\item\label{S6b}
if $u\in\kappa_J^{\scaleto{(\mathbf X)}{4pt}}$ and 
$\Phi(u)\in\kappa_o^{\scaleto{(\mathbf Y)}{4pt}}$
then for $x\in G^{\scaleto{(\mathbf X)}{4pt}}_u$,
$\Phi(x_{\downarrow_u})=\Phi(x)$,
\end{enumerate}

\item\label{KATEGOR_injective}
$\Phi$ is partially injective:
for $u,v\in{\kappa^{\scaleto{(\mathbf X)}{4pt}}}$,
$x\in G^{\scaleto{(\mathbf X)}{4pt}}_u$,
$y\in G^{\scaleto{(\mathbf X)}{4pt}}_v$,
%\begin{enumerate}
%\item \label{KATEGOR_inj_1}
%if
%$\Phi(u) <_{\kappa^{\scaleto{(\mathbf Y)}{3pt}}}  \Phi(v)$
%and
%$\varsigma_{u\to v}^{\scaleto{(\mathbf X)}{4pt}}(x)\succ^{\scaleto{(\mathbf X)}{4pt}}_v y$
%then
%$
%\Phi_v(\varsigma_{u\to v}^{\scaleto{(\mathbf X)}{4pt}}(x)) 
%\succ^{\scaleto{(\mathbf Y)}{4pt}}_{\Phi(v)}
%\Phi_v(y)
%$,
if
$\Phi(uv)\in \kappa_I^{\scaleto{(\mathbf Y)}{4pt}}$
and
$
\varsigma_{v\to uv}^{\scaleto{(\mathbf X)}{4pt}}(y)
\prec^{\scaleto{(\mathbf X)}{4pt}}_{uv}
\varsigma_{u\to uv}^{\scaleto{(\mathbf X)}{4pt}}(x)
\in
H^{\scaleto{(\mathbf X)}{4pt}}_{uv}
$
then
$
\Phi_{uv}(\varsigma_{v\to uv}^{\scaleto{(\mathbf X)}{4pt}}(y)) 
\prec^{\scaleto{(\mathbf Y)}{4pt}}_{\Phi(uv)}
\Phi_{uv}(\varsigma_{u\to uv}^{\scaleto{(\mathbf X)}{4pt}}(x)) 
$.

%\end{enumerate}
%$\varsigma_{\Phi(u)\to\Phi(v)}^{\scaleto{(\mathbf Y)}{4pt}}(\Phi_u(x))
%\succ^{\scaleto{(\mathbf X)}{4pt}}_{\Phi(v)}
%\Phi_v(y)
%$.
\end{enumerate}
Not only the $\Phi_u$'s can be defined from $\Phi$ via (\ref{PhiSlice}) but also $\Phi$ can be recovered from the $\Phi_u$'s via (\ref{PHIdomain}) and (\ref{KATEGOR_}) in the obvious manner. 
Therefore, alternatively, one may think of $\Phi$ as a system of $o$-group homomorphisms 
$
\{\Phi_u : u\in\kappa^{\scaleto{(\mathbf X)}{4pt}}\}
$ 
with properties \ref{KATEGOR_embedding}, \ref{KATEGOR_commutes}--\ref{KATEGOR_injective}.
Bunch homomorphisms are homomorphisms, in the sense of Definition~\ref{homodirsyst}, of the underlying direct systems of the bunches
which satisfy the additional conditions \ref{KATEGOR_legkisebbelem}--\ref{KATEGOR_injective}.
\end{definition}

\begin{lemma}\label{KATEGOR_HogyanLatszik}
Let 
$$
\begin{array}{l}
\mathbf X=(X,\leq^{\scaleto{(\mathbf X)}{4pt}},\teiksz,\ite{}^{\scaleto{(\mathbf X)}{4pt}},t^{\scaleto{(\mathbf X)}{4pt}},f^{\scaleto{(\mathbf X)}{4pt}}),
\\
\mathbf Y=(Y,\leq^{\scaleto{(\mathbf Y)}{4pt}},\teipsz,\ite{}^{\scaleto{(\mathbf Y)}{4pt}},t^{\scaleto{(\mathbf Y)}{4pt}},f^{\scaleto{(\mathbf Y)}{4pt}})
\end{array}
$$ 
be involutive FL$_e$-chains, 
%both odd or both even,
along with their respective bunches of layer groups
\begin{equation}\label{EQrend875skjdhJG}
\begin{array}{l}
{\pazocal X}=\langle 
\textbf{\textit{G$^{\scaleto{(\mathbf X)}{4pt}}_u$}},
\textbf{\textit{H$^{\scaleto{(\mathbf X)}{4pt}}_u$}}, 
\varsigma_{u\to v}^{\scaleto{(\mathbf X)}{4pt}}  \rangle_{\langle \kappa_o^{\scaleto{(\mathbf X)}{4pt}}, \kappa_J^{\scaleto{(\mathbf X)}{4pt}}, \kappa_I^{\scaleto{(\mathbf X)}{4pt}},\leq_{\kappa^{\scaleto{(\mathbf X)}{3pt}}}\rangle}
\mbox{ with $\textbf{\textit{G$^{\scaleto{(\mathbf X)}{4pt}}_u$}}=(G^{\scaleto{(\mathbf X)}{4pt}}_u,\preceq^{\scaleto{(\mathbf X)}{4pt}}_u,\cdot^{\scaleto{(\mathbf X)}{4pt}}_u,\ { }^{-1^{\scaleto{(\mathbf X)}{3pt}}_u},u),%\ \ \ (u\in\kappa^{\scaleto{(\mathbf X)}{4pt}}),
$}
\\
{\pazocal Y}=\langle \textbf{\textit{G$^{\scaleto{(\mathbf Y)}{4pt}}_u$}},\textbf{\textit{H$^{\scaleto{(\mathbf Y)}{4pt}}_u$}}, \varsigma_{u\to v}^{\scaleto{(\mathbf Y)}{4pt}}  \rangle_{\langle \kappa_o^{\scaleto{(\mathbf Y)}{4pt}}, \kappa_J^{\scaleto{(\mathbf Y)}{4pt}}, \kappa_I^{\scaleto{(\mathbf Y)}{4pt}},\leq_{\kappa^{\scaleto{(\mathbf Y)}{3pt}}}\rangle}
\mbox{ with $ \textbf{\textit{G$^{\scaleto{(\mathbf Y)}{4pt}}_u$}}=(G^{\scaleto{(\mathbf Y)}{4pt}}_u,\preceq^{\scaleto{(\mathbf Y)}{4pt}}_u,\cdot^{\scaleto{(\mathbf Y)}{4pt}}_u,\ { }^{-1^{\scaleto{(\mathbf Y)}{3pt}}_u},u).%\ \ \ (u\in\kappa^{\scaleto{(\mathbf Y)}{4pt}}).
$}
\end{array}
\end{equation}
%Adapt the notation of Lemma~\ref{KATEGOR_KiterjeDezik}. 
\begin{enumerate}
\item
If $\varphi:X\to Y$ is a homomorphism 
from $\mathbf X$ to $\mathbf Y$
then 
\begin{equation}\label{KATEGOR_EQmegszoRITOm}
\Phi:=\varphi|_{G^{\scaleto{(\mathbf X)}{3pt}}}
\end{equation}
%the restriction of $\varphi$ to $G^{\scaleto{(\mathbf X)}{4pt}}$ 
is a bunch homomorphism from $\pazocal X$ to $\pazocal Y$.

\item 
If $\Phi%=\{\Phi_u : u\in\kappa^{\scaleto{(\mathbf X)}{4pt}}\}
$ is a bunch homomorphism from $\pazocal X$ to $\pazocal Y$ then 
$\Phi$ extends to a homomorphims $\varphi:X\to Y$ 
from $\mathbf X$ to $\mathbf Y$
via
%$\varphi:X\to Y$ given by
%\begin{equation}\label{KATEGOR_ViSSZaEpul}
%\varphi(x)
%=
%\left\{
%\begin{array}{ll}
%\Phi(x)(=\Phi_u(x)) & \mbox{ if $x\in G_u$},\\
%\Phi(z)(=\Phi_u(z)) & \mbox{ if $\accentset{\bullet}z=x\in \accentset{\bullet}H_u$ and $\Phi(u)\in\kappa_o^{\scaleto{(\mathbf Y)}{4pt}}$},\\
%\accentset{\bullet}{\Phi(z)}(=\accentset{\bullet}{\Phi_u(z)}) & \mbox{ if $\accentset{\bullet}z=x\in \accentset{\bullet}H_u$ and $\Phi(u)\in\kappa_I^{\scaleto{(\mathbf Y)}{4pt}}$}.\\
%\end{array}
%\right. 
%\footnote{At $\accentset{\bullet}{\Phi_u(z)}$ the $\accentset{\bullet}{.}$ operation is computed in the respective layer group (in \textbf{\textit{G$^{\scaleto{(\mathbf Y)}{4pt}}_{\Phi(u)}$}}).}
%\end{equation}
\begin{equation}\label{KATEGOR_ViSSZaEpulEredeti}
\varphi(x)
=
\left\{
\begin{array}{ll}
\Phi(x)
& \mbox{ if $x\in G^{\scaleto{(\mathbf X)}{4pt}}_u$},
\\
\Phi(a) 
& \mbox{ if $\accentset{\bullet}a=x\in \accentset{\bullet}H_u$ and $\Phi(u)\in\kappa_o^{\scaleto{(\mathbf Y)}{4pt}}$},\\
\accentset{\bullet}{\Phi(a)}
 & \mbox{ if $\accentset{\bullet}a=x\in \accentset{\bullet}H_u$ and $\Phi(u)\in\kappa_I^{\scaleto{(\mathbf Y)}{4pt}}$}.\\
\end{array}
\right. 
\footnote{In the expression $\accentset{\bullet}{\Phi_u(z)}$ the $\accentset{\bullet}{.}$ operation is computed in the respective layer group (in \textbf{\textit{G$^{\scaleto{(\mathbf Y)}{4pt}}_{\Phi(u)}$}}).}
\end{equation}

\item 
The previous two items describe a bijective correspondence\footnote{Here we tacitly identify, for any $x\in X$, $\accentset{\bullet}x$ given in (\ref{KATEGOR_XHiGYkESZUL}) by $\accentset{\bullet}x$ in (\ref{KATEGOR_Hukeszul}).} between homomorphisms of odd or even involutive FL$_e$-chains and bunch homomorphisms of their corresponding bunch representations.

\end{enumerate}
\end{lemma}
\begin{proof}
We adopt the notation of Theorem~\ref{KATEGOR_BUNCH_X} with the appropriate superscript $^{\scaleto{(\mathbf X)}{4pt}}$ or $^{\scaleto{(\mathbf Y)}{4pt}}$.
A mapping $\varphi$ is a homomorphism 
from $\mathbf X$ to $\mathbf Y$
if and only if the following conditions hold.
For $x,y\in X$,
\begin{enumerate}[(B1)]
\item\label{KATEGOR_rESzhaLMaZ} $\varphi(X)\subseteq Y$, 
\item\label{KATEGOR_OrderINGresZe} $x\leq^{\scaleto{(\mathbf X)}{4pt}} y$ implies 
$\varphi(x)\leq^{\scaleto{(\mathbf Y)}{4pt}}\varphi(y)$, 
\item\label{KATEGOR_ReSTrict}
$\varphi(\giksz{x}{y})=\gipsz{\varphi(x)}{\varphi(y)}$ 
\item\label{KATEGOR_ReSTrictITE} 
$\varphi(\resiksz{x}{y})=\resipsz{\varphi(x)}{\varphi(y)}$ 
\item\label{KATEGOR_TkegyenlOEk} 
$\varphi(t^{\scaleto{(\mathbf X)}{4pt}})=t^{\scaleto{(\mathbf Y)}{4pt}}$,
\item\label{KATEGOR_FkegyenlOEk} 
$\varphi(f^{\scaleto{(\mathbf X)}{4pt}})=f^{\scaleto{(\mathbf Y)}{4pt}}$.
\end{enumerate}

\bigskip
\noindent
(1) Assume that $\varphi:X\to Y$ is a homomorphism.
For $x\in G^{\scaleto{(\mathbf X)}{4pt}}_u$,
\begin{equation}\label{EQmindenPHI}
\varphi(x)\overset{(\ref{KATEGOR_EQmegszoRITOm})}{=}
\Phi(x)
\overset{(\ref{PhiSlice})}{=}
\Phi_u(x)
.
\end{equation}

\bigskip
\noindent
\ref{KATEGOR_particio}:
We start with two observations which we shall use without mention in the sequel.
\begin{itemize}
\item
Note that 
$u\in\kappa^{\scaleto{(\mathbf X)}{4pt}}_J$ if and only if $\negaiksz{u}$ is not idempotent, 
whereas
$u\in\kappa^{\scaleto{(\mathbf X)}{4pt}}_o\cup\kappa^{\scaleto{(\mathbf X)}{4pt}}_I$ if and only if $\negaiksz{u}$ is idempotent, 
see Table~\ref{KATEGOR_ThetaPsiOmegaAGAIN} and (\ref{KATEGOR_kappaIJ}).
\item
Since for $u \in\kappa^{\scaleto{(\mathbf X)}{4pt}}$, $u\overset{(\ref{EQrend875skjdhJG})}{\in}G^{\scaleto{(\mathbf X)}{4pt}}_u$, $\Phi(u)=\varphi(u)$ follows from (\ref{EQmindenPHI}).
\end{itemize}

\begin{description}
\item
If $u\in\kappa_o^{\scaleto{(\mathbf X)}{4pt}}$ then by Table~\ref{KATEGOR_ThetaPsiOmegaAGAIN}, $u=t^{\scaleto{(\mathbf X)}{4pt}}=f^{\scaleto{(\mathbf X)}{4pt}}$, hence 
$\varphi(u)=\varphi(t^{\scaleto{(\mathbf X)}{4pt}})=\varphi(f^{\scaleto{(\mathbf X)}{4pt}})$ follows, yielding $\varphi(u)=t^{\scaleto{(\mathbf Y)}{4pt}}=f^{\scaleto{(\mathbf Y)}{4pt}}$ by \ref{KATEGOR_TkegyenlOEk} and \ref{KATEGOR_FkegyenlOEk}, and in turn,
$\varphi(u)\in\kappa_o^{\scaleto{(\mathbf Y)}{4pt}}$ by Table~\ref{KATEGOR_ThetaPsiOmegaAGAIN}.

\item 
%$\varphi(\kappa_I^{\scaleto{(\mathbf X)}{4pt}})\subseteq\kappa_I^{\scaleto{(\mathbf Y)}{4pt}}\cup\kappa_o^{\scaleto{(\mathbf Y)}{4pt}}$:
If $u\in\kappa_I^{\scaleto{(\mathbf X)}{4pt}}$
then %either $u\in\bar\kappa_I^{\scaleto{(\mathbf X)}{4pt}}$ or $u=t$ and $f$ is idempotent. In both cases 
%by the third row of Table~\ref{KATEGOR_ThetaPsiOmegaAGAIN} and (\ref{KATEGOR_kappaIJ}),
$
\negaiksz{u}
$ 
 is idempotent,
hence $\varphi(\negaiksz{u})=\negaipsz{\varphi(u)}$ 
 is idempotent, too, 
yielding $\varphi(u)\in\kappa_I^{\scaleto{(\mathbf Y)}{4pt}}\cup\kappa_o^{\scaleto{(\mathbf Y)}{4pt}}$.

\end{description}

%Throughout this item we shall use Table~\ref{KATEGOR_ThetaPsiOmegaAGAIN} without explicitly referring to it. 
%Let $u\in\kappa^{\scaleto{(\mathbf X)}{4pt}}$.
%If $\varphi$ maps both constants of $\mathbf X$ to $t^{\scaleto{(\mathbf Y)}{4pt}}$ then $\mathbf Y$ is odd, and hence $t^{\scaleto{(\mathbf Y)}}\in \kappa_o^{\scaleto{(\mathbf Y)}{4pt}}$.
\begin{description}
\item 
%$\varphi(\kappa_J^{\scaleto{(\mathbf X)}{4pt}})\subseteq\kappa_J^{\scaleto{(\mathbf Y)}{4pt}}\cup\kappa_o^{\scaleto{(\mathbf Y)}{4pt}}$:
Let $u\in\kappa_J^{\scaleto{(\mathbf X)}{4pt}}$.
%Then, $\negaiksz{u}$ is not idempotent.
If
$%\varphi(\negaiksz{u})=
\negaipsz{\varphi(u)}
$ 
is not idempotent then 
%either $\varphi(u)=t^{\scaleto{(\mathbf Y)}{4pt}}$ %and $\mathbf Y$ is even, in which case $\varphi(u)\in\kappa_o^{\scaleto{(\mathbf Y)}{4pt}}$,
$\varphi(u)$ is in $\kappa_J^{\scaleto{(\mathbf Y)}{4pt}}$.
Assume $%\varphi(\negaiksz{u})=
\negaipsz{\varphi(u)}$ is idempotent. %(this is necessarily the case if $u\in\kappa_o^{\scaleto{(\mathbf X)}{4pt}}$).
%We shall prove that $\varphi(u)\in\kappa_o^{\scaleto{(\mathbf Y)}{4pt}}$.
%Since $u\in\kappa_J^{\scaleto{(\mathbf X)}{4pt}}$, 
By %the first rows of 
(\ref{KATEGOR_XHiGYkESZUL})-(\ref{KATEGOR_DEFcsopi}), 
$G^{\scaleto{(\mathbf X)}{4pt}}_u=\{x\in X : \resiksz{x}{x}=u\}$.
Therefore,
 $
 u
\overset{u\in G^{\scaleto{(\mathbf X)}{4pt}}_u}{=}
\resiksz{u}{u}
=
\resiksz{\negaiksz{u}}{\negaiksz{u}}
$
holds and thus 
$\negaiksz{u}\in G^{\scaleto{(\mathbf X)}{4pt}}_u$.
By (\ref{KATEGOR_EzLeSzainVerZ}), 
${\left(\negaiksz{u}\right)}^{-1^{\scaleto{(\mathbf X)}{3pt}}_u}=\resiksz{\negaiksz{u}}{u}=\negaiksz{\left(\giksz{\negaiksz{u}}{\negaiksz{u}}\right)}$ follows,
hence
$
u
=
\gikszu{\negaiksz{u}}{{(\negaiksz{u})}^{-1^{\scaleto{(\mathbf X)}{3pt}}_u}}
\overset{(\ref{KATEGOR_IgyTorzulaSzorzat})}{=}
\giksz{\negaiksz{u}}{{(\negaiksz{u})}^{-1^{\scaleto{(\mathbf X)}{3pt}}_u}}
=
\giksz{\negaiksz{u}}{\negaiksz{\left(\giksz{\negaiksz{u}}{\negaiksz{u}}\right)}}
$.
Therefore, 
$
\varphi(u)
 =  
\varphi\left(\giksz{\negaiksz{u}}{\negaiksz{\left(\giksz{\negaiksz{u}}{\negaiksz{u}}\right)}}\right)
=
\gipsz{\negaipsz{\varphi(u)}}{\negaipsz{\left(\gipsz{\negaipsz{\varphi(u)}}{\negaipsz{\varphi(u)}}\right)}}
%  =\gipsz{\negaipsz{\varphi(u)}}{\negaipsz{\left(\negaipsz{\varphi(u)}\right)}}
  = 
\gipsz{\negaipsz{\varphi(u)}}{\varphi(u)}
=
\varphi(\giksz{\negaiksz{u}}{u})
\overset{(\ref{KATEGOR_IgyNeznekKi})}{=}
\varphi(\negaiksz{u})
=
\negaipsz{\varphi(u)}$.
Since
$\varphi(u)=\negaipsz{\varphi(u)}$ and 
$\varphi(u)$ is positive,
%Since $\varphi(u)$ is a positive idempotent, $\negaipsz{\varphi(u)}$ is negative. 
$\varphi(u)=t^{\scaleto{(\mathbf Y)}{4pt}}=
\negaipsz{(t^{\scaleto{(\mathbf Y)}{4pt}})}%=f^{\scaleto{(\mathbf Y)}{4pt}}
$ follows, and hence by Table~\ref{KATEGOR_ThetaPsiOmegaAGAIN},
$\varphi(u)\in\kappa_o^{\scaleto{(\mathbf Y)}{4pt}}$.
\end{description}

\bigskip
\noindent
\ref{KATEGOR_embedding}:
For $u \in\kappa^{\scaleto{(\mathbf X)}{4pt}}$,
$u$ is a positive idempotent element in $X$ by (\ref{KATEGOR_IGYleszSKELETON}),
hence $\varphi(u)$ is a positive idempotent element in $Y$ (see \ref{KATEGOR_TkegyenlOEk} and \ref{KATEGOR_OrderINGresZe} for positivity and \ref{KATEGOR_ReSTrict} for idempotence), and thus $\Phi(u)\in\kappa^{\scaleto{(\mathbf Y)}{4pt}}$ by (\ref{KATEGOR_IGYleszSKELETON}).
Let $u,v\in{\kappa^{\scaleto{(\mathbf X)}{4pt}}}$, 
$u\leq_{\kappa^{\scaleto{(\mathbf X)}{3pt}}} v$.
Then $u\leq^{\scaleto{(\mathbf X)}{4pt}} v$ by (\ref{KATEGOR_IGYleszSKELETON}).
It holds true that, $u\overset{(\ref{KATEGOR_IgyNeznekKi})}{\in} G^{\scaleto{(\mathbf X)}{4pt}}_u%\overset{(\ref{PHIdomain})}{\subseteq} G^{\scaleto{(\mathbf X)}{4pt}}
$ 
and likewise,
$v\in G^{\scaleto{(\mathbf X)}{4pt}}$, hence
%\begin{equation}\label{phiVSPhi}\mbox{
$
\Phi(u)
\overset{(\ref{EQmindenPHI})}{=}
\varphi(u)
$ 
%\ \ 
and 
%\ \ 
$
\Phi(v)
\overset{(\ref{EQmindenPHI})}{=}
\varphi(v)
$.
%}\end{equation}
Since $u\leq^{\scaleto{(\mathbf X)}{4pt}} v$, %and $\varphi$ is a homomorphism, 
$
\Phi(u)
=%\overset{(\ref{phiVSPhi})}{=}
\varphi(u)
\overset{\ref{KATEGOR_OrderINGresZe}}{\leq^{\scaleto{(\mathbf Y)}{4pt}}}
\varphi(v)
=%\overset{(\ref{phiVSPhi})}{=}
\Phi(v)
$
follows
and 
since
$\Phi(u),\Phi(v)\in\kappa^{\scaleto{(\mathbf Y)}{4pt}}$,
$
\Phi(u)
\leq_{\kappa^{\scaleto{(\mathbf Y)}{3pt}}}
\Phi(v)
$
follows
by (\ref{KATEGOR_IGYleszSKELETON}).

\bigskip
\noindent
\ref{KATEGOR_homomorfizmus}: 
Let  $u\in\kappa^{\scaleto{(\mathbf X)}{3pt}}$. 
Until the end of the proof of \ref{KATEGOR_szomszed}, we denote, for short,
\begin{equation}\label{EzaV}
w
:=
\varphi(u)
\overset{(\ref{EQmindenPHI})}{=}
\Phi(u)
\overset{(\ref{EQmindenPHI})}{=}
\Phi_u(u)
.
\end{equation}
First we prove that $w$ is the unit element of $\textbf{\textit{G$^{\scaleto{(\mathbf Y)}{4pt}}_w$}}$.
By (\ref{KATEGOR_IGYleszSKELETON}), $u$ is a positive idempotent element in $X$, 
and $u$ is the  unit element  of $\textbf{\textit{G$^{\scaleto{(\mathbf X)}{4pt}}_u$}}$, c.f.\ %(\ref{KATEGOR_IgyNeznekKi}) or 
(\ref{EQrend875skjdhJG}). 
%This means that $x\in G_u$ for some positive idempotent element $u$ of $\mathbf X$.
Since positive idempotent elements are mapped into positive idempotent elements by  homomorphisms, and since for every layer group there corresponds a single positive idempotent element which is its unit element and which indexes that particular layer group, it follows that $\Phi_u$, which coincides with $\varphi$ over $G^{\scaleto{(\mathbf X)}{4pt}}_u$
by (\ref{EQmindenPHI}), preserves the unit element.
\\
%First we prove that $\Phi_u(x)$ is in the $v$-layer group \textbf{\textit{G$^{\scaleto{(\mathbf Y)}{4pt}}_v$}} of $\mathbf Y$ if $x$ is in the $u$-layer group \textbf{\textit{G$^{\scaleto{(\mathbf X)}{4pt}}_u$}} of $\mathbf X$.
Next we prove that $\Phi_u$ maps $G^{\scaleto{(\mathbf X)}{4pt}}_u$ into $G^{\scaleto{(\mathbf Y)}{4pt}}_w$. 
%\begin{enumerate}[i)]
%\item 
%$\tau_Y(\Phi_u(x))=
Let $x\in G^{\scaleto{(\mathbf X)}{4pt}}_u$.
Then 
$
\resipsz{\Phi_u(x)}{\Phi_u(x)}
%\overset{(\ref{PhiSlice})}{=}
%\resiksz{\Phi(x)}{\Phi(x)}
\overset{(\ref{EQmindenPHI})}{=}
\resipsz{\varphi(x)}{\varphi(x)}
\overset{\ref{KATEGOR_ReSTrictITE}}{=}
\varphi(\resiksz{x}{x})
=
\varphi(u)
\overset{(\ref{EzaV})}{=}
w
$, %=\varphi(\resiksz{y}{y})=\resiksz{\varphi(y)}{\varphi(y)}=\resiksz{\Phi_u(y)}{\Phi_u(y)}$. %=\tau_Y(\Phi_u(y))$.
hence $\Phi_u(x)\overset{(\ref{KATEGOR_XHiGYkESZUL})}{\in}
L^{\scaleto{(\mathbf Y)}{4pt}}_w$. 
%\item 
If $w\notin\kappa_I^{\scaleto{(\mathbf Y)}{4pt}}$ then 
$%\Phi_u(x)\in 
G^{\scaleto{(\mathbf Y)}{4pt}}_w
\overset{(\ref{KATEGOR_DEFcsopi})}{=}
L^{\scaleto{(\mathbf Y)}{4pt}}_w$ and we are done.
Assume $w\in\kappa_I^{\scaleto{(\mathbf Y)}{4pt}}$.
Then 
$u\overset{\ref{KATEGOR_particio}}{\in}\kappa_I^{\scaleto{(\mathbf X)}{4pt}}$
%and 
%{\color{blue} hence $G^{\scaleto{(\mathbf X)}{4pt}}_u\overset{(\ref{KATEGOR_DEFcsopi})}{=}L^{\scaleto{(\mathbf X)}{4pt}}_u\setminus\accentset{\bullet}H^{\scaleto{(\mathbf X)}{4pt}}_u$.}
%by Table~\ref{KATEGOR_ThetaPsiOmegaAGAIN},
%$v>t^{\scaleto{(\mathbf Y)}{4pt}}$, and $\nega{v}<v$ since $\nega{v}<\negat^{\scaleto{(\mathbf Y)}}=f^{\scaleto{(\mathbf Y)}}\leqt^{\scaleto{(\mathbf Y)}}<v$, 
%$L^{\scaleto{(\mathbf Y)}{4pt}}_v$ is even with an 
%\begin{equation}\label{idVlesz}
%\mbox{
%idempotent falsum constant $\negaipsz{v}$.
%}
%\end{equation}
and
we need to prove that $\Phi_u(x)$ is not in ${\accentset{\bullet}H_w^{\scaleto{(\mathbf Y)}{4pt}}}$, c.f.\ (\ref{KATEGOR_DEFcsopi}).
Contrary, assume  
%\begin{equation}\label{IndasS}
$
\Phi_u(x)\in{\accentset{\bullet}H_{w}^{\scaleto{(\mathbf Y)}{4pt}}}
.
$
%\end{equation}
%that is, by (\ref{KATEGOR_XHiGYkESZUL})
%\begin{equation}\label{EQlkkjh87jh}
%\Phi_u(x)
%=
%\gipsz{z}{\negaipsz{v}}
%<^{\scaleto{(\mathbf Y)}{4pt}}
%z
%\end{equation}
%holds for some $z\in H^{\scaleto{(\mathbf Y)}{4pt}}_{v}$. 
%($\resiksz{z}{z}=\varphi(u)$)
%with $\gipsz{z}{\nega{v}}<z$.
By (\ref{KATEGOR_IgyNeznekKi}), 
$
H^{\scaleto{(\mathbf Y)}{4pt}}_w
\subseteq
G^{\scaleto{(\mathbf Y)}{4pt}}_w
$
and by 
(\ref{KATEGOR_DEFcsopi}),
$G^{\scaleto{(\mathbf Y)}{4pt}}_w\cap \accentset{\bullet}H_{w}^{\scaleto{(\mathbf Y)}{4pt}}=\emptyset$, hence 
$H^{\scaleto{(\mathbf Y)}{4pt}}_w\cap \accentset{\bullet}H_{w}^{\scaleto{(\mathbf Y)}{4pt}}=\emptyset$ follows, and thus
$\Phi_u(x)\overset{(\ref{EQmindenPHI})}{=}
\varphi(x)$ is not $w$-invertible (with respect to $\teipsz$) by (\ref{KATEGOR_XHiGYkESZUL}).
Since $\varphi$ maps $u$-invertible elements into $w$-invertible ones by \ref{KATEGOR_ReSTrict} and (\ref{EzaV}), 
$x$ cannot be $u$-invertible, hence $x\notin H^{\scaleto{(\mathbf X)}{4pt}}_u$ by (\ref{KATEGOR_XHiGYkESZUL}). 
%If $x\in H^{\scaleto{(\mathbf X)}{4pt}}_u$ then, by (\ref{KATEGOR_XHiGYkESZUL}), $x$ is $u$-invertible (with respect to $\teiksz$), and hence $\Phi_u(x)$ would be $v$-invertible (with respect to $\teipsz$), which is not the case.
Therefore,  
$x\in G^{\scaleto{(\mathbf X)}{4pt}}_u\setminus H^{\scaleto{(\mathbf X)}{4pt}}_u%\overset{(\ref{KATEGOR_DEFcsopi})}{=}
$. 
%Since $G^{\scaleto{(\mathbf X)}{4pt}}_u$ and $\accentset{\bullet}H^{\scaleto{(\mathbf X)}{4pt}}_u$ are disjoint, see (\ref{KATEGOR_Hukeszul}), 
Since $u\in\kappa_I^{\scaleto{(\mathbf X)}{4pt}}$,
%by (\ref{KATEGOR_DEFcsopi}) it holds true that 
%\begin{equation}\label{KATEGOR_cupPONT}
%$
%L^{\scaleto{(\mathbf X)}{4pt}}_u
%=
%(G^{\scaleto{(\mathbf X)}{4pt}}_u\setminus H^{\scaleto{(\mathbf X)}{4pt}}_u)
%\overset{\cdot}{\cup}
%H^{\scaleto{(\mathbf X)}{4pt}}_u
%\overset{\cdot}{\cup}
% \accentset{\bullet}H^{\scaleto{(\mathbf X)}{4pt}}_u
%,%\footnote{See \cite[(6.15)]{JenRepr2020}}
%$
%\end{equation}
%and 
by the first two rows of (\ref{KATEGOR_SplitNega}),
%\begin{equation}\label{KATEGOR_NegaEQviselkedik}
%a\in H^{\scaleto{(\mathbf X)}{4pt}}_u \mbox{ if and only if } \negaiksz{a}\in\accentset{\bullet}H^{\scaleto{(\mathbf X)}{4pt}}_u.
%\footnote{See \cite[(6.9) and (6.6) in Definition~6.2]{JenRepr2020}}
%\end{equation}
$\negaiksz{}$maps $H^{\scaleto{(\mathbf X)}{4pt}}_u$ into $\accentset{\bullet}H^{\scaleto{(\mathbf X)}{4pt}}_u$ and vice versa,
%\begin{equation}\label{KATEGOR_NegaEQviselkedik}
%\mbox{$\negaiksz{}$ maps $H^{\scaleto{(\mathbf X)}{4pt}}_u$ into $\accentset{\bullet}H^{\scaleto{(\mathbf X)}{4pt}}_u$ and vice versa,}
%\end{equation}
%By (\ref{KATEGOR_cupPONT}) and (\ref{KATEGOR_NegaEQviselkedik}),
hence
$
\negaiksz{x}
\in
G^{\scaleto{(\mathbf X)}{4pt}}_u
\setminus
H^{\scaleto{(\mathbf X)}{4pt}}_u
$
and
$
{H_{w}^{\scaleto{(\mathbf Y)}{4pt}}}
\overset{\Phi_u(x)\in{\accentset{\bullet}H_{w}^{\scaleto{(\mathbf Y)}{4pt}}}}{\ni}
\negaipsz{\Phi_u(x)}
%$.
%from $x\in G^{\scaleto{(\mathbf X)}{4pt}}_u\setminus H^{\scaleto{(\mathbf X)}{4pt}}_u$.
%Since $\negaiksz{x}\notin H^{\scaleto{(\mathbf X)}{4pt}}_u$ it follows by (\ref{KATEGOR_XHiGYkESZUL}) that 
%\begin{equation}\label{EQlskdfvgdlkj42}
%\giksz{\negaiksz{x}}{\negaiksz{u}}
%=
%\overset{(\ref{KATEGOR_XHiGYkESZUL})}{=}
%\negaiksz{x}
%.
%\end{equation}
%It holds true that
%$
%{H_{w}^{\scaleto{(\mathbf Y)}{4pt}}}
%\ni
%\negaipsz{\Phi_u(x)}
\overset{(\ref{EQmindenPHI})}{=}
\negaipsz{\varphi(x)}
\overset{\ref{KATEGOR_ReSTrictITE},\ref{KATEGOR_FkegyenlOEk}}{=}
\varphi\left(\negaiksz{x}\right)
\overset{(\ref{EQmindenPHI})}{=}
\Phi_u\left(\negaiksz{x}\right)
$
which,
since $w$ is the unit element of $\textbf{\textit{G$^{\scaleto{(\mathbf Y)}{4pt}}_w$}}$,
leads to
$
\Phi_u\left(\negaiksz{x}\right)
\overset{(\ref{KATEGOR_XHiGYkESZUL})}{>}
\gipsz{\Phi_u\left(\negaiksz{x}\right)}{\negaipsz{w}}
\overset{(\ref{EzaV})}{=}
\gipsz{\Phi_u\left(\negaiksz{x}\right)}{\negaipsz{\Phi_u(u)}}
\overset{(\ref{EQmindenPHI})}{=}
\gipsz{\varphi(\negaiksz{x})}{\negaipsz{\varphi(u)}}
\overset{\ref{KATEGOR_ReSTrict}}{=}
\varphi\left(
\giksz{\negaiksz{x}}{\negaiksz{u}}
\right)
\overset{(\ref{EQmindenPHI})}{=}
%\gipsz{\Phi_u\left(\negaiksz{x}\right)}{\Phi_u(\negaiksz{u})}=
\Phi_u\left(
\giksz{\negaiksz{x}}{\negaiksz{u}}
\right)
%\overset{(\ref{EQlskdfvgdlkj42})}{=}
\overset{\negaiksz{x}\notin H^{\scaleto{(\mathbf X)}{4pt}}_u,\,(\ref{KATEGOR_XHiGYkESZUL})}{=}
\Phi_u\left(
\negaiksz{x}
\right)
$,
a contradiction.
\\
Third, we verify that for $u\in\kappa^{\scaleto{(\mathbf X)}{4pt}}$ and $x,y\in G^{\scaleto{(\mathbf X)}{4pt}}_u$
it holds true that
$
\Phi_u(\gdotu{x}{y})
=
\gdiamondw{\Phi_u(x)}{\Phi_u(y)}
$.
%Throughout this item we shall use (\ref{KATEGOR_}) without further reference.
If $u\notin\kappa_I^{\scaleto{(\mathbf X)}{4pt}}$ then 
by \ref{KATEGOR_particio} $w\notin\kappa_I^{\scaleto{(\mathbf Y)}{4pt}}$, and 
$
\Phi_u(\gdotu{x}{y})
\overset{(\ref{EQmindenPHI})}{=}
\varphi(\gdotu{x}{y})
\overset{u\notin\kappa_I^{\scaleto{(\mathbf X)}{4pt}},\ (\ref{KATEGOR_IgyTorzulaSzorzat})}{=}
\varphi(\giksz{x}{y})
\overset{\ref{KATEGOR_ReSTrict}}{=}
\gipsz{\varphi(x)}{\varphi(y)}
\overset{(\ref{EQmindenPHI})}{=}
\gipsz{\Phi_u(x)}{\Phi_u(y)}
\overset{w\notin\kappa_I^{\scaleto{(\mathbf Y)}{4pt}},\ (\ref{KATEGOR_IgyTorzulaSzorzat})}{=}
\gdiamondw{\Phi_u(x)}{\Phi_u(y)}
$.
If $u\in\kappa_I^{\scaleto{(\mathbf X)}{4pt}}$ then 
$
\Phi_u(\gdotu{x}{y})
\overset{(\ref{EQmindenPHI})}{=}
\varphi(\gdotu{x}{y})
\overset{(\ref{KATEGOR_IgyTorzulaSzorzat})}{=}
\varphi(\resiksz{(\resiksz{(\giksz{x}{y})}{u})}{u})
\overset{\ref{KATEGOR_ReSTrictITE}}{=}
\resipsz{(\resipsz{(\gipsz{\varphi(x)}{\varphi(y)})}{w})}{w}
\overset{(\ref{EQmindenPHI})}{=}
\resipsz{(\resipsz{(\gipsz{\Phi_u(x)}{\Phi_u(y)})}{w})}{w}
$.
By \ref{KATEGOR_particio}, $w\in\kappa_I^{\scaleto{(\mathbf Y)}{4pt}}\cup\kappa_o^{\scaleto{(\mathbf Y)}{4pt}}$.
If $w\in\kappa_I^{\scaleto{(\mathbf Y)}{4pt}}$ then 
$
\resipsz{(\resipsz{(\gipsz{\Phi_u(x)}{\Phi_u(y)})}{w})}{w}
\overset{(\ref{KATEGOR_IgyTorzulaSzorzat})}{=}
\gdiamondw{\Phi_u(x)}{\Phi_u(y)}
$,
whereas if $w\in\kappa_o^{\scaleto{(\mathbf Y)}{4pt}}$ then $w=t^{\scaleto{(\mathbf Y)}{4pt}}=f^{\scaleto{(\mathbf Y)}{4pt}}$ holds by Table~\ref{KATEGOR_ThetaPsiOmegaAGAIN},
and hence
$
\resipsz{(\resipsz{(\gipsz{\Phi_u(x)}{\Phi_u(y)})}{w})}{w}
=
\resipsz{(\resipsz{(\gipsz{\Phi_u(x)}{\Phi_u(y)})}f^{\scaleto{(\mathbf Y)}{4pt}})}f^{\scaleto{(\mathbf Y)}{4pt}}
=
\nega{\nega{(\gipsz{\Phi_u(x)}{\Phi_u(y)})}}
=
\gipsz{\Phi_u(x)}{\Phi_u(y)}
\overset{(\ref{KATEGOR_IgyNeznekKi})}{=}
\gdiamondw{\Phi_u(x)}{\Phi_u(y)}
$.
\\
Summing up, we have shown that $\Phi_u$ is a group homomorphism from 
\textbf{\textit{G$^{\scaleto{(\mathbf X)}{4pt}}_u$}}
to
\textbf{\textit{G$^{\scaleto{(\mathbf Y)}{4pt}}_w$}}.
Finally, we verify that $\Phi_u$ is an $o$-group homomorphism.
Let $x,y\in G^{\scaleto{(\mathbf X)}{4pt}}_u$, $x\preceq^{\scaleto{(\mathbf X)}{4pt}}_u y$.
By (\ref{KATEGOR_RendeZesINNOVATIVAN})--(\ref{KATEGOR_KibovitettRendezesITTIS}), 
$x\leq^{\scaleto{(\mathbf X)}{4pt}} y$ and thus by \ref{KATEGOR_OrderINGresZe}, 
$
\Phi_u(x)
\overset{(\ref{EQmindenPHI})}{=}
\varphi(x)\leq^{\scaleto{(\mathbf X)}{4pt}}\varphi(y)\overset{(\ref{EQmindenPHI})}{=}
\Phi_u(y)$ follows.
Since we have already shown that $\Phi_u(x),\Phi_u(y)\in G^{\scaleto{(\mathbf Y)}{4pt}}_w$, by (\ref{KATEGOR_RendeZesINNOVATIVAN})--(\ref{KATEGOR_KibovitettRendezesITTIS})
we obtain 
$\Phi_u(x)\preceq^{\scaleto{(\mathbf Y)}{4pt}}_w\Phi_u(y)$.

\bigskip
\noindent
\ref{KATEGOR_commutes}:
For $x\in G^{\scaleto{(\mathbf X)}{4pt}}_u$,
$
\Phi_v(\varsigma_{u\to v}^{\scaleto{(\mathbf X)}{4pt}}(x))
\overset{(\ref{tipusSIGMA}),\,(\ref{EQmindenPHI})}{=}
\varphi(\varsigma_{u\to v}^{\scaleto{(\mathbf X)}{4pt}}(x))
\overset{(\ref{KATEGOR_MapAzSzorzas})}{=}
\varphi(\giksz{v}{x})
\overset{\ref{KATEGOR_ReSTrict}}{=}
\gipsz{\varphi(v)}{\varphi(x)}
\overset{(\ref{EQmindenPHI})}{=}
\gipsz{\Phi(v)}{\Phi_u(x)}
\overset{(\ref{KATEGOR_MapAzSzorzas}),\,\ref{KATEGOR_embedding}}{=}
\varsigma_{\Phi(u)\to \Phi(v)}^{\scaleto{(\mathbf Y)}{4pt}}(\Phi_u(x))
$.

\bigskip
\noindent
\ref{KATEGOR_legkisebbelem}:
By (\ref{KATEGOR_IGYleszSKELETON}), $t^{\scaleto{(\mathbf X)}{4pt}}$ is the least element of 
$\kappa^{\scaleto{(\mathbf X)}{4pt}}$, and $t^{\scaleto{(\mathbf Y)}{4pt}}$ is the least element of 
$\kappa^{\scaleto{(\mathbf Y)}{4pt}}$. Thus the claim follows from \ref{KATEGOR_TkegyenlOEk}.
%By (\ref{KATEGOR_IGYleszSKELETON})-Table~\ref{KATEGOR_ThetaPsiOmegaAGAIN}, $\kappa_o^{\scaleto{(\mathbf X)}{4pt}}\cup\kappa_J^{\scaleto{(\mathbf X)}{4pt}}\cup\kappa_I^{\scaleto{(\mathbf X)}{4pt}}={\kappa^{\scaleto{(\mathbf X)}{4pt}}}=\{u\geq_X t : u \mbox{ is idempotent}\}$ and $\leq_{\kappa^{\scaleto{(\mathbf X)}{3pt}}}\,=\,\leq^{\scaleto{(\mathbf X)}{4pt}}$ over ${\kappa^{\scaleto{(\mathbf X)}{4pt}}}$. Therefore, $\min_{\leq_{\kappa^{\scaleto{(\mathbf X)}{3pt}}}}(\kappa_o^{\scaleto{(\mathbf X)}{4pt}}\cup\kappa_J^{\scaleto{(\mathbf X)}{4pt}}\cup\kappa_I^{\scaleto{(\mathbf X)}{4pt}})=t^{\scaleto{(\mathbf X)}{4pt}}$. Likewise follows $\min_{\leq_{\kappa^{\scaleto{(\mathbf Y)}{3pt}}}}(\kappa_o^{\scaleto{(\mathbf Y)}{4pt}}\cup\kappa_J^{\scaleto{(\mathbf Y)}{4pt}}\cup\kappa_I^{\scaleto{(\mathbf Y)}{4pt}})=t^{\scaleto{(\mathbf Y)}{4pt}}$. Therefore, condition \ref{KATEGOR_legkisebbelem} follows from \ref{KATEGOR_TkegyenlOEk}. 

\bigskip
\noindent
\ref{KATEGOR_subgroups_and_complements}:
Assume $u\in\kappa_I^{\scaleto{(\mathbf X)}{4pt}}$
and
$w%\Phi(u)
\in\kappa_I^{\scaleto{(\mathbf Y)}{4pt}}$.
If $x\in H^{\scaleto{(\mathbf X)}{4pt}}_u$ then by (\ref{KATEGOR_XHiGYkESZUL}),
$x$ is $u$-invertible (with respect to $\teiksz$), and hence $\Phi_u(x)$ must be $w$-invertible (with respect to $\teipsz$, by \ref{KATEGOR_homomorfizmus} and (\ref{EzaV})) and thus $\Phi_u(x)\in H^{\scaleto{(\mathbf Y)}{4pt}}_w$ by (\ref{KATEGOR_XHiGYkESZUL}), as stated.
Assume
$x\in G^{\scaleto{(\mathbf X)}{4pt}}_u\setminus H^{\scaleto{(\mathbf X)}{4pt}}_u$.
%Since ${ }^{-1_u}$ maps to $G^{\scaleto{(\mathbf X)}{4pt}}_u$, it follows from 
Then 
$\negaiksz{x}\in G^{\scaleto{(\mathbf X)}{4pt}}_u$ holds
by
%the third row of 
(\ref{KATEGOR_SplitNega}).
%and (\ref{KATEGOR_NegaEQviselkedik}) 
%that
%$\negaiksz{x}\in G^{\scaleto{(\mathbf X)}{4pt}}_u\setminus H^{\scaleto{(\mathbf X)}{4pt}}_u$.
Therefore, 
$
\Phi_u(x)
\overset{(\ref{EQmindenPHI})}{=}
\varphi(x)
=
\varphi\left(\negaiksz{\left(\negaiksz{x}\right)}\right)
=
\negaipsz{\varphi\left(\negaiksz{x}\right)}
\overset{(\ref{EQmindenPHI})}{=}
\negaipsz{\Phi_u\left(\negaiksz{x}\right)}
$.
Hence if, contrary to the statement,
$\Phi_u(x)$
were in
$H^{\scaleto{(\mathbf Y)}{4pt}}_w$
then 
$\negaipsz{\Phi_u\left(\negaiksz{x}\right)}$
were in
$H^{\scaleto{(\mathbf Y)}{4pt}}_w$, too,
and
using that 
$\negaipsz{}$ maps $H^{\scaleto{(\mathbf Y)}{4pt}}_w$ to $\accentset{\bullet}H^{\scaleto{(\mathbf Y)}{4pt}}_w$
(see the second row of (\ref{KATEGOR_SplitNega})),
$\Phi_u\left(\negaiksz{x}\right)$ would be in
$\accentset{\bullet}H^{\scaleto{(\mathbf Y)}{4pt}}_w$,
a contradiction to \ref{KATEGOR_homomorfizmus}.

\bigskip
\noindent
\ref{KATEGOR_szomszed}: 
\\
Suppose that $u\in\kappa_J$.
It holds true that
$
%\Phi\left(u_{\downarrow_u}\right)
%\overset{(\ref{KATEGOR_SplitNega})}{=}
\Phi\left(\negaiksz{u}\right)
\overset{(\ref{EQmindenPHI})}{=}
\varphi\left(\negaiksz{u}\right)
\overset{\ref{KATEGOR_ReSTrictITE},\ref{KATEGOR_FkegyenlOEk}}{=}
\negaipsz{\varphi(u)}
\overset{(\ref{EzaV})}{=}
\negaipsz{w}
$.
%Let $u\in\kappa_J^{\scaleto{(\mathbf X)}{4pt}}$ and $v\in{\kappa_J^{\scaleto{(\mathbf Y)}{4pt}}}$. 
Therefore,
$
\Phi(x_{\downarrow_u})
\overset{(\ref{PhiSlice})}{=}
\Phi_u(x_{\downarrow_u})
\overset{(\ref{87JHxdhFldD})}{=}
\Phi_u\left(
x
\cdot^{\scaleto{(\mathbf X)}{4pt}}_u
u_{\downarrow_u}
\right)
\overset{(\ref{KATEGOR_SplitNega})}{=}
\Phi_u\left(
x
\cdot^{\scaleto{(\mathbf X)}{4pt}}_u
\negaiksz{u}
\right)
\overset{\ref{KATEGOR_homomorfizmus}}{=}
\Phi_u(x)
\cdot^{\scaleto{(\mathbf Y)}{4pt}}_v
\Phi_u\left(\negaiksz{u}\right)
=
\Phi_u(x)
\cdot^{\scaleto{(\mathbf Y)}{4pt}}_w
\negaipsz{w}
.
$
\begin{enumerate}[(a)]
\item
It holds true that 
$
\Phi_u(x)
\cdot^{\scaleto{(\mathbf Y)}{4pt}}_w
\negaipsz{w}
\overset{(\ref{KATEGOR_SplitNega})}{=}
\Phi_u(x)
\cdot^{\scaleto{(\mathbf Y)}{4pt}}_w
w_{\downarrow_w}
\overset{(\ref{87JHxdhFldD})}{=}
\Phi_u(x)_{\downarrow_w}
$
if $w\in{\kappa_J^{\scaleto{(\mathbf Y)}{4pt}}}$,
and hence
$
\Phi(x_{\downarrow_u})
=
\Phi_u(x)
\cdot^{\scaleto{(\mathbf Y)}{4pt}}_w
\negaipsz{w}
=
{\Phi_u(x)}_{\downarrow_w}
\overset{(\ref{PhiSlice})}{=}
{\Phi(x)}_{\downarrow_w}
$.

\item
If
$w\in\kappa_o^{\scaleto{(\mathbf Y)}{4pt}}$ then 
$\Phi(\mathbf X)$ is odd and $w$ is its unit element, see Table~\ref{KATEGOR_ThetaPsiOmegaAGAIN}, and hence
$\negaipsz{w}=w$ holds.
Therefore,
$
\Phi_u(x)
\cdot^{\scaleto{(\mathbf Y)}{4pt}}_w
\negaipsz{w}
=
\Phi_u(x)
\cdot^{\scaleto{(\mathbf Y)}{4pt}}_w
w
=
\Phi_u(x)
\overset{(\ref{PhiSlice})}{=}
\Phi(x)
$.
\end{enumerate}

\bigskip
\noindent
\ref{KATEGOR_injective}:
By adopting the conditions of \ref{KATEGOR_injective},
by
(\ref{KATEGOR_KibovitettRendezesITTIS}),  
$
\varsigma_{u\to uv}^{\scaleto{(\mathbf X)}{4pt}}(y)
\prec^{\scaleto{(\mathbf X)}{4pt}}_{uv}
\varsigma_{u\to uv}^{\scaleto{(\mathbf X)}{4pt}}(x)
\in
H^{\scaleto{(\mathbf Y)}{4pt}}_{\Phi(uv)}
$
implies 
$
\varsigma_{u\to uv}^{\scaleto{(\mathbf X)}{4pt}}(y)
<^{\scaleto{(\mathbf X)}{4pt}}_{uv}
\accentset{\bullet}{\varsigma_{u\to uv}^{\scaleto{(\mathbf X)}{4pt}}(x)}
<^{\scaleto{(\mathbf X)}{4pt}}_{uv}
\varsigma_{u\to uv}^{\scaleto{(\mathbf X)}{4pt}}(x)
%\in H^{\scaleto{(\mathbf Y)}{4pt}}_{\Phi(uv)}
$,
therefore, 
referring to 
%$u<_{\kappa^{\scaleto{(\mathbf X)}{3pt}}} v$ and
(\ref{KATEGOR_RendeZesINNOVATIVAN})
and
(\ref{KATEGOR_P5}),  
$
\varsigma_{u\to uv}^{\scaleto{(\mathbf X)}{4pt}}(y)
<^{\scaleto{(\mathbf X)}{4pt}}
\accentset{\bullet}{\varsigma_{u\to uv}^{\scaleto{(\mathbf X)}{4pt}}(x)}
<^{\scaleto{(\mathbf X)}{4pt}}
\varsigma_{u\to uv}^{\scaleto{(\mathbf X)}{4pt}}(x)
$.
Applying \ref{KATEGOR_OrderINGresZe} ensures
$
\varphi(\varsigma_{u\to uv}^{\scaleto{(\mathbf X)}{4pt}}(y))
\leq^{\scaleto{(\mathbf X)}{4pt}}
\varphi(\accentset{\bullet}{\varsigma_{u\to uv}^{\scaleto{(\mathbf X)}{4pt}}(x)})
\leq^{\scaleto{(\mathbf X)}{4pt}}
\varphi(\varsigma_{u\to uv}^{\scaleto{(\mathbf X)}{4pt}}(x))
$.
Assuming the opposite of the statement, which is
$
\varphi(\varsigma_{u\to uv}^{\scaleto{(\mathbf X)}{4pt}}(y))
\overset{(\ref{PhiSlice})}{=}
\Phi_{uv}(\varsigma_{u\to uv}^{\scaleto{(\mathbf X)}{4pt}}(y)) 
=
\Phi_{uv}(\varsigma_{u\to uv}^{\scaleto{(\mathbf X)}{4pt}}(x)) 
\overset{(\ref{PhiSlice})}{=}
\varphi(\varsigma_{u\to uv}^{\scaleto{(\mathbf X)}{4pt}}(x))
$
by \ref{KATEGOR_homomorfizmus},
would yield
$
\varphi\left(\accentset{\bullet}{\varsigma^{\scaleto{(\mathbf X)}{4pt}}_{u\to uv}(x)}\right)
=
\varphi(\varsigma^{\scaleto{(\mathbf X)}{4pt}}_{u\to uv}(x))
$,
a contradiction, since
$
\varphi\left(\accentset{\bullet}{\varsigma^{\scaleto{(\mathbf X)}{4pt}}_{u\to uv}(x)}\right)
\overset{(\ref{KATEGOR_XHiGYkESZUL})}{=}
\varphi\left(\giksz{\varsigma^{\scaleto{(\mathbf X)}{4pt}}_{u\to uv}(x)}{\negaiksz{(uv)}}\right)
=
\gipsz{\varphi(\varsigma^{\scaleto{(\mathbf X)}{4pt}}_{u\to uv}(x))}{\varphi\left(\negaiksz{(uv)}\right)}
\overset{(\ref{EzaV})}{=}
\gipsz{\varphi(\varsigma^{\scaleto{(\mathbf X)}{4pt}}_{u\to uv}(x))}{\negaipsz{\varphi(uv)}}
\overset{(\ref{KATEGOR_XHiGYkESZUL})}{=} 
\accentset{\bullet}{\varphi(\varsigma^{\scaleto{(\mathbf X)}{4pt}}_{u\to uv}(x))}$,
where the last equality uses $\Phi(uv)\in\kappa^{\scaleto{(\mathbf Y)}{4pt}}_I$ and
that, by \ref{KATEGOR_subgroups_and_complements}, $\varphi(\varsigma^{\scaleto{(\mathbf X)}{4pt}}_{u\to uv}(x))\in H^{\scaleto{(\mathbf Y)}{4pt}}_{\varphi(uv)}$.

\bigskip
\bigskip
\noindent
(2) Assume \ref{KATEGOR_embedding}-\ref{KATEGOR_injective}.
We shall refer to (\ref{KATEGOR_ViSSZaEpulEredeti}) in the following more detailed form:
\begin{equation}\label{KATEGOR_ViSSZaEpul}
\varphi(x)
=
\left\{
\begin{array}{ll}
\Phi(x)\left(\overset{(\ref{PhiSlice})}{=}
\Phi_u(x)
\overset{(\ref{DEFgamma})}{=} 
\Phi_u(
\gamma^{\scaleto{(\mathbf X)}{4pt}}_u(x))
)
\right) 
& \mbox{ if $x\in G^{\scaleto{(\mathbf X)}{4pt}}_u$},\\
\Phi(a)\left(\overset{(\ref{PhiSlice})}{=}\Phi_u(a)
\overset{(\ref{DEFgamma})}{=} 
\Phi_u(
\gamma^{\scaleto{(\mathbf X)}{4pt}}_u(x))
)
\right) & \mbox{ if $\accentset{\bullet}a=x\in \accentset{\bullet}H_u$ and $\Phi(u)\in\kappa_o^{\scaleto{(\mathbf Y)}{4pt}}$},\\
\accentset{\bullet}{\Phi(a)}\left(\overset{(\ref{PhiSlice})}{=}\accentset{\bullet}{\Phi_u(a)}
\overset{(\ref{DEFgamma})}{=} 
\accentset{\bullet}{
\Phi_u(
\gamma^{\scaleto{(\mathbf X)}{4pt}}_u(x))
)
}
\right)
 & \mbox{ if $\accentset{\bullet}a=x\in \accentset{\bullet}H_u$ and $\Phi(u)\in\kappa_I^{\scaleto{(\mathbf Y)}{4pt}}$},\\
\end{array}
\right.
\footnote{In the expression $\accentset{\bullet}{\Phi_u(a)}$ the $\accentset{\bullet}{.}$ operation is computed in the respective layer group (in \textbf{\textit{G$^{\scaleto{(\mathbf Y)}{4pt}}_{\Phi(u)}$}}).}
\end{equation}
from which it follows that for $x\in X$,
\begin{equation}\label{EziStrUe}
\gamma^{\scaleto{(\mathbf Y)}{4pt}}_{\Phi(u)}(\varphi(x))
=
\Phi_u(
\gamma^{\scaleto{(\mathbf X)}{4pt}}_u(x))
)
.
\end{equation}

\bigskip
\noindent
\ref{KATEGOR_rESzhaLMaZ}:
For $x\in X$,
\begin{equation}\label{KATEGOR_ViSSZaEpulCALC}
\varphi(x)
\overset{(\ref{KATEGOR_ViSSZaEpul})}{=}
\left\{
\begin{array}{ll}
\Phi_u(x) 
\overset{\ref{KATEGOR_homomorfizmus}}{\in}
G^{\scaleto{(\mathbf Y)}{4pt}}_{\Phi(u)}
\overset{(\ref{KATEGOR_IkszU}),(\ref{KATEGOR_EZazX})}{\subseteq}Y
& \mbox{ if $x\in G^{\scaleto{(\mathbf X)}{4pt}}$},\\
\Phi_u(a)
\overset{\ref{KATEGOR_subgroups_and_complements}}{\in}
H^{\scaleto{(\mathbf Y)}{4pt}}_{\Phi(u)}
\overset{\ref{KATEGOR_G2}}{\subseteq}
 G^{\scaleto{(\mathbf Y)}{4pt}}_{\Phi(u)}
 \subseteq Y
& \mbox{ if $\accentset{\bullet}a=x\in \accentset{\bullet}H^{\scaleto{(\mathbf X)}{4pt}}_u$ and $\Phi(u)\in\kappa_o^{\scaleto{(\mathbf Y)}{4pt}}$},\\
\accentset{\bullet}{\Phi_u(a)} \in \accentset{\bullet} H^{\scaleto{(\mathbf Y)}{4pt}}_{\Phi(u)}
\overset{(\ref{KATEGOR_IkszU}),(\ref{KATEGOR_EZazX})}{\subseteq} 
Y
& \mbox{ if $\accentset{\bullet}a=x\in \accentset{\bullet}H^{\scaleto{(\mathbf X)}{4pt}}_u$ and $\Phi(u)\in\kappa_I^{\scaleto{(\mathbf Y)}{4pt}}$}.\\
\end{array}
\right. 
\end{equation}

\bigskip
\noindent
\ref{KATEGOR_OrderINGresZe}:
Let $x,y\in X$.
Then $x\in L^{\scaleto{(\mathbf X)}{4pt}}_u$ and $y\in L^{\scaleto{(\mathbf X)}{4pt}}_v$ for some $u,v\in\kappa^{\scaleto{(\mathbf X)}{4pt}}$, c.f.\ (\ref{KATEGOR_EZazX}).
Assume  $x<^{\scaleto{(\mathbf X)}{4pt}} y$. 
By (\ref{KATEGOR_RendeZesINNOVATIVAN}), 
it 
%$x<^{\scaleto{(\mathbf X)}{4pt}} y$ 
is equivalent to
\begin{equation}\label{KATEGOR_RendeZesINNOVATIVANelsoNEMKELLb}
\mbox{
$
\rho^{\scaleto{(\mathbf X)}{4pt}}_{uv}(x)
<^{\scaleto{(\mathbf X)}{4pt}}_{uv}\rho^{\scaleto{(\mathbf X)}{4pt}}_{uv}(y)$ or 
$\rho^{\scaleto{(\mathbf X)}{4pt}}_{uv}(x)=\rho^{\scaleto{(\mathbf X)}{4pt}}_{uv}(y)$ and $u<_{\kappa^{\scaleto{(\mathbf X)}{4pt}}} v$
}
,
\end{equation}
where
by (\ref{KATEGOR_P5}),
\begin{equation}\label{ezVOLT} 
\begin{array}{ll}
\rho^{\scaleto{(\mathbf X)}{4pt}}_{uv}(x)
=
\left\{
\begin{array}{ll}
x & \mbox{ if $x\in L^{\scaleto{(\mathbf X)}{4pt}}_u$ and $u\geq_{\kappa^{\scaleto{(\mathbf X)}{3pt}}} v$}\\
\varsigma^{\scaleto{(\mathbf X)}{4pt}}_{u\to uv}(\gamma^{\scaleto{(\mathbf X)}{4pt}}_u(x))) & \mbox{ otherwise}\\
\end{array}
\right.
,
\\
\rho^{\scaleto{(\mathbf X)}{4pt}}_{uv}(y)
=
\left\{
\begin{array}{ll}
y & \mbox{ if $y\in L^{\scaleto{(\mathbf X)}{4pt}}_v$ and $v\geq_{\kappa^{\scaleto{(\mathbf X)}{3pt}}} u$}\\
\varsigma^{\scaleto{(\mathbf X)}{4pt}}_{v\to uv}(\gamma^{\scaleto{(\mathbf X)}{4pt}}_v(y))) & \mbox{ otherwise}\\
\end{array}
\right.
.
\end{array}
\end{equation}
We need to prove
$\varphi(x)\leq^{\scaleto{(\mathbf Y)}{4pt}}\varphi(y)$ which,
referring to 
$
\Phi(u)\Phi(v)
\overset{\ref{KATEGOR_embedding}}{=}
\Phi(uv)
$
is equivalent to 
\begin{equation}\label{KATEGOR_RendeZesINNOVATIVANmasodikNEMKELLc}
\mbox{
$\rho^{\scaleto{(\mathbf Y)}{4pt}}_{\Phi(uv)}(\varphi(x))<^{\scaleto{(\mathbf Y)}{4pt}}_{\Phi(uv)}\rho^{\scaleto{(\mathbf Y)}{4pt}}_{\Phi(uv)}(\varphi(y))$ 
or 
$\rho^{\scaleto{(\mathbf Y)}{4pt}}_{\Phi(uv)}(\varphi(x))=\rho^{\scaleto{(\mathbf Y)}{4pt}}_{\Phi(uv)}(\varphi(y))$ and 
$\Phi(u)\leq^{\scaleto{(\mathbf Y)}{4pt}}_{\kappa^{\scaleto{(\mathbf Y)}{4pt}}} \Phi(v)$
}
,
\end{equation}
c.f.\ (\ref{KATEGOR_RendeZesINNOVATIVAN}).
We claim that for $x,y\in X$, 
\begin{equation}\label{ezLETT}
%\color{blue}
\begin{array}{ll}
\begin{array}{ccl}
\rho^{\scaleto{(\mathbf Y)}{4pt}}_{\Phi(uv)}(\varphi(x))
=
\left\{
\begin{array}{lll}
\accentset{\bullet}{\Phi_{uv}(\varsigma_{u\to uv}^{\scaleto{(\mathbf X)}{4pt}}(\gamma^{\scaleto{(\mathbf X)}{4pt}}_u(x)))} & \mbox{ if $x\in \accentset{\bullet}H^{\scaleto{(\mathbf X)}{4pt}}_u$,}   & \mbox{ $\kappa_I^{\scaleto{(\mathbf Y)}{4pt}}\ni\Phi(u)\geq_{\kappa^{\scaleto{(\mathbf Y)}{3pt}}} \Phi(v)$}\\
\Phi_{uv}(\varsigma_{u\to uv}^{\scaleto{(\mathbf X)}{4pt}}(\gamma^{\scaleto{(\mathbf X)}{4pt}}_u(x)))
&
\mbox{ otherwise}\\
\end{array}
\right.
,
\end{array}
\\
\\
\begin{array}{ccl}
\rho^{\scaleto{(\mathbf Y)}{4pt}}_{\Phi(uv)}(\varphi(y))
=
\left\{
\begin{array}{lll}
\accentset{\bullet}{\Phi_{uv}(\varsigma_{v\to uv}^{\scaleto{(\mathbf X)}{4pt}}(\gamma^{\scaleto{(\mathbf X)}{4pt}}_v(y)))} & \mbox{ if $y\in \accentset{\bullet}H^{\scaleto{(\mathbf X)}{4pt}}_v$,}   & \mbox{ $\kappa_I^{\scaleto{(\mathbf Y)}{4pt}}\ni\Phi(v)\geq_{\kappa^{\scaleto{(\mathbf Y)}{3pt}}} \Phi(u)$}\\
\Phi_{uv}(\varsigma_{v\to uv}^{\scaleto{(\mathbf X)}{4pt}}(\gamma^{\scaleto{(\mathbf X)}{4pt}}_v(y)))
&
\mbox{ otherwise}\\
\end{array}
\right.
\end{array}
.
\end{array}
\end{equation}
We shall prove only the first one, the other proof is analogous.
For
$x\in L^{\scaleto{(\mathbf X)}{4pt}}_u$,
$$
\color{midgrey}
%\footnotesize
\small
\begin{array}{ccl}
\color{black}
\rho^{\scaleto{(\mathbf Y)}{4pt}}_{\Phi(uv)}(\varphi(x))
&
\color{black}
\overset{(\ref{KATEGOR_P5})}{=}
&
\color{black}
\left\{
\begin{array}{ll}
\varsigma^{\scaleto{(\mathbf Y)}{4pt}}_{\Phi(u)\to \Phi(v)}(
\gamma^{\scaleto{(\mathbf Y)}{4pt}}_{\Phi(u)}(\varphi(x))
)
& 
\mbox{ if %$\varphi(x)\in L^{\scaleto{(\mathbf Y)}{4pt}}_{\Phi(u)}$,  
$\Phi(u)<_{\kappa^{\scaleto{(\mathbf Y)}{3pt}}} \Phi(v)$}\\
\varphi(x)
&
\mbox{ if %$\varphi(x)\in L^{\scaleto{(\mathbf Y)}{4pt}}_{\Phi(u)}$,  
$\Phi(u)\geq_{\kappa^{\scaleto{(\mathbf Y)}{3pt}}} \Phi(v)$}
\end{array}
\right.
\\
&
\color{black}
\overset{(\ref{KATEGOR_ViSSZaEpul})}{=}
&
\left\{
\begin{array}{ll}
\varsigma^{\scaleto{(\mathbf Y)}{4pt}}_{\Phi(u)\to \Phi(v)}(
\color{black}
\Phi_u(
\gamma^{\scaleto{(\mathbf X)}{4pt}}_u(x))
\color{midgrey}
)
& 
\mbox{ if %$\varphi(x)\in L^{\scaleto{(\mathbf Y)}{4pt}}_{\Phi(u)}$,  
$\Phi(u)<_{\kappa^{\scaleto{(\mathbf Y)}{3pt}}} \Phi(v)$}\\
\color{black} \accentset{\bullet}{\Phi_u(\gamma^{\scaleto{(\mathbf X)}{4pt}}_u(x)))} & 
\mbox{ \color{black} if $x\in \accentset{\bullet}H^{\scaleto{(\mathbf X)}{4pt}}_u$, %\color{midgrey}$\varphi(x)\in L^{\scaleto{(\mathbf Y)}{4pt}}_{\Phi(u)}$,  
\color{black}$\kappa_I^{\scaleto{(\mathbf Y)}{4pt}}\ni\color{midgrey}\Phi(u)\geq_{\kappa^{\scaleto{(\mathbf Y)}{3pt}}} \Phi(v)$}\\\color{black} 
\Phi_u(
\gamma^{\scaleto{(\mathbf X)}{4pt}}_u(x))
)
 & \mbox{ {\color{black} otherwise}}\\
% & \mbox{ {\color{black} if $x\in G^{\scaleto{(\mathbf X)}{4pt}}_u$}, $\varphi(x)\in L^{\scaleto{(\mathbf Y)}{4pt}}_{\Phi(u)}$,  $\Phi(u)\geq_{\kappa^{\scaleto{(\mathbf Y)}{3pt}}} \Phi(v)$}\\
%\color{black} 
%\Phi_u(
%\gamma^{\scaleto{(\mathbf X)}{4pt}}_u(x))
%)
%& 
%\mbox{ \color{black} if $x\in \accentset{\bullet}H^{\scaleto{(\mathbf X)}{4pt}}_u$, \color{midgrey}$\varphi(x)\in L^{\scaleto{(\mathbf Y)}{4pt}}_{\Phi(u)}$,  \color{black}$\kappa_o^{\scaleto{(\mathbf Y)}{4pt}}\ni\color{midgrey}\Phi(u)\geq_{\kappa^{\scaleto{(\mathbf Y)}{3pt}}} \Phi(v)$}\\
\end{array}
\right.
\\
&
\color{black}
\overset{\ref{KATEGOR_commutes},\,(\ref{KATEGOR_ViSSZaEpulCALC})}{=}
&
\left\{
\begin{array}{ll}
\color{black}
\Phi_v(\varsigma_{u\to v}^{\scaleto{(\mathbf X)}{4pt}}(
\color{midgrey}
\gamma^{\scaleto{(\mathbf X)}{4pt}}_u(x)
\color{black}
)))
&
\mbox{ \color{midgrey}if %$x\in L^{\scaleto{(\mathbf X)}{4pt}}_u$, 
$\Phi(u)<_{\kappa^{\scaleto{(\mathbf Y)}{3pt}}} \Phi(v)$}
\\
\accentset{\bullet}{\Phi_u(
\gamma^{\scaleto{(\mathbf X)}{4pt}}_u(x))
)
}
& \mbox{ \color{midgrey}
if $x\in \accentset{\bullet}H^{\scaleto{(\mathbf X)}{4pt}}_u$,  $\kappa_I^{\scaleto{(\mathbf Y)}{4pt}}\ni\Phi(u)\geq_{\kappa^{\scaleto{(\mathbf Y)}{3pt}}} \Phi(v)$}\\\Phi_u(
\gamma^{\scaleto{(\mathbf X)}{4pt}}_u(x))
)
 & \mbox{ \color{midgrey}
otherwise}\\
% & \mbox{ \color{black} if $x\in G^{\scaleto{(\mathbf X)}{4pt}}_u$,  $\Phi(u)\geq_{\kappa^{\scaleto{(\mathbf Y)}{3pt}}} \Phi(v)$}\\
%\Phi_u(
%\gamma^{\scaleto{(\mathbf X)}{4pt}}_u(x))
%)
% & \mbox{ \color{black}
%if $x\in \accentset{\bullet}H^{\scaleto{(\mathbf X)}{4pt}}_u$,  $\kappa_o^{\scaleto{(\mathbf Y)}{4pt}}\ni\Phi(u)
%\geq_{\kappa^{\scaleto{(\mathbf Y)}{3pt}}}
%\Phi(v)$}\\
\end{array}
\right.
\end{array}
.
$$
In the first row, $u<_{\kappa^{\scaleto{(\mathbf X)}{3pt}}}v$ follows from 
$\Phi(u)<_{\kappa^{\scaleto{(\mathbf Y)}{3pt}}} \Phi(v)$
by \ref{KATEGOR_embedding}, hence $v=uv$ and thus we may write the value in that row as 
$\Phi_{uv}(\varsigma_{u\to uv}^{\scaleto{(\mathbf X)}{4pt}}(\gamma^{\scaleto{(\mathbf X)}{4pt}}_u(x)))$.
Consider the last two rows.
If $u\geq_{\kappa^{\scaleto{(\mathbf X)}{3pt}}}v$ then $u=uv$
and thus 
by \ref{IDes}
we may write the values as
$
\accentset{\bullet}{
\Phi_{uv}(\varsigma_{u\to uv}^{\scaleto{(\mathbf X)}{4pt}}(\gamma^{\scaleto{(\mathbf X)}{4pt}}_u(x)))
}
$
and
$\Phi_{uv}(\varsigma_{u\to uv}^{\scaleto{(\mathbf X)}{4pt}}(\gamma^{\scaleto{(\mathbf X)}{4pt}}_u(x)))$,
respectively.
Finally, if $u<_{\kappa^{\scaleto{(\mathbf X)}{3pt}}}v$ then 
$\Phi(u)\leq_{\kappa^{\scaleto{(\mathbf Y)}{3pt}}} \Phi(v)$ holds
by \ref{KATEGOR_embedding}, hence we obtain
$\Phi(u)=\Phi(v)$.
Therefore, $\Phi_u=\Phi_v\circ\varsigma_{u\to v}^{\scaleto{(\mathbf X)}{4pt}}$
holds by \ref{KATEGOR_commutes},
and thus also here 
we can write the values as
$
\accentset{\bullet}{
\Phi_{uv}(\varsigma_{u\to uv}^{\scaleto{(\mathbf X)}{4pt}}(\gamma^{\scaleto{(\mathbf X)}{4pt}}_u(x)))
}
$
and
$\Phi_{uv}(\varsigma_{u\to uv}^{\scaleto{(\mathbf X)}{4pt}}(\gamma^{\scaleto{(\mathbf X)}{4pt}}_u(x)))$,
respectively.
Summing up, (\ref{ezLETT}) holds.
Consequently,
\begin{equation}\label{Atugorja}
\begin{array}{r}
\gamma^{\scaleto{(\mathbf Y)}{4pt}}_{\Phi(uv)}(\rho^{\scaleto{(\mathbf Y)}{4pt}}_{\Phi(uv)}(\varphi(x)))
=
\Phi_{uv}(\varsigma_{u\to uv}^{\scaleto{(\mathbf X)}{4pt}}(\gamma^{\scaleto{(\mathbf X)}{4pt}}_u(x))),
\\
\gamma^{\scaleto{(\mathbf Y)}{4pt}}_{\Phi(uv)}(\rho^{\scaleto{(\mathbf Y)}{4pt}}_{\Phi(uv)}(\varphi(y)))
=
\Phi_{uv}(\varsigma_{v\to uv}^{\scaleto{(\mathbf X)}{4pt}}(\gamma^{\scaleto{(\mathbf X)}{4pt}}_v(y)))
.
\end{array}
\end{equation}
%Therefore,
%\begin{equation}\label{JolAtugorja}
%\begin{array}{r}
%\gamma^{\scaleto{(\mathbf Y)}{4pt}}_{\Phi(uv)}(\rho^{\scaleto{(\mathbf Y)}{4pt}}_{\Phi(uv)}(\varphi(x)))
%=
%\Phi_{uv}(
%\gamma^{\scaleto{(\mathbf X)}{4pt}}_{uv}
%(\rho^{\scaleto{(\mathbf X)}{4pt}}_{uv}(x))
%),
%\\
%\gamma^{\scaleto{(\mathbf Y)}{4pt}}_{\Phi(uv)}(\rho^{\scaleto{(\mathbf Y)}{4pt}}_{\Phi(uv)}(\varphi(y)))
%=
%\Phi_{uv}(\gamma^{\scaleto{(\mathbf X)}{4pt}}_{uv}
%(\rho^{\scaleto{(\mathbf X)}{4pt}}_{uv}(y))
%).
%\end{array}
%\end{equation}

\begin{enumerate}[I)]
\item
Assume $u\leq_{\kappa^{\scaleto{(\mathbf X)}{3pt}}} v$.
\\
%By letting $z=x$ if $x\in G^{\scaleto{(\mathbf X)}{4pt}}_u$ and $z=a$ if $\accentset{\bullet}a=x\in \accentset{\bullet}H^{\scaleto{(\mathbf X)}{4pt}}_u$, by (\ref{ezVOLT}), (\ref{KATEGOR_RendeZesINNOVATIVANelsoNEMKELLb}) becomes equivalent to
%$$
%\varsigma^{\scaleto{(\mathbf X)}{4pt}}_{u\to v}(z)
%\leq^{\scaleto{(\mathbf X)}{4pt}}_{v}
%y.
%$$
By \ref{KATEGOR_embedding}, $\Phi(u)\leq_{\kappa^{\scaleto{(\mathbf Y)}{3pt}}} \Phi(v)$ follows.
Therefore, referring to (\ref{KATEGOR_RendeZesINNOVATIVANmasodikNEMKELLc}),
we need to prove 
$
%\begin{equation}\label{KATEGOR_RendeZesINNOVATIVANmasodikNEMKELLc2}
\rho^{\scaleto{(\mathbf Y)}{4pt}}_{\Phi(v)}(\varphi(x))
\leq^{\scaleto{(\mathbf Y)}{4pt}}_{\Phi(v)}
\rho^{\scaleto{(\mathbf Y)}{4pt}}_{\Phi(v)}(\varphi(y))
%\end{equation}
$.
By (\ref{KATEGOR_KibovitettRendezesITTIS55}), (\ref{KATEGOR_ViSSZaEpulCALC}),  %(\ref{ezLETT}), 
and (\ref{Atugorja})
it is equivalent to
\begin{equation}\label{EzkellettTEHAT}
\small
\begin{array}{cc}
\Phi_{v}(\varsigma_{u\to v}^{\scaleto{(\mathbf X)}{4pt}}(\gamma^{\scaleto{(\mathbf X)}{4pt}}_u(x)))
\prec^{\scaleto{(\mathbf Y)}{4pt}}_{\Phi(uv)}
\Phi_{v}(\gamma^{\scaleto{(\mathbf X)}{4pt}}_v(y))
\mbox{ or}
\\
\left(
\Phi(v)\in\kappa_I,
\mbox{
$
\Phi_{v}(\varsigma_{u\to v}^{\scaleto{(\mathbf X)}{4pt}}(\gamma^{\scaleto{(\mathbf X)}{4pt}}_u(x)))
=
\Phi_{v}(\gamma^{\scaleto{(\mathbf X)}{4pt}}_v(y))
$
and
$
\left(
\mbox{
$\rho^{\scaleto{(\mathbf Y)}{4pt}}_{\Phi(v)}(\varphi(x))\in\accentset{\bullet}H^{\scaleto{(\mathbf X)}{4pt}}_{\Phi(v)}$
or
$
\rho^{\scaleto{(\mathbf Y)}{4pt}}_{\Phi(v)}(
\varphi(y)
)
\in
G^{\scaleto{(\mathbf Y)}{4pt}}_{\Phi(v)}
$
}\right)$
}\right)
.
\end{array}
\end{equation}
Assume
$u<_{\kappa^{\scaleto{(\mathbf X)}{3pt}}} v$.
By (\ref{ezVOLT}),
(\ref{KATEGOR_RendeZesINNOVATIVANelsoNEMKELLb})
is equivalent to 
$
\varsigma^{\scaleto{(\mathbf X)}{4pt}}_{u\to v}(\gamma^{\scaleto{(\mathbf X)}{4pt}}_u(x))
\leq^{\scaleto{(\mathbf X)}{4pt}}_{v}
y
.
$
\begin{enumerate}[a)]
\item 
If  
$y\in \accentset{\bullet}{H^{\scaleto{(\mathbf X)}{4pt}}_v}$
for  
%$\Phi(v)\in\kappa^{\scaleto{(\mathbf Y)}{3pt}}_I$ then 
$
v
\in
%\overset{\ref{KATEGOR_particio}}{\in}
\kappa^{\scaleto{(\mathbf X)}{3pt}}_I$ 
and
$\Phi(v)\in\kappa_I^{\scaleto{(\mathbf Y)}{4pt}}$
then
$
%{H^{\scaleto{(\mathbf X)}{4pt}}_v}\overset{\ref{KATEGOR_G2}}{\ni}
\varsigma^{\scaleto{(\mathbf X)}{4pt}}_{u\to v}(\gamma^{\scaleto{(\mathbf X)}{4pt}}_u(x))
\leq^{\scaleto{(\mathbf X)}{4pt}}_{v}
y
\overset{(\ref{DEFgamma}),\,(\ref{KATEGOR_KibovitettRendezesITTIS})}{<^{\scaleto{(\mathbf X)}{4pt}}_{v}}
\gamma^{\scaleto{(\mathbf X)}{4pt}}_v(y)
\overset{(\ref{DEFgamma})}{\in}
H^{\scaleto{(\mathbf X)}{4pt}}_v
$
holds, hence
$
\varsigma^{\scaleto{(\mathbf X)}{4pt}}_{u\to v}(\gamma^{\scaleto{(\mathbf X)}{4pt}}_u(x))
\overset{(\ref{KATEGOR_KibovitettRendezesITTIS})}{\prec^{\scaleto{(\mathbf X)}{4pt}}_{v}}
\gamma^{\scaleto{(\mathbf X)}{4pt}}_v(y)
\overset{\ref{IDes}}{=}
\varsigma^{\scaleto{(\mathbf X)}{4pt}}_{v\to v}(\gamma^{\scaleto{(\mathbf X)}{4pt}}_v(y))
\in
H^{\scaleto{(\mathbf X)}{4pt}}_v
$
follows
yielding
$
\Phi_v(
\varsigma^{\scaleto{(\mathbf X)}{4pt}}_{u\to v}(\gamma^{\scaleto{(\mathbf X)}{4pt}}_u(x))
)
\overset{\ref{KATEGOR_injective}}{\prec^{\scaleto{(\mathbf Y)}{4pt}}_{\Phi(v)}}
\Phi_v(
\gamma^{\scaleto{(\mathbf X)}{4pt}}_v(y)
)
$
and in turn (\ref{EzkellettTEHAT}).

\item 
If $y\in G^{\scaleto{(\mathbf X)}{4pt}}_v$
or 
($y\in \accentset{\bullet}{H^{\scaleto{(\mathbf X)}{4pt}}_v}$
for  
$
v
\in
\kappa^{\scaleto{(\mathbf X)}{3pt}}_I
$ 
and
$\Phi(v)\in\kappa_o^{\scaleto{(\mathbf Y)}{4pt}}$)
then
$
G^{\scaleto{(\mathbf X)}{4pt}}_v
\overset{(\ref{tipusSIGMA})}{\ni}
\varsigma^{\scaleto{(\mathbf X)}{4pt}}_{u\to v}(\gamma^{\scaleto{(\mathbf X)}{4pt}}_u(x))
\leq^{\scaleto{(\mathbf X)}{4pt}}_{v}
y
\overset{(\ref{DEFgamma})}{\leq}
\gamma^{\scaleto{(\mathbf X)}{4pt}}_v(y)
\in
G^{\scaleto{(\mathbf X)}{4pt}}_v
$
implies
$
\varsigma^{\scaleto{(\mathbf X)}{4pt}}_{u\to v}(\gamma^{\scaleto{(\mathbf X)}{4pt}}_u(x))
\overset{(\ref{KATEGOR_KibovitettRendezesITTIS})}{\preceq^{\scaleto{(\mathbf X)}{4pt}}_{v}}
\gamma^{\scaleto{(\mathbf X)}{4pt}}_v(y)
$
and hence
$
\Phi_v(\varsigma^{\scaleto{(\mathbf X)}{4pt}}_{u\to v}(\gamma^{\scaleto{(\mathbf X)}{4pt}}_u(x)))
\overset{\ref{KATEGOR_homomorfizmus}}{\preceq^{\scaleto{(\mathbf Y)}{4pt}}_{\Phi(v)}}
\Phi_v(\gamma^{\scaleto{(\mathbf X)}{4pt}}_v(y))
$,
which together with 
$
\varphi(y)
\overset{(\ref{KATEGOR_ViSSZaEpulCALC})}{\in}
G^{\scaleto{(\mathbf Y)}{4pt}}_{\Phi(v)}
$
and in turn,
$
\rho^{\scaleto{(\mathbf Y)}{4pt}}_{\Phi(v)}(\varphi(y))
\overset{(\ref{KATEGOR_P5})}{\in}
G^{\scaleto{(\mathbf Y)}{4pt}}_{\Phi(v)}
$
ensures (\ref{EzkellettTEHAT}).
\end{enumerate}
Assume
$u=v$.
By (\ref{ezVOLT}),
(\ref{KATEGOR_RendeZesINNOVATIVANelsoNEMKELLb})
%$z<^{\scaleto{(\mathbf X)}{4pt}} y$
is equivalent to
$
x
<^{\scaleto{(\mathbf X)}{4pt}}_{u}
y
$,
which, by (\ref{KATEGOR_KibovitettRendezesITTIS66})
is equivalent to
\begin{equation}\label{sd8758jhGJ}
\mbox{
$\gamma^{\scaleto{(\mathbf X)}{4pt}}_u(x)
\prec_u 
\gamma^{\scaleto{(\mathbf X)}{4pt}}_u(y)$
or 
$
(
\gamma^{\scaleto{(\mathbf X)}{4pt}}_u(x)=\gamma^{\scaleto{(\mathbf X)}{4pt}}_u(y)
$, 
$u\in\kappa^{\scaleto{(\mathbf Y)}{4pt}}_I$,
$x\in\accentset{\bullet}H^{\scaleto{(\mathbf X)}{4pt}}_u$,
$y\in G^{\scaleto{(\mathbf X)}{4pt}}_u
)
$

}
.
\end{equation}
\begin{enumerate}[a)]
\item 
If $y\in G^{\scaleto{(\mathbf X)}{4pt}}_u$
then
$
\rho^{\scaleto{(\mathbf Y)}{4pt}}_{\Phi(u)}(\varphi(y))
\overset{(\ref{ezLETT})}{=}
\Phi_{u}(\gamma^{\scaleto{(\mathbf X)}{4pt}}_u(y))
\overset{\ref{KATEGOR_homomorfizmus}}{\in}
G^{\scaleto{(\mathbf Y)}{4pt}}_{\Phi(u)}
$
holds, and hence
(\ref{EzkellettTEHAT}) is equivalent to
$
\Phi_{u}(\gamma^{\scaleto{(\mathbf X)}{4pt}}_u(x))
\preceq^{\scaleto{(\mathbf Y)}{4pt}}_{\Phi(u)}
\Phi_{u}(\gamma^{\scaleto{(\mathbf X)}{4pt}}_u(y))
$
which follows from (\ref{sd8758jhGJ}) by \ref{KATEGOR_homomorfizmus}.
\item
Therefore, it remains to assume $y\in\accentset{\bullet}H^{\scaleto{(\mathbf X)}{4pt}}_u$ for $u\in\kappa^{\scaleto{(\mathbf Y)}{4pt}}_I$,
in which case
(\ref{sd8758jhGJ}) is equivalent to
$
\gamma^{\scaleto{(\mathbf X)}{4pt}}_u(x)
\prec_u 
\gamma^{\scaleto{(\mathbf X)}{4pt}}_u(y)
\overset{(\ref{DEFgamma})}{\in}
H^{\scaleto{(\mathbf X)}{4pt}}_u
$.
Therefore, 
if $\Phi(u)\in \kappa_I^{\scaleto{(\mathbf Y)}{4pt}}$ then 
$\Phi_{u}(\gamma^{\scaleto{(\mathbf X)}{4pt}}_u(x))
\overset{\ref{KATEGOR_injective}}{\prec^{\scaleto{(\mathbf Y)}{4pt}}_{\Phi(u)}}
\Phi_{u}(\gamma^{\scaleto{(\mathbf X)}{4pt}}_u(y))
$
holds and hence so does (\ref{EzkellettTEHAT}),
whereas
if $\Phi(u)\in \kappa_o^{\scaleto{(\mathbf Y)}{4pt}}$ then 
$
\Phi_v(\gamma^{\scaleto{(\mathbf X)}{4pt}}_u(x))
\overset{\ref{KATEGOR_homomorfizmus}}{\preceq^{\scaleto{(\mathbf Y)}{4pt}}_{\Phi(v)}}
\Phi_v(\gamma^{\scaleto{(\mathbf X)}{4pt}}_v(y))
$
together with 
$
\varphi(y)
\overset{(\ref{KATEGOR_ViSSZaEpulCALC})}{\in}
G^{\scaleto{(\mathbf Y)}{4pt}}_{\Phi(v)}
$
and in turn,
$
\rho^{\scaleto{(\mathbf Y)}{4pt}}_{\Phi(v)}(\varphi(y))
\overset{(\ref{KATEGOR_P5})}{\in}
G^{\scaleto{(\mathbf Y)}{4pt}}_{\Phi(v)}
$
ensures (\ref{EzkellettTEHAT}).
\end{enumerate}

\item 
Assume
%ezmarismert$x\in L^{\scaleto{(\mathbf X)}{4pt}}_u$ and 
$u>_{\kappa^{\scaleto{(\mathbf X)}{3pt}}} v$.
\\
By \ref{KATEGOR_embedding}, $\Phi(u)\geq_{\kappa^{\scaleto{(\mathbf Y)}{3pt}}} \Phi(v)$ holds.
By (\ref{ezVOLT}),
(\ref{KATEGOR_RendeZesINNOVATIVANelsoNEMKELLb})
%$z<^{\scaleto{(\mathbf X)}{4pt}} y$
is equivalent to
$$
x
<^{\scaleto{(\mathbf X)}{4pt}}_{uv}
\varsigma^{\scaleto{(\mathbf X)}{4pt}}_{v\to uv}(\gamma^{\scaleto{(\mathbf X)}{4pt}}_v(y))
.
$$

\begin{enumerate}[a)]
\item 
Assume $x\in G^{\scaleto{(\mathbf X)}{4pt}}_u$.   
\\
Then 
$
G^{\scaleto{(\mathbf X)}{4pt}}_{uv}\ni
\varsigma^{\scaleto{(\mathbf X)}{4pt}}_{u\to uv}(\gamma^{\scaleto{(\mathbf X)}{4pt}}_{u}(x))
<^{\scaleto{(\mathbf X)}{4pt}}_{uv}
\varsigma^{\scaleto{(\mathbf X)}{4pt}}_{v\to uv}(\gamma^{\scaleto{(\mathbf X)}{4pt}}_v(y))
\in G^{\scaleto{(\mathbf X)}{4pt}}_{uv}
$
follows by (\ref{DEFgamma}) and \ref{IDes}, 
hence by (\ref{KATEGOR_KibovitettRendezesITTIS}) it holds true that
$$
\varsigma^{\scaleto{(\mathbf X)}{4pt}}_{u\to uv}(\gamma^{\scaleto{(\mathbf X)}{4pt}}_{u}(x))
\prec^{\scaleto{(\mathbf X)}{4pt}}_{uv}
\varsigma^{\scaleto{(\mathbf X)}{4pt}}_{v\to uv}(\gamma^{\scaleto{(\mathbf X)}{4pt}}_v(y))
.
$$ 

- If $\Phi(u)\notin \kappa_I^{\scaleto{(\mathbf Y)}{4pt}}$
then
$
\rho^{\scaleto{(\mathbf Y)}{4pt}}_{\Phi(uv)}(\varphi(x))
\overset{(\ref{ezLETT})}{=}
\Phi_{uv}(\varsigma^{\scaleto{(\mathbf X)}{4pt}}_{u\to uv}(\gamma^{\scaleto{(\mathbf X)}{4pt}}_{u}(x)))
\overset{\ref{KATEGOR_homomorfizmus}}{\preceq^{\scaleto{(\mathbf Y)}{4pt}}_{\Phi(uv)}}
\Phi_{uv}(\varsigma^{\scaleto{(\mathbf X)}{4pt}}_{v\to uv}(\gamma^{\scaleto{(\mathbf X)}{4pt}}_v(y)))
$
follows.
The latest is equal to 
$
\rho^{\scaleto{(\mathbf Y)}{4pt}}_{\Phi(uv)}(\varphi(y))
$
by (\ref{ezLETT}) since 
$\kappa_I^{\scaleto{(\mathbf Y)}{4pt}}\ni\Phi(v)\geq_{\kappa^{\scaleto{(\mathbf Y)}{3pt}}} \Phi(u)$
would imply
$\kappa_I^{\scaleto{(\mathbf Y)}{4pt}}\ni\Phi(v)=\Phi(u)\notin\kappa_I^{\scaleto{(\mathbf Y)}{4pt}}$, a contradiction.
\\
- If $\Phi(u)\in \kappa_I^{\scaleto{(\mathbf Y)}{4pt}}$
then
$
\rho^{\scaleto{(\mathbf Y)}{4pt}}_{\Phi(uv)}(\varphi(x))
\overset{(\ref{ezLETT})}{=}
\Phi_{uv}(\varsigma^{\scaleto{(\mathbf X)}{4pt}}_{u\to uv}(\gamma^{\scaleto{(\mathbf X)}{4pt}}_{u}(x)))
\overset{\ref{KATEGOR_injective}}{\prec^{\scaleto{(\mathbf Y)}{4pt}}_{\Phi(uv)}}
\Phi_{uv}(\varsigma^{\scaleto{(\mathbf X)}{4pt}}_{v\to uv}(\gamma^{\scaleto{(\mathbf X)}{4pt}}_v(y)))
$
follows
yielding
$
\rho^{\scaleto{(\mathbf Y)}{4pt}}_{\Phi(uv)}(\varphi(x))
\overset{(\ref{KATEGOR_KibovitettRendezesITTIS})}{<^{\scaleto{(\mathbf Y)}{4pt}}_{\Phi(uv)}}
\Phi_{uv}(\varsigma^{\scaleto{(\mathbf X)}{4pt}}_{v\to uv}(\gamma^{\scaleto{(\mathbf X)}{4pt}}_v(y)))
$.
If $\Phi(u)>_{\kappa^{\scaleto{(\mathbf Y)}{3pt}}} \Phi(v)$ then the latest is equal to
$
%\Phi_u(\varsigma^{\scaleto{(\mathbf X)}{4pt}}_{v\to u}(\gamma^{\scaleto{(\mathbf X)}{4pt}}_v(y)))
%\overset{(\ref{ezLETT})}{=}
\rho^{\scaleto{(\mathbf Y)}{4pt}}_{\Phi(uv)}(\varphi(y))
$
by (\ref{ezLETT}),
thus
%$
%\rho^{\scaleto{(\mathbf Y)}{4pt}}_{\Phi(uv)}(\varphi(x))
%<^{\scaleto{(\mathbf Y)}{4pt}}_{\Phi(uv)}
%\rho^{\scaleto{(\mathbf Y)}{4pt}}_{\Phi(uv)}(\varphi(y))
%$
%and hence
(\ref{KATEGOR_RendeZesINNOVATIVANmasodikNEMKELLc}) holds,
whereas if $\Phi(u)=\Phi(v)$ then 
since by (\ref{ezLETT}), $\rho^{\scaleto{(\mathbf Y)}{4pt}}_{\Phi(uv)}(\varphi(y))$ is equal to either
$\Phi_{uv}(\varsigma^{\scaleto{(\mathbf X)}{4pt}}_{v\to uv}(\gamma^{\scaleto{(\mathbf X)}{4pt}}_v(y)))$
or
$\accentset{\bullet}{\Phi_{uv}(\varsigma^{\scaleto{(\mathbf X)}{4pt}}_{v\to uv}(\gamma^{\scaleto{(\mathbf X)}{4pt}}_v(y)))}$, and since by (\ref{KATEGOR_KibovitettRendezesITTIS}) the latter is the lower cover of the former,
$
\rho^{\scaleto{(\mathbf Y)}{4pt}}_{\Phi(uv)}(\varphi(x))
\leq^{\scaleto{(\mathbf Y)}{4pt}}_{\Phi(uv)}
\rho^{\scaleto{(\mathbf Y)}{4pt}}_{\Phi(uv)}(\varphi(y))
$ 
follows, hence
(\ref{KATEGOR_RendeZesINNOVATIVANmasodikNEMKELLc}) holds.

\item 
Assume $x\in \accentset{\bullet}H^{\scaleto{(\mathbf X)}{4pt}}_u$
for $u\in\kappa_I^{\scaleto{(\mathbf X)}{4pt}}$.
\\
Then
$\Phi(u)\overset{\ref{KATEGOR_particio}}{\in}\kappa^{\scaleto{(\mathbf Y)}{4pt}}_I\cup\,\kappa^{\scaleto{(\mathbf Y)}{4pt}}_o$. 
%$u>_{\kappa^{\scaleto{(\mathbf X)}{3pt}}} v$ follows from  \ref{KATEGOR_embedding}.Now
From
$
x
<^{\scaleto{(\mathbf X)}{4pt}}_{uv}
\varsigma^{\scaleto{(\mathbf X)}{4pt}}_{v\to uv}(\gamma^{\scaleto{(\mathbf X)}{4pt}}_v(y))
$,
$
G^{\scaleto{(\mathbf X)}{4pt}}_{uv}
\overset{(\ref{tipusSIGMA})}{\ni}
\varsigma^{\scaleto{(\mathbf X)}{4pt}}_{u\to uv}(\gamma^{\scaleto{(\mathbf X)}{4pt}}_{u}(x))
\overset{(\ref{DEFgamma})}{\leq^{\scaleto{(\mathbf X)}{4pt}}_{uv}} 
\varsigma^{\scaleto{(\mathbf X)}{4pt}}_{v\to uv}(\gamma^{\scaleto{(\mathbf X)}{4pt}}_v(y))
\overset{(\ref{tipusSIGMA})}{\in}
G^{\scaleto{(\mathbf X)}{4pt}}_{uv}
$
and in turn,
$
\varsigma^{\scaleto{(\mathbf X)}{4pt}}_{u\to uv}(\gamma^{\scaleto{(\mathbf X)}{4pt}}_{u}(x))
\preceq^{\scaleto{(\mathbf X)}{4pt}}_{uv}
\varsigma^{\scaleto{(\mathbf X)}{4pt}}_{v\to uv}(\gamma^{\scaleto{(\mathbf X)}{4pt}}_v(y))
$
follows by (\ref{KATEGOR_KibovitettRendezesITTIS}).
Hence by \ref{KATEGOR_homomorfizmus}
$
\Phi_{uv}(\varsigma^{\scaleto{(\mathbf X)}{4pt}}_{u\to uv}(\gamma^{\scaleto{(\mathbf X)}{4pt}}_{u}(x)))
\preceq^{\scaleto{(\mathbf Y)}{4pt}}_{\Phi(uv)}
\Phi_{uv}(\varsigma^{\scaleto{(\mathbf X)}{4pt}}_{v\to uv}(\gamma^{\scaleto{(\mathbf X)}{4pt}}_v(y)))
$
follows.
Therefore,
by 
(\ref{KATEGOR_KibovitettRendezesITTIS}),
\begin{equation}\label{Kiindulas}
\Phi_{uv}(\varsigma^{\scaleto{(\mathbf X)}{4pt}}_{u\to uv}(\gamma^{\scaleto{(\mathbf X)}{4pt}}_{u}(x)))
\leq^{\scaleto{(\mathbf Y)}{4pt}}_{\Phi(uv)}
\Phi_{uv}(\varsigma^{\scaleto{(\mathbf X)}{4pt}}_{v\to uv}(\gamma^{\scaleto{(\mathbf X)}{4pt}}_v(y)))
.
\end{equation}
\begin{enumerate}[b1)]
\item 
Assume $\Phi(u)=\Phi(v)$.
Then (\ref{KATEGOR_RendeZesINNOVATIVANmasodikNEMKELLc}) is equivalent to
\begin{equation}\label{75JHCJK7Nl}
\rho^{\scaleto{(\mathbf Y)}{4pt}}_{\Phi(u)}(\varphi(x))
\leq^{\scaleto{(\mathbf Y)}{4pt}}_{\Phi(u)}\rho^{\scaleto{(\mathbf Y)}{4pt}}_{\Phi(u)}(\varphi(y)).
\end{equation}
It is equivalent to (\ref{Kiindulas})
by (\ref{ezLETT})
if 
$(\Phi(v)=)\Phi(u)\in\kappa_o^{\scaleto{(\mathbf Y)}{4pt}}$.
Assume
$(\Phi(v)=)\Phi(u)\in\kappa_I^{\scaleto{(\mathbf Y)}{4pt}}$.
Since $u>_{\kappa^{\scaleto{(\mathbf X)}{3pt}}} v$, it holds true that 
$\Phi_{uv}(\varsigma^{\scaleto{(\mathbf X)}{4pt}}_{v\to uv}(\gamma^{\scaleto{(\mathbf X)}{4pt}}_v(y)))
\overset{\ref{KATEGOR_G2},\ref{KATEGOR_homomorfizmus}}{\in}
{H^{\scaleto{(\mathbf Y)}{4pt}}_{\Phi(uv)}}
$.
Either equality holds in (\ref{Kiindulas}) and then 
$
\accentset{\bullet}
{\Phi_{uv}(\varsigma^{\scaleto{(\mathbf X)}{4pt}}_{u\to uv}(\gamma^{\scaleto{(\mathbf X)}{4pt}}_{u}(x)))}
=
\accentset{\bullet}
{\Phi_{uv}(\varsigma^{\scaleto{(\mathbf X)}{4pt}}_{v\to uv}(\gamma^{\scaleto{(\mathbf X)}{4pt}}_v(y)))}
$
follows from (\ref{Kiindulas}),
or there is strict inequality in (\ref{Kiindulas}) and then 
$
{\Phi_{uv}(\varsigma^{\scaleto{(\mathbf X)}{4pt}}_{u\to uv}(\gamma^{\scaleto{(\mathbf X)}{4pt}}_{u}(x)))}
\leq^{\scaleto{(\mathbf Y)}{4pt}}_{\Phi(uv)}
\accentset{\bullet}
{\Phi_{uv}(\varsigma^{\scaleto{(\mathbf X)}{4pt}}_{v\to uv}(\gamma^{\scaleto{(\mathbf X)}{4pt}}_v(y)))}
$
follows from (\ref{Kiindulas}).
In both cases, 
(\ref{ezLETT}) shows that (\ref{75JHCJK7Nl}) holds.

\item 
Assume 
$\Phi(u)>_{\kappa^{\scaleto{(\mathbf Y)}{3pt}}} \Phi(v)$.
Then
$\Phi(u)\in\kappa_I^{\scaleto{(\mathbf Y)}{4pt}}$
since
$
\kappa_o^{\scaleto{(\mathbf Y)}{4pt}}
\ni
\Phi(u)
\overset{Table~\ref{KATEGOR_ThetaPsiOmegaAGAIN}}{=}
t^{\scaleto{(\mathbf Y)}{4pt}}
>_{\kappa^{\scaleto{(\mathbf Y)}{3pt}}} \Phi(v)
$
is impossible.
Since $\Phi(u)>_{\kappa^{\scaleto{(\mathbf Y)}{3pt}}} \Phi(v)$,
(\ref{KATEGOR_RendeZesINNOVATIVANmasodikNEMKELLc}) is equivalent to
\begin{equation}\label{75JHCJK7Nl2}
\rho^{\scaleto{(\mathbf Y)}{4pt}}_{\Phi(u)}(\varphi(x))
<^{\scaleto{(\mathbf Y)}{4pt}}_{\Phi(u)}\rho^{\scaleto{(\mathbf Y)}{4pt}}_{\Phi(u)}(\varphi(y)).
\end{equation}
Since 
$u>_{\kappa^{\scaleto{(\mathbf X)}{3pt}}} v$
and
$x\in \accentset{\bullet}H^{\scaleto{(\mathbf X)}{4pt}}_u$,
$
\varsigma^{\scaleto{(\mathbf X)}{4pt}}_{u\to uv}(\gamma^{\scaleto{(\mathbf X)}{4pt}}_{u}(x))
\overset{\ref{IDes}}{=}
\gamma^{\scaleto{(\mathbf X)}{4pt}}_{u}(x)
\overset{(\ref{DEFgamma})}{\in}
H^{\scaleto{(\mathbf X)}{4pt}}_{uv}$,
and hence
$
\Phi_{uv}(\varsigma^{\scaleto{(\mathbf X)}{4pt}}_{u\to uv}(\gamma^{\scaleto{(\mathbf X)}{4pt}}_{u}(x)))
\overset{\ref{KATEGOR_subgroups_and_complements}}{\in} 
H^{\scaleto{(\mathbf Y)}{4pt}}_{\Phi(uv)}
$.
Therefore, 
$
\accentset{\bullet}
{
\Phi_{uv}(\varsigma^{\scaleto{(\mathbf X)}{4pt}}_{u\to uv}(\gamma^{\scaleto{(\mathbf X)}{4pt}}_{u}(x)))
}
<^{\scaleto{(\mathbf Y)}{4pt}}_{\Phi(uv)}
\Phi_{uv}(\varsigma^{\scaleto{(\mathbf X)}{4pt}}_{v\to uv}(\gamma^{\scaleto{(\mathbf X)}{4pt}}_v(y)))
$
follows from (\ref{Kiindulas}) by (\ref{KATEGOR_KibovitettRendezesITTIS}),
which is equivalent to (\ref{75JHCJK7Nl2})
by (\ref{ezLETT}).
\end{enumerate}
\end{enumerate}
\end{enumerate}

\bigskip
\noindent
\ref{KATEGOR_ReSTrict}:
Let $x,y\in X$.
Then $x\in L^{\scaleto{(\mathbf X)}{4pt}}_u$ and $y\in L^{\scaleto{(\mathbf X)}{4pt}}_v$ for some $u,v\in\kappa^{\scaleto{(\mathbf X)}{4pt}}$, c.f.\ (\ref{KATEGOR_EZazX}). 
First we prove that if $\Phi(uv)\in\kappa_I^{\scaleto{(\mathbf Y)}{4pt}}$
then
\begin{equation}\label{NaVegreHogy}
\begin{array}{l}
\mbox{
$\rho^{\scaleto{(\mathbf X)}{4pt}}_{uv}(x)\in H^{\scaleto{(\mathbf X)}{4pt}}_{uv}$
if and only if 
$
\rho^{\scaleto{(\mathbf Y)}{4pt}}_{\Phi(uv)}(\varphi(x))
\in H^{\scaleto{(\mathbf Y)}{4pt}}_{\Phi(uv)}
$.
}
\\
\mbox{
$\rho^{\scaleto{(\mathbf X)}{4pt}}_{uv}(y)\in H^{\scaleto{(\mathbf X)}{4pt}}_{uv}$
if and only if 
$
\rho^{\scaleto{(\mathbf Y)}{4pt}}_{\Phi(uv)}(\varphi(y))
\in H^{\scaleto{(\mathbf Y)}{4pt}}_{\Phi(uv)}
$.
}
\end{array}
\end{equation}
Indeed, 
by (\ref{KATEGOR_ViSSZaEpulCALC}), 
if $\Phi(u)\in\kappa_I^{\scaleto{(\mathbf Y)}{4pt}}$
then
$\varphi(x)\in G^{\scaleto{(\mathbf Y)}{4pt}}_{\Phi(u)}$
is equivalent to
$x\in G^{\scaleto{(\mathbf X)}{4pt}}_u$.
Therefore, by \ref{KATEGOR_subgroups_and_complements}, \ref{KATEGOR_particio}, and the first row of (\ref{KATEGOR_ViSSZaEpulCALC}), 
$\varphi(x)\in H^{\scaleto{(\mathbf Y)}{4pt}}_{\Phi(u)}$ is equivalent to
$x\in H^{\scaleto{(\mathbf X)}{4pt}}_u$.
Therefore, if
$u\geq_{\kappa^{\scaleto{(\mathbf X)}{3pt}}} v$ then 
$
\rho^{\scaleto{(\mathbf Y)}{4pt}}_{\Phi(uv)}(\varphi(x))
\overset{(\ref{KATEGOR_P5})}{=}
\varphi(x)
\in 
H^{\scaleto{(\mathbf Y)}{4pt}}_{\Phi(uv)}
$
is equivalent to
$
H^{\scaleto{(\mathbf X)}{4pt}}_{uv}
\ni
x
\overset{(\ref{KATEGOR_P5})}{=}
\rho^{\scaleto{(\mathbf X)}{4pt}}_{uv}(x)
$,
whereas if $u<_{\kappa^{\scaleto{(\mathbf X)}{3pt}}} v$
then
both $
\rho^{\scaleto{(\mathbf X)}{4pt}}_{uv}(x)
\overset{(\ref{KATEGOR_P5})}{=}
\varsigma^{\scaleto{(\mathbf X)}{4pt}}_{u\to uv}(\gamma^{\scaleto{(\mathbf X)}{4pt}}_u(x))
\overset{\ref{KATEGOR_G2}}{\in}
H^{\scaleto{(\mathbf X)}{4pt}}_{uv}
$
and
$
\rho^{\scaleto{(\mathbf Y)}{4pt}}_{\Phi(uv)}(\varphi(x))
\overset{(\ref{KATEGOR_P5})}{=}
\varsigma^{\scaleto{(\mathbf Y)}{4pt}}_{\Phi(u)\to\Phi(uv)}(\gamma^{\scaleto{(\mathbf Y)}{4pt}}_{\Phi(u)}(x))
\overset{\ref{KATEGOR_G2}}{\in}
H^{\scaleto{(\mathbf Y)}{4pt}}_{\Phi(uv)}
$
hold true. Analogous proof works for the second row of (\ref{NaVegreHogy}).
Also we claim

\begin{equation}\label{CserBerE} 
\begin{array}{rll}
\varsigma_{u\to uv}^{\scaleto{(\mathbf X)}{4pt}}(\gamma^{\scaleto{(\mathbf X)}{4pt}}_u(x))
&
=
&
\gamma^{\scaleto{(\mathbf X)}{4pt}}_{uv}
(
\rho^{\scaleto{(\mathbf X)}{4pt}}_{uv}(x)
),
\\
\varsigma_{v\to uv}^{\scaleto{(\mathbf X)}{4pt}}(\gamma^{\scaleto{(\mathbf X)}{4pt}}_v(y))
&
=
&
\gamma^{\scaleto{(\mathbf X)}{4pt}}_{uv}
(
\rho^{\scaleto{(\mathbf X)}{4pt}}_{uv}(y)
)
,
\end{array}
\end{equation}
since
$$
\begin{array}{rll}
\varsigma_{u\to uv}^{\scaleto{(\mathbf X)}{4pt}}(\gamma^{\scaleto{(\mathbf X)}{4pt}}_u(x))
&
\overset{\ref{IDes}}{=} 
&
\left\{
\begin{array}{rr}
\varsigma_{u\to uv}^{\scaleto{(\mathbf X)}{4pt}}(\gamma^{\scaleto{(\mathbf X)}{4pt}}_u(x))
&
\mbox{if $u<_{\kappa^{\scaleto{(\mathbf X)}{4pt}}} v$}\\
\gamma^{\scaleto{(\mathbf X)}{4pt}}_u(x)
&
\mbox{if $u\geq_{\kappa^{\scaleto{(\mathbf X)}{4pt}}} v$}
\end{array}
\right.
\\
&
\overset{(\ref{tipusSIGMA}),\,(\ref{DEFgamma})}{=} 
&
\left\{
\begin{array}{rr}
\gamma^{\scaleto{(\mathbf X)}{4pt}}_{uv}(\varsigma_{u\to uv}^{\scaleto{(\mathbf X)}{4pt}}(\gamma^{\scaleto{(\mathbf X)}{4pt}}_u(x)))
&
\mbox{if $u<_{\kappa^{\scaleto{(\mathbf X)}{4pt}}} v$}\\
\gamma^{\scaleto{(\mathbf X)}{4pt}}_u(x)
&
\mbox{if $u\geq_{\kappa^{\scaleto{(\mathbf X)}{4pt}}} v$}
\end{array}
\right.
\\
&
\overset{(\ref{KATEGOR_P5})}{=} 
&
\gamma^{\scaleto{(\mathbf X)}{4pt}}_{uv}
(
\rho^{\scaleto{(\mathbf X)}{4pt}}_{uv}(x)
)
.
\end{array}
$$
Now,
$$
\small
%\footnotesize
%\tiny
\begin{array}{ccl}
\varphi(\giksz{x}{y}) 
& 
\overset{(\ref{KATEGOR_EgySzeruTe})}{=}
&
\varphi\left(
\gteMX{uv}
{\rho^{\scaleto{(\mathbf X)}{4pt}}_{uv}(x)}
{\rho^{\scaleto{(\mathbf X)}{4pt}}_{uv}(y)}
\right)
\\
&
\overset{(\ref{KATEGOR_uPRODigySHORT})}{=}
&
\left\{
\begin{array}{l}
\varphi\left(
{\gamma^{\scaleto{(\mathbf X)}{4pt}}_{uv}(\rho^{\scaleto{(\mathbf X)}{4pt}}_{uv}(x))}
\cdot^{\scaleto{(\mathbf X)}{4pt}}_{uv}
{\gamma^{\scaleto{(\mathbf X)}{4pt}}_{uv}(\rho^{\scaleto{(\mathbf X)}{4pt}}_{uv}(y))}
^\bullet
\right)
\\
\hskip1cm \mbox{ if $uv\in\kappa^{\scaleto{(\mathbf X)}{3pt}}_I$, 
$
{\gamma^{\scaleto{(\mathbf X)}{4pt}}_{uv}(\rho^{\scaleto{(\mathbf X)}{4pt}}_{uv}(x))}
\cdot^{\scaleto{(\mathbf X)}{4pt}}_{uv}
{\gamma^{\scaleto{(\mathbf X)}{4pt}}_{uv}(\rho^{\scaleto{(\mathbf X)}{4pt}}_{uv}(y))}
\in H^{\scaleto{(\mathbf X)}{4pt}}_{uv}
$,
\
}
\\
\hfill
\mbox{
$\neg(
\rho^{\scaleto{(\mathbf X)}{4pt}}_{uv}(x),
\rho^{\scaleto{(\mathbf X)}{4pt}}_{uv}(y)
\in H^{\scaleto{(\mathbf X)}{4pt}}_{uv}
)$
}
\\
\varphi\left(
{\gamma^{\scaleto{(\mathbf X)}{4pt}}_{uv}(\rho^{\scaleto{(\mathbf X)}{4pt}}_{uv}(x))}
\cdot^{\scaleto{(\mathbf X)}{4pt}}_{uv}
{\gamma^{\scaleto{(\mathbf X)}{4pt}}_{uv}(\rho^{\scaleto{(\mathbf X)}{4pt}}_{uv}(y))}
\right)
\\
	 \hskip1cm \mbox{ otherwise}\\
\end{array}
\right.
\\
&
\overset{
%{\tiny
%\mbox{$
%{\gamma^{\scaleto{(\mathbf X)}{4pt}}_{uv}(\rho^{\scaleto{(\mathbf X)}{4pt}}_{uv}(x))}
%\cdot^{\scaleto{(\mathbf X)}{4pt}}_{uv}
%{\gamma^{\scaleto{(\mathbf X)}{4pt}}_{uv}(\rho^{\scaleto{(\mathbf X)}{4pt}}_{uv}(y))}
%\in G^{\scaleto{(\mathbf X)}{4pt}}_{uv}
%$}},
%{\tiny\mbox{$arg_\varphi\in G^{\scaleto{(\mathbf X)}{4pt}}_{uv}$}},
(\ref{KATEGOR_ViSSZaEpul}),\,(\ref{DEFgamma})
}{=}
&
\left\{
\begin{array}{l}
\Phi_{uv}
\left(
{\gamma^{\scaleto{(\mathbf X)}{4pt}}_{uv}(\rho^{\scaleto{(\mathbf X)}{4pt}}_{uv}(x))}
\cdot^{\scaleto{(\mathbf X)}{4pt}}_{uv}
{\gamma^{\scaleto{(\mathbf X)}{4pt}}_{uv}(\rho^{\scaleto{(\mathbf X)}{4pt}}_{uv}(y))}
\right)^\bullet
\\
	 \hskip1cm 
\mbox{
{
\color{midgrey}
if
\color{black}
$\Phi(uv)\in\kappa^{\scaleto{(\mathbf Y)}{3pt}}_I$, 
\color{midgrey}
$uv\in\kappa^{\scaleto{(\mathbf X)}{3pt}}_I$,} 
\color{midgrey}
$
{\gamma^{\scaleto{(\mathbf X)}{4pt}}_{uv}(\rho^{\scaleto{(\mathbf X)}{4pt}}_{uv}(x))}
\cdot^{\scaleto{(\mathbf X)}{4pt}}_{uv}
{\gamma^{\scaleto{(\mathbf X)}{4pt}}_{uv}(\rho^{\scaleto{(\mathbf X)}{4pt}}_{uv}(y))}
\in H^{\scaleto{(\mathbf X)}{4pt}}_{uv}
$,
},
\\
\hfill
\color{midgrey}
\mbox{
$
\neg(
\rho^{\scaleto{(\mathbf X)}{4pt}}_{uv}(x),
\rho^{\scaleto{(\mathbf X)}{4pt}}_{uv}(y)
\in H^{\scaleto{(\mathbf X)}{4pt}}_{uv})
$,
}
\\
\Phi_{uv}
\left(
{\gamma^{\scaleto{(\mathbf X)}{4pt}}_{uv}(\rho^{\scaleto{(\mathbf X)}{4pt}}_{uv}(x))}
\cdot^{\scaleto{(\mathbf X)}{4pt}}_{uv}
{\gamma^{\scaleto{(\mathbf X)}{4pt}}_{uv}(\rho^{\scaleto{(\mathbf X)}{4pt}}_{uv}(y))}
\right)
	\\  \hskip1cm \mbox{{\color{midgrey} otherwise}}\\
\end{array}
\right.
\\
&
\overset{\ref{KATEGOR_particio}}{=}
&
\left\{
\color{midgrey}
\begin{array}{l}
\Phi_{uv}
\left(
{\gamma^{\scaleto{(\mathbf X)}{4pt}}_{uv}(\rho^{\scaleto{(\mathbf X)}{4pt}}_{uv}(x))}
\cdot^{\scaleto{(\mathbf X)}{4pt}}_{uv}
{\gamma^{\scaleto{(\mathbf X)}{4pt}}_{uv}(\rho^{\scaleto{(\mathbf X)}{4pt}}_{uv}(y))}
\right)^\bullet
\\
	 \hskip1cm \mbox{ 
\color{black}
if $\Phi(uv)\in\kappa^{\scaleto{(\mathbf Y)}{3pt}}_I$, 
$
{\gamma^{\scaleto{(\mathbf X)}{4pt}}_{uv}(\rho^{\scaleto{(\mathbf X)}{4pt}}_{uv}(x))}
\cdot^{\scaleto{(\mathbf X)}{4pt}}_{uv}
{\gamma^{\scaleto{(\mathbf X)}{4pt}}_{uv}(\rho^{\scaleto{(\mathbf X)}{4pt}}_{uv}(y))}
\in H^{\scaleto{(\mathbf X)}{4pt}}_{uv}
$,
}
\\
\hfill
\mbox{$\neg(
\rho^{\scaleto{(\mathbf X)}{4pt}}_{uv}(x),
\rho^{\scaleto{(\mathbf X)}{4pt}}_{uv}(y)
\in H^{\scaleto{(\mathbf X)}{4pt}}_{uv})$
}
\\
\Phi_{uv}
\left(
{\gamma^{\scaleto{(\mathbf X)}{4pt}}_{uv}(\rho^{\scaleto{(\mathbf X)}{4pt}}_{uv}(x))}
\cdot^{\scaleto{(\mathbf X)}{4pt}}_{uv}
{\gamma^{\scaleto{(\mathbf X)}{4pt}}_{uv}(\rho^{\scaleto{(\mathbf X)}{4pt}}_{uv}(y))}
\right)
	\\  \hskip1cm \mbox{{\color{midgrey} otherwise}}\\
\end{array}
\right.
\\
&
\overset{(\ref{NaVegreHogy}),\,(\ref{CserBerE})}{=}
&
\left\{
\color{midgrey}
\begin{array}{l}
\Phi_{uv}
\left(
{\color{black}
\varsigma^{\scaleto{(\mathbf X)}{4pt}}_{u\to uv}(\gamma^{\scaleto{(\mathbf X)}{4pt}}_u(x))}
\cdot^{\scaleto{(\mathbf X)}{4pt}}_{uv}
{\color{black}
\varsigma^{\scaleto{(\mathbf X)}{4pt}}_{v\to uv}(\gamma^{\scaleto{(\mathbf X)}{4pt}}_v(y))}
\right)^\bullet
\\
	 \hskip1cm \mbox{ 
\color{midgrey}
if $\Phi(uv)\in\kappa^{\scaleto{(\mathbf Y)}{3pt}}_I$, 
$
{\varsigma^{\scaleto{(\mathbf X)}{4pt}}_{u\to uv}(\gamma^{\scaleto{(\mathbf X)}{4pt}}_u(x))}
\cdot^{\scaleto{(\mathbf X)}{4pt}}_{uv}
{\varsigma^{\scaleto{(\mathbf X)}{4pt}}_{v\to uv}(\gamma^{\scaleto{(\mathbf X)}{4pt}}_v(y))}
\in H^{\scaleto{(\mathbf X)}{4pt}}_{uv}
$,
}
\\
\hfill
\color{midgrey}
\mbox{
\color{black}
$\neg\left(
\rho^{\scaleto{(\mathbf Y)}{4pt}}_{\Phi(uv)}(\varphi(x)),
\rho^{\scaleto{(\mathbf Y)}{4pt}}_{\Phi(uv)}(\varphi(y))
\in H^{\scaleto{(\mathbf Y)}{4pt}}_{\Phi(uv)}
\right)$,
}
\\
\Phi_{uv}
\left(
{\color{black}
\varsigma^{\scaleto{(\mathbf X)}{4pt}}_{u\to uv}(\gamma^{\scaleto{(\mathbf X)}{4pt}}_u(x))}
\cdot^{\scaleto{(\mathbf X)}{4pt}}_{uv}
{\color{black}
\varsigma^{\scaleto{(\mathbf X)}{4pt}}_{v\to uv}(\gamma^{\scaleto{(\mathbf X)}{4pt}}_v(y))}
\right)
	\\  \hskip1cm \mbox{{\color{midgrey} otherwise}}\\
\end{array}
\right.
\\
&
\overset{\ref{KATEGOR_homomorfizmus},\,\ref{KATEGOR_subgroups_and_complements}}{=}
&
\left\{
\begin{array}{l}
\left(
\Phi_{uv}\left({\varsigma^{\scaleto{(\mathbf X)}{4pt}}_{u\to uv}(\gamma^{\scaleto{(\mathbf X)}{4pt}}_u(x))}\right)
\cdot^{\scaleto{(\mathbf Y)}{3pt}}_{\Phi(uv)}
\Phi_{uv}\left({\varsigma^{\scaleto{(\mathbf X)}{4pt}}_{u\to uv}(\gamma^{\scaleto{(\mathbf X)}{4pt}}_v(y))}\right)
\right)^\bullet

\\
	 \hskip1cm \mbox{ \color{midgrey}
	if $\Phi(uv)\in\kappa^{\scaleto{(\mathbf Y)}{3pt}}_I$, 
$\neg\left(
\rho^{\scaleto{(\mathbf Y)}{4pt}}_{\Phi(uv)}(\varphi(x)),
\rho^{\scaleto{(\mathbf Y)}{4pt}}_{\Phi(uv)}(\varphi(y))
\in H^{\scaleto{(\mathbf Y)}{4pt}}_{\Phi(uv)}
\right)$,
}
\\
\hskip1cm 
\mbox{
\color{black}
$
\Phi_{uv}(\varsigma_{u\to uv}^{\scaleto{(\mathbf X)}{4pt}}(\gamma^{\scaleto{(\mathbf X)}{4pt}}_u(x)))
\cdot^{\scaleto{(\mathbf Y)}{3pt}}_{\Phi(uv)}
\Phi_{uv}(\varsigma_{v\to uv}^{\scaleto{(\mathbf X)}{4pt}}(\gamma^{\scaleto{(\mathbf X)}{4pt}}_v(y)))
\in H^{\scaleto{(\mathbf Y)}{4pt}}_{\Phi(uv)}$,}\\
\Phi_{uv}
\left({\varsigma^{\scaleto{(\mathbf X)}{4pt}}_{u\to uv}(\gamma^{\scaleto{(\mathbf X)}{4pt}}_u(x))}\right)
\cdot^{\scaleto{(\mathbf Y)}{3pt}}_{\Phi(uv)}
\Phi_{uv}\left({\varsigma^{\scaleto{(\mathbf X)}{4pt}}_{u\to uv}(\gamma^{\scaleto{(\mathbf X)}{4pt}}_v(y))}\right)
	\\  \hskip1cm \mbox{{\color{midgrey} otherwise}}\\
\end{array}
\right.
\\ 
&
\overset{(\ref{Atugorja})}{=}
&
\left\{
\begin{array}{ll}
\left(
{\gamma^{\scaleto{(\mathbf Y)}{4pt}}_{\Phi(uv)}(\rho^{\scaleto{(\mathbf Y)}{4pt}}_{\Phi(uv)}(\varphi(x)))}
\cdot^{\scaleto{(\mathbf Y)}{3pt}}_{\Phi(uv)}
{\gamma^{\scaleto{(\mathbf Y)}{4pt}}_{\Phi(uv)}(\rho^{\scaleto{(\mathbf Y)}{4pt}}_{\Phi(uv)}(\varphi(y)))}
\right)^\bullet
\\
\hskip1cm
\mbox{ 
\color{midgrey}
if $\Phi(uv)\in\kappa^{\scaleto{(\mathbf Y)}{4pt}}_I$, 
$
\color{black}
{\gamma^{\scaleto{(\mathbf Y)}{4pt}}_{\Phi(uv)}(\rho^{\scaleto{(\mathbf Y)}{4pt}}_{\Phi(uv)}(\varphi(x)))}
\cdot^{\scaleto{(\mathbf Y)}{3pt}}_{\Phi(uv)}
{\gamma^{\scaleto{(\mathbf Y)}{4pt}}_{\Phi(uv)}(\rho^{\scaleto{(\mathbf Y)}{4pt}}_{\Phi(uv)}(\varphi(y)))}
\color{midgrey}
\in H^{\scaleto{(\mathbf Y)}{4pt}}_{\Phi(uv)}$}
\\
\hfill
\color{midgrey}
\mbox{ and $\neg(
\rho^{\scaleto{(\mathbf Y)}{4pt}}_{\Phi(uv)}(\varphi(x)),
\color{midgrey}
\rho^{\scaleto{(\mathbf Y)}{4pt}}_{\Phi(uv)}(\varphi(y))
\in H^{\scaleto{(\mathbf Y)}{4pt}}_{\Phi(uv)}
)$}\\
{\gamma^{\scaleto{(\mathbf Y)}{4pt}}_{uv}(\rho^{\scaleto{(\mathbf Y)}{4pt}}_{\Phi(uv)}(\varphi(x)))}
\cdot^{\scaleto{(\mathbf Y)}{3pt}}_{\Phi(uv)}
{\gamma^{\scaleto{(\mathbf Y)}{4pt}}_{uv}(\rho^{\scaleto{(\mathbf Y)}{4pt}}_{\Phi(uv)}(\varphi(y)))}\\
\hskip1cm
\color{midgrey}
\mbox{ otherwise}\\
\end{array}
\right.
\\
&
\overset{(\ref{KATEGOR_uPRODigySHORT})}{=}
&
\gteMY{\Phi(uv)}
{\rho^{\scaleto{(\mathbf Y)}{4pt}}_{\Phi(uv)}(\varphi(x))}
{\rho^{\scaleto{(\mathbf Y)}{4pt}}_{\Phi(uv)}(\varphi(y))}
\\
&
\overset{(\ref{KATEGOR_EgySzeruTe})}{=}
&
\gipsz{\varphi(x)}{\varphi(y)}.

\end{array}
$$

%\eject 
%\pdfpageheight=11.6in
\bigskip
\noindent
\ref{KATEGOR_ReSTrictITE}:
For $x\in X$, $x\in L^{\scaleto{(\mathbf X)}{4pt}}_u$ it holds true that

\begin{equation}\label{RKompisOK}
\begin{array}{lll}
\varphi\left(\negaiksz{x}\right)
&\overset{(\ref{KATEGOR_SplitNegaSIMP})}{=}&
\left\{
\begin{array}{ll}
\varphi\left(\left(
\gamma^{\scaleto{(\mathbf X)}{4pt}}_u(x)
^{-1^{\scaleto{(\mathbf X)}{3pt}}_u}\right)^\bullet\right)
& \mbox{ if $u\in\kappa^{\scaleto{(\mathbf X)}{4pt}}_I$, $x\in H^{\scaleto{(\mathbf X)}{4pt}}_u$}\\
\varphi\left({\gamma^{\scaleto{(\mathbf X)}{4pt}}_u(x)^{-1^{\scaleto{(\mathbf X)}{3pt}}_u}}_{\downarrow_{u}}\right)
& \mbox{ if $u\in\kappa^{\scaleto{(\mathbf X)}{4pt}}_J$, $x\in G^{\scaleto{(\mathbf X)}{4pt}}_u$}\\
\varphi\left(\gamma^{\scaleto{(\mathbf X)}{4pt}}_u(x)^{-1^{\scaleto{(\mathbf X)}{3pt}}_u}\right)
%& \mbox{ if $u\in\kappa^{\scaleto{(\mathbf X)}{4pt}}_o$ or ($u\in\kappa^{\scaleto{(\mathbf X)}{4pt}}_I$, $x\in L^{\scaleto{(\mathbf X)}{4pt}}_u\setminus H^{\scaleto{(\mathbf X)}{4pt}}_u$)}\\
& \mbox{ otherwise}\\
\end{array}
\right.
\\
&\overset{\ref{KATEGOR_particio}}{=}&
\left\{
\color{midgrey}
\begin{array}{ll}

\varphi\left(\left(\gamma^{\scaleto{(\mathbf X)}{4pt}}_u(x)^{-1^{\scaleto{(\mathbf X)}{3pt}}_u}\right)^\bullet\right)
& \mbox{ if $u\in\kappa^{\scaleto{(\mathbf X)}{4pt}}_I$, 
{\color{black}
$\Phi(u)
\in
\kappa_o^{\scaleto{(\mathbf Y)}{4pt}}$,}
$x\in H^{\scaleto{(\mathbf X)}{4pt}}_u$}\\

\varphi\left(\left(\gamma^{\scaleto{(\mathbf X)}{4pt}}_u(x)^{-1^{\scaleto{(\mathbf X)}{3pt}}_u}\right)^\bullet\right)
& \mbox{ if %$u\in\kappa^{\scaleto{(\mathbf X)}{4pt}}_I$, 
{\color{black}
$\Phi(u)
\in
\kappa_I^{\scaleto{(\mathbf Y)}{4pt}}$,}
$x\in H^{\scaleto{(\mathbf X)}{4pt}}_u$}\\

\varphi\left({\gamma^{\scaleto{(\mathbf X)}{4pt}}_u(x)^{-1^{\scaleto{(\mathbf X)}{3pt}}_u}}_{\downarrow_{u}}
\right) 
& \mbox{ if $u\in\kappa^{\scaleto{(\mathbf X)}{4pt}}_J$,
{\color{black}
$\Phi(u)
\in
\kappa_o^{\scaleto{(\mathbf Y)}{4pt}}$,}
$x\in G^{\scaleto{(\mathbf X)}{4pt}}_u$}\\

\varphi\left({\gamma^{\scaleto{(\mathbf X)}{4pt}}_u(x)^{-1^{\scaleto{(\mathbf X)}{3pt}}_u}}_{\downarrow_{u}}\right) 
& \mbox{ if %$u\in\kappa^{\scaleto{(\mathbf X)}{4pt}}_J$, 
{\color{black}
$\Phi(u)
\in
\kappa_J^{\scaleto{(\mathbf Y)}{4pt}}$,}
$x\in G^{\scaleto{(\mathbf X)}{4pt}}_u$}\\

%\varphi\left(\gamma^{\scaleto{(\mathbf X)}{4pt}}_u(x)^{-1^{\scaleto{(\mathbf X)}{3pt}}_u}\right) & \mbox{ if $u\in\kappa^{\scaleto{(\mathbf X)}{4pt}}_o$, {\color{black} $\Phi(u)\in\kappa_o^{\scaleto{(\mathbf Y)}{4pt}}$} or \color{black} $\Phi(u)\in\kappa_I^{\scaleto{(\mathbf Y)}{4pt}}$, \color{midgrey} $x\in L^{\scaleto{(\mathbf X)}{4pt}}_u\setminus H^{\scaleto{(\mathbf X)}{4pt}}_u$ }\\

\varphi\left(\gamma^{\scaleto{(\mathbf X)}{4pt}}_u(x)^{-1^{\scaleto{(\mathbf X)}{3pt}}_u}\right) & \mbox{ otherwise}\\

\end{array}
\right.
\\
&\overset{(\ref{KATEGOR_ViSSZaEpul})}{=}&
\left\{
\color{midgrey}
\begin{array}{ll}
\color{black}
\Phi_u\left(\gamma^{\scaleto{(\mathbf X)}{4pt}}_u(x)^{-1^{\scaleto{(\mathbf X)}{3pt}}_u}\right)
\overset{\ref{KATEGOR_homomorfizmus}}{=}
\Phi_u(\gamma^{\scaleto{(\mathbf X)}{4pt}}_u(x))^{-1^{\scaleto{(\mathbf Y)}{3pt}}_{\Phi(u)}}
& 
\\
\hfill \mbox{ if $u\in\kappa^{\scaleto{(\mathbf X)}{4pt}}_I$, 
$\Phi(u)
\in
\kappa_o^{\scaleto{(\mathbf Y)}{4pt}}$,
$x\in H^{\scaleto{(\mathbf X)}{4pt}}_u$}\\

\color{black}
\Phi_u\left(\gamma^{\scaleto{(\mathbf X)}{4pt}}_u(x)^{-1^{\scaleto{(\mathbf X)}{3pt}}_u}\right)^\bullet
\overset{\ref{KATEGOR_homomorfizmus}}{=}
\left(\Phi_u(\gamma^{\scaleto{(\mathbf X)}{4pt}}_u(x))^{-1^{\scaleto{(\mathbf Y)}{3pt}}_{\Phi(u)}}\right)^\bullet
& 
\\
\hfill \mbox{ if %$u\in\kappa^{\scaleto{(\mathbf X)}{4pt}}_I$, 
$\Phi(u)
\in
\kappa_I^{\scaleto{(\mathbf Y)}{4pt}}$,
$x\in H^{\scaleto{(\mathbf X)}{4pt}}_u$}\\

\color{black}\Phi_u\left({\gamma^{\scaleto{(\mathbf X)}{4pt}}_u(x)^{-1^{\scaleto{(\mathbf X)}{3pt}}_u}}_{\downarrow_{u}}\right) 
\overset{\ref{S6b},\,\ref{KATEGOR_homomorfizmus}}{=}
\Phi_u(\gamma^{\scaleto{(\mathbf X)}{4pt}}_u(x))^{-1^{\scaleto{(\mathbf Y)}{3pt}}_{\Phi(u)}}

& 
\\
\hfill \mbox{ if $u\in\kappa^{\scaleto{(\mathbf X)}{4pt}}_J$,
$\Phi(u)
\in
\kappa_o^{\scaleto{(\mathbf Y)}{4pt}}$,
$x\in G^{\scaleto{(\mathbf X)}{4pt}}_u$}\\

\color{black}
\Phi_u\left({\gamma^{\scaleto{(\mathbf X)}{4pt}}_u(x)^{-1^{\scaleto{(\mathbf X)}{3pt}}_u}}_{\downarrow_{u}}\right)
\overset{\ref{S6a},\,\ref{KATEGOR_homomorfizmus}}{=}
\left(\Phi_u(\gamma^{\scaleto{(\mathbf X)}{4pt}}_u(x))^{-1^{\scaleto{(\mathbf Y)}{3pt}}_{\Phi(u)}}\right)
_{\downarrow_{\Phi(u)}}

& 
\\
\hfill \mbox{ if %$u\in\kappa^{\scaleto{(\mathbf X)}{4pt}}_J$, 
$\Phi(u)
\in
\kappa_J^{\scaleto{(\mathbf Y)}{4pt}}$,
$x\in G^{\scaleto{(\mathbf X)}{4pt}}_u$}\\

%\color{black}
%\Phi_u\left(\gamma^{\scaleto{(\mathbf X)}{4pt}}_u(x)^{-1^{\scaleto{(\mathbf X)}{3pt}}_u}\right)
%\overset{\ref{KATEGOR_homomorfizmus}}{=}
%%\overset{\ref{KATEGOR_homomorfizmus},\,(\ref{KATEGOR_ViSSZaEpul})}{=}
%\Phi_u(\gamma^{\scaleto{(\mathbf X)}{4pt}}_u(x))^{-1^{\scaleto{(\mathbf Y)}{3pt}}_{\Phi(u)}}
%& \mbox{ if $u\in\kappa^{\scaleto{(\mathbf X)}{4pt}}_o$,
%$\Phi(u)
%\in
%\kappa_o^{\scaleto{(\mathbf Y)}{4pt}}$
%or 
%$\Phi(u)\in\kappa_I^{\scaleto{(\mathbf Y)}{4pt}}$, 
%$x\in L^{\scaleto{(\mathbf X)}{4pt}}_u\setminus H^{\scaleto{(\mathbf X)}{4pt}}_u$
%}\\

\color{black}
\Phi_u\left(\gamma^{\scaleto{(\mathbf X)}{4pt}}_u(x)^{-1^{\scaleto{(\mathbf X)}{3pt}}_u}\right)
\overset{\ref{KATEGOR_homomorfizmus}}{=}
\Phi_u(\gamma^{\scaleto{(\mathbf X)}{4pt}}_u(x))^{-1^{\scaleto{(\mathbf Y)}{3pt}}_{\Phi(u)}}
& 
\\
\hfill \mbox{ otherwise}\\
\end{array}
\right.
\\
&=&
\left\{
\color{midgrey}
\begin{array}{ll}

\left(\Phi_u(\gamma^{\scaleto{(\mathbf X)}{4pt}}_u(x)))^{-1^{\scaleto{(\mathbf Y)}{3pt}}_{\Phi(u)}}\right)^\bullet
& \mbox{ if %$u\in\kappa^{\scaleto{(\mathbf X)}{4pt}}_I$, 
$\Phi(u)
\in
\kappa_I^{\scaleto{(\mathbf Y)}{4pt}}$,
\color{black}
$\varphi(x)
\overset{(\ref{KATEGOR_ViSSZaEpulCALC}),\,\ref{KATEGOR_subgroups_and_complements}}{\in}
H^{\scaleto{(\mathbf Y)}{4pt}}_{\Phi(u)}
$}\\

\left(\Phi_u(\gamma^{\scaleto{(\mathbf X)}{4pt}}_u(x)))^{-1^{\scaleto{(\mathbf Y)}{3pt}}_{\Phi(u)}}\right)
_{\downarrow_{\Phi(u)}}

& \mbox{ if %$u\in\kappa^{\scaleto{(\mathbf X)}{4pt}}_J$, 
$\Phi(u)
\in
\kappa_J^{\scaleto{(\mathbf Y)}{4pt}}$,
\color{black}
$\varphi(x)
\overset{(\ref{KATEGOR_ViSSZaEpulCALC})}{\in}
G^{\scaleto{(\mathbf Y)}{4pt}}_{\Phi(u)}
$
}\\

\Phi_u(\gamma^{\scaleto{(\mathbf X)}{4pt}}_u(x))^{-1^{\scaleto{(\mathbf Y)}{3pt}}_{\Phi(u)}}
& \mbox{ \color{midgrey}
otherwise}\\

%\Phi_u(\gamma^{\scaleto{(\mathbf X)}{4pt}}_u(x))^{-1^{\scaleto{(\mathbf Y)}{3pt}}_{\Phi(u)}}
%& \mbox{ \color{black}
%if 
%$\Phi(u)\in\kappa_o^{\scaleto{(\mathbf Y)}{4pt}}$
%or 
%$\Phi(u)\in\kappa_I^{\scaleto{(\mathbf Y)}{4pt}}$, 
%$
%\varphi(x)
%\overset{%(\ref{KATEGOR_IkszU}),\,
%(\ref{KATEGOR_ViSSZaEpulCALC}),\,\ref{KATEGOR_subgroups_and_complements}}{\in}
%L^{\scaleto{(\mathbf Y)}{4pt}}_{\Phi(u)}\setminus H^{\scaleto{(\mathbf Y)}{4pt}}_{\Phi(u)}
%$
%}\\

\end{array}
\right.
\\
&\overset{(\ref{EziStrUe})}{=}&
\left\{
\color{midgrey}
\begin{array}{ll}

\color{black}
\left(\gamma^{\scaleto{(\mathbf Y)}{4pt}}_{\Phi(u)}(\varphi(x)))^{-1^{\scaleto{(\mathbf Y)}{3pt}}_{\Phi(u)}}\right)^\bullet
%\left(\Phi_u(\gamma^{\scaleto{(\mathbf X)}{4pt}}_u(x)))^{-1^{\scaleto{(\mathbf Y)}{3pt}}_{\Phi(u)}}\right)^\bullet
& \mbox{ if %$u\in\kappa^{\scaleto{(\mathbf X)}{4pt}}_I$, 
$\Phi(u)
\in
\kappa_I^{\scaleto{(\mathbf Y)}{4pt}}$,
$\varphi(x)
\in
H^{\scaleto{(\mathbf Y)}{4pt}}_{\Phi(u)}
$}\\

\color{black}
\left(
\gamma^{\scaleto{(\mathbf Y)}{4pt}}_{\Phi(u)}(\varphi(x))^{-1^{\scaleto{(\mathbf Y)}{3pt}}_{\Phi(u)}}
\right)
_{\downarrow_{\Phi(u)}}

& \mbox{ if %$u\in\kappa^{\scaleto{(\mathbf X)}{4pt}}_J$, 
$\Phi(u)
\in
\kappa_J^{\scaleto{(\mathbf Y)}{4pt}}$,
$\varphi(x)
\in
G^{\scaleto{(\mathbf Y)}{4pt}}_{\Phi(u)}
$
}\\

\color{black}
\gamma^{\scaleto{(\mathbf Y)}{4pt}}_{\Phi(u)}(\varphi(x))^{-1^{\scaleto{(\mathbf Y)}{3pt}}_{\Phi(u)}}
& \mbox{ otherwise}\\

%\varphi(\gamma^{\scaleto{(\mathbf X)}{4pt}}_u(x))^{-1^{\scaleto{(\mathbf Y)}{3pt}}_{\Phi(u)}}
%& \mbox{ 
%if 
%$\Phi(u)\in\kappa_o^{\scaleto{(\mathbf Y)}{4pt}}$
%or 
%$\Phi(u)\in\kappa_I^{\scaleto{(\mathbf Y)}{4pt}}$, 
%$
%\varphi(x)
%\in
%L^{\scaleto{(\mathbf Y)}{4pt}}_{\Phi(u)}\setminus H^{\scaleto{(\mathbf Y)}{4pt}}_{\Phi(u)}
%$
%}\\

\end{array}
\right.
\\
&\overset{(\ref{KATEGOR_SplitNegaSIMP})}{=}&
\negaipsz{\varphi(x)}
,
\end{array}
\end{equation}
hence
$
\varphi\left(\resiksz{x}{y}\right)
\overset{(\ref{KATEGOR_IgYaReSi})}{=}
\varphi\left(\negaiksz{\left(\giksz{x}{\negaiksz{y}}\right)}\right)
\overset{(\ref{RKompisOK}),\,\ref{KATEGOR_ReSTrict}}{=}
\negaipsz{\left(\gipsz{\varphi(x)}{\negaipsz{\varphi(y)}}\right)}
\overset{(\ref{KATEGOR_IgYaReSi})}{=}
\resipsz{\varphi(x)}{\varphi(y)}
$.

\bigskip
\noindent
\ref{KATEGOR_TkegyenlOEk}:
$
\varphi(t^{\scaleto{(\mathbf X)}{4pt}})
\overset{(\ref{KATEGOR_ViSSZaEpul})}{=}
\Phi_{t^{\scaleto{(\mathbf X)}{4pt}}}(t^{\scaleto{(\mathbf X)}{4pt}})
\overset{(\ref{KATEGOR_tLESZez}),\,\ref{KATEGOR_legkisebbelem}}{=}
t^{\scaleto{(\mathbf Y)}{4pt}}
$
.

\bigskip
\noindent
\ref{KATEGOR_FkegyenlOEk}:
%:
$
\varphi(f^{\scaleto{(\mathbf X)}{4pt}})
\overset{(\ref{KATEGOR_tLESZaz})}{=}
\varphi\left(\negaiksz{t^{\scaleto{(\mathbf X)}{4pt}}}\right)
\overset{(\ref{RKompisOK})}{=}
\negaipsz{\varphi\left(t^{\scaleto{(\mathbf X)}{4pt}}\right)}
\overset{\ref{KATEGOR_TkegyenlOEk}}{=}
\negaipsz{\left(t^{\scaleto{(\mathbf Y)}{4pt}}\right)}
\overset{(\ref{KATEGOR_tLESZaz})}{=}
f^{\scaleto{(\mathbf Y)}{4pt}}
$.

\bigskip
\bigskip
\noindent
(3)
Since by (\ref{KATEGOR_EQmegszoRITOm}), $\Phi$ arises from $\varphi$ by restriction, it suffices to show that any bunch homomorphism $\Phi$ {\em uniquely} extends to a homomorphism $\varphi$. To this end, assume that $\varphi$ is a homomorphism,
%from $\mathbf X$ to $\mathbf Y$, 
and that its restriction to $G^{\scaleto{(\mathbf X)}{4pt}}$ is the bunch homomorphism $\Phi$.
We shall prove that $\varphi$ coincides with the one given in (\ref{KATEGOR_ViSSZaEpul}).
Since the first row of (\ref{KATEGOR_ViSSZaEpul}) obviously holds, 
let $x\in \accentset{\bullet}H^{\scaleto{(\mathbf X)}{4pt}}_{u}$ for some $u\in\kappa_I^{\scaleto{(\mathbf X)}{4pt}}$.
It holds true by (\ref{KATEGOR_XHiGYkESZUL}) that 
$x=\giksz{y}{\negaiksz{u}}$ for some 
$y\in H^{\scaleto{(\mathbf X)}{4pt}}_{u}
\overset{\ref{KATEGOR_G2},\,(\ref{PHIdomain})}{\subseteq}
G^{\scaleto{(\mathbf X)}{4pt}}$
and
$y
\overset{(\ref{DEFgamma})}{=}
\gamma^{\scaleto{(\mathbf X)}{4pt}}_u(x)$.
Therefore,
$\varphi(x)=\varphi(\giksz{y}{\negaiksz{u}})=\gipsz{\varphi(y)}{\negaipsz{\varphi(u)}}$,
where
$
\varphi(y)
\overset{y\in G^{\scaleto{(\mathbf X)}{3pt}}}{=}
\Phi(y)
\overset{(\ref{PhiSlice})}{=}
\Phi_u(y)
\overset{\ref{KATEGOR_subgroups_and_complements}}{\in}
H^{\scaleto{(\mathbf Y)}{4pt}}_{\Phi(u)}
$.
By \ref{KATEGOR_particio}, $\Phi(u)\in
\kappa_o^{\scaleto{(\mathbf Y)}{4pt}}
\cup
\kappa_I^{\scaleto{(\mathbf Y)}{4pt}}
$.
If
$
\Phi(u)
\in
\kappa_o^{\scaleto{(\mathbf Y)}{4pt}}
$
then 
$
\negaipsz{\varphi(u)}
\overset{u\in G^{\scaleto{(\mathbf X)}{3pt}}}{=}
\negaipsz{\Phi(u)}
\overset{(\ref{PhiSlice})}{=}
\negaipsz{\Phi_u(u)}
\overset{Table~\ref{KATEGOR_ThetaPsiOmegaAGAIN}}{=}
\negaipsz{t^{\scaleto{(\mathbf Y)}{4pt}}}
=
t^{\scaleto{(\mathbf Y)}{4pt}}
$
is the unit element for $L^{\scaleto{(\mathbf Y)}{4pt}}_{\Phi(u)}$, and hence 
$
\gipsz{\varphi(y)}{\negaipsz{\varphi(u)}}
=
\gipsz{\Phi_u(y)}{t^{\scaleto{(\mathbf Y)}{4pt}}}
=
\Phi_u(y)
=
\Phi_u(\gamma^{\scaleto{(\mathbf X)}{4pt}}_u(x))
$,
whereas if
$
\Phi(u)
\in
\kappa_I^{\scaleto{(\mathbf Y)}{4pt}}
$
then
$
\gipsz{\varphi(y)}{\negaipsz{\varphi(u)}}
=
\gipsz{\Phi_u(y)}{\negaipsz{\Phi_u(u)}}
\overset{(\ref{KATEGOR_XHiGYkESZUL})}{=}
\accentset{\bullet}{\Phi_u(y)}
=
\accentset{\bullet}{\Phi_u(\gamma^{\scaleto{(\mathbf X)}{4pt}}_u(x))}
$.
\end{proof}

\begin{proposition}
The category $\mathcal B_\mathcal G$ is well-defined.
\end{proposition}
\begin{proof}
We shall verify that 
the function composition of two bunch homomorphisms
%$\Phi:\pazocal X\to\pazocal Y$ and $\Phi:\pazocal Y\to\mathcal Z$ given by $\Psi\circ\Phi:\pazocal X\to\mathcal Z$, $\Psi\circ\Phi(x)=\Psi(\Phi(x))$
is a bunch homomorphism, that this composition is associative, 
and that the identity mapping is a bunch homomorphism and also an identity morphism.
Adapt the notation of Definition~\ref{DEFbunchHom} for three bunches of layer groups $\pazocal X$, $\pazocal Y$, and $\pazocal Z$, and
let $\Phi$ and $\Psi$ be bunch homomorphisms from $\pazocal X$ to $\pazocal Y$ and from $\pazocal Y$ to $\pazocal Z$, respectively.
Obviously, $\Psi\circ\Phi$ is of type 
$G^{\scaleto{(\mathbf X)}{4pt}} \to G^{\scaleto{(\mathbf Y)}{4pt}}$, and it is also obvious that  $\Psi\circ\Phi$ satisfies properties \ref{KATEGOR_embedding}--\ref{KATEGOR_particio}.
%, \ref{KATEGOR_homomorfizmus}, \ref{KATEGOR_commutes}, and \ref{KATEGOR_legkisebbelem}. 
As for \ref{KATEGOR_subgroups_and_complements}, if 
$u\in\kappa_I^{\scaleto{(\mathbf X)}{4pt}}$
and
$(\Psi\circ\Phi)(u)\in\kappa_I^{\scaleto{(\mathbf Z)}{4pt}}$
then from the latter, by \ref{KATEGOR_particio},
$\Phi(u)\in\kappa_I^{\scaleto{(\mathbf Y)}{4pt}}$ follows, and hence 
$\Psi\circ\Phi$
clearly satisfies
\ref{KATEGOR_subgroups_and_complements}, too.
Invoking \ref{KATEGOR_particio} in the very same manner as above shows that
$\Psi\circ\Phi$ satisfies \ref{S6a}. 
To prove \ref{S6b}, notice that since $u\in\kappa^{\scaleto{(\mathbf X)}{4pt}}_J$, it follows from \ref{KATEGOR_particio} that $\Phi(u)$ is either in $\kappa^{\scaleto{(\mathbf Y)}{4pt}}_J$ or in $\kappa^{\scaleto{(\mathbf Y)}{4pt}}_o$. In either case application of \ref{KATEGOR_szomszed} for $\Phi$ and then for $\Psi$ to concludes the proof of \ref{S6b}.
As for \ref{KATEGOR_injective}, from
$(\Psi\circ\Phi)(uv)\in \kappa_I^{\scaleto{(\mathbf Z)}{4pt}}$
it follows by \ref{KATEGOR_particio} that $\Phi(uv)\in \kappa_I^{\scaleto{(\mathbf Y)}{4pt}}$,
hence
$
\varsigma_{u\to uv}^{\scaleto{(\mathbf X)}{4pt}}(y)
\prec^{\scaleto{(\mathbf X)}{4pt}}_{uv}
\varsigma_{u\to uv}^{\scaleto{(\mathbf X)}{4pt}}(x)
\in
H^{\scaleto{(\mathbf X)}{4pt}}_{uv}
$
implies, by using \ref{KATEGOR_injective} for $\Phi$,
that 
$
\Phi_{uv}(\varsigma_{v\to uv}^{\scaleto{(\mathbf X)}{4pt}}(y)) 
\prec^{\scaleto{(\mathbf Y)}{4pt}}_{\Phi(uv)}
\Phi_{uv}(\varsigma_{u\to uv}^{\scaleto{(\mathbf X)}{4pt}}(x))
\overset{\ref{KATEGOR_subgroups_and_complements}}{\in}
H^{\scaleto{(\mathbf Y)}{4pt}}_{\Phi(uv)}$,
that is,
$
\varsigma_{\Phi(v)\to\Phi(uv)}^{\scaleto{(\mathbf Y)}{4pt}}(\Phi_{u}(y))
\prec^{\scaleto{(\mathbf Y)}{4pt}}_{\Phi(uv)}
\varsigma_{\Phi(u)\to\Phi(uv)}^{\scaleto{(\mathbf Y)}{4pt}}(\Phi_{u}(x))
\in
H^{\scaleto{(\mathbf Y)}{4pt}}_{\Phi(uv)}$
by \ref{KATEGOR_commutes},
and hence by \ref{KATEGOR_injective} for $\Psi$,
$
(\Psi_{\Phi(uv)}\circ\Phi_{uv})(\varsigma_{v\to uv}^{\scaleto{(\mathbf X)}{4pt}}(y)) 
\prec^{\scaleto{(\mathbf Y)}{4pt}}_{(\Psi\circ\Phi)(uv)}
(\Psi_{\Phi(uv)}\circ\Phi_{uv})(\varsigma_{u\to uv}^{\scaleto{(\mathbf X)}{4pt}}(x))
$.
\\
It is obvious to verify that the identity mapping on $G^{\scaleto{(\mathbf X)}{4pt}}$ 
satisfies \ref{KATEGOR_embedding}--\ref{KATEGOR_injective}, so it is a bunch homomorphism.
Composition of functions is associative, hence so is the composition of bunch homomorphisms, cf.\,(\ref{KATEGOR_}), which also makes it immediate that the identity mapping is an identity morphism.
\end{proof}

\begin{theorem}\label{MainCATtheo}
The functor $\Upsilon : \mathcal I^{\mathfrak c}_{\mathfrak 0\mathfrak 1} \to \mathcal B_\mathcal G$
from the category  of odd or even involutive FL$_e$-chains with FL$_e$-algebra homomorphisms
to the category of bunches of layer groups with bunch homomorphisms
given by
$\Upsilon\mathbf X=\pazocal G_{\mathbf X}$
(see Theorem~\ref{KATEGOR_BUNCH_X}/\ref{KATEGOR_errefere})
and
$\Upsilon\varphi=\Phi$
(see (\ref{KATEGOR_EQmegszoRITOm}))
defines a categorical equivalence,
its inverse functor $\Upsilon^{-1}$ being 
$\Upsilon^{-1}\pazocal G=\mathbf X_{\pazocal G}$
(see Theorem~\ref{KATEGOR_BUNCH_X}/\ref{KATEGOR_arrafere})
and
$\Upsilon^{-1}\Phi=\varphi$
(see (\ref{KATEGOR_ViSSZaEpulEredeti})).
\qed
\end{theorem}
\begin{proof}
A functor $F : C \to D$ is known to yield an equivalence of categories if and only if it is simultaneously fully faithful, i.e., for any two objects $c_1$ and $c_2$ of $C$, the map $Hom_C(c_1,c_2)\to Hom_D(Fc_1,Fc_2)$ induced by $F$ is bijective; and
essentially surjective (dense), i.e. each object $d$ in $D$ is isomorphic to an object of the form $Fc$, for $c$ in $C$.
$\Upsilon$ satisfies these by 
Theorem~\ref{KATEGOR_BUNCH_X} and Lemma~\ref{KATEGOR_HogyanLatszik}.
\end{proof}

\begin{remark}\label{ISOis}
In fact, the functor $\Upsilon$ exhibits a categorical isomorphism between $\mathcal I^{\mathfrak c}_{\mathfrak 0\mathfrak 1}$ and $\mathcal B_\mathcal G$ if in the definition of the latter the objects are modified in accordance with Remark~\ref{ModifiCaTO}.
\end{remark}

%Call a bunch of layer groups $\pazocal G$ odd if $t\in\kappa_0$, and even otherwise.
\begin{remark}
By Theorem~\ref{KATEGOR_BUNCH_X}/\ref{KATEGOR_arrafere}, odd (resp. even) 
involutive FL$_e$-chains correspond bunches of layer groups where $t\in\kappa_0$ (resp. $t\notin\kappa_0$).
By \cite[Example 8.2]{JenRepr2020}, 
even Sugihara chains correspond to bunches of layer groups of the form 
$\langle \mathbbm 1_u, \mathbbm 1_u, \varsigma_{u\to v} \rangle_{\langle\emptyset,\emptyset,\kappa,\leq_\kappa\rangle}$,
and
odd Sugihara chains correspond bunches of layer groups of the form
$\langle \mathbbm 1_u, \mathbbm 1_u, \varsigma_{u\to v} \rangle_{\langle\{t\},\emptyset,\kappa\setminus\{t\},\leq_\kappa\rangle}$,
where $\mathbbm 1$ denotes the trivial (one-element) group.
By restricting the objects of
$\mathcal I^{\mathfrak c}_{\mathfrak 0\mathfrak 1}$
and
$\mathcal B_{\mathcal G}$ 
to these classes, respectively,
as a corollary of Theorem~\ref{MainCATtheo}
we obtain categorical equivalences between the above described subcategories of
$\mathcal I^{\mathfrak c}_{\mathfrak 0\mathfrak 1}$
and
$\mathcal B_{\mathcal G}$, respectively.
\end{remark}

%Therefore, we obtain the following corollaries of Theorem~\ref{MainCATtheo}.

%\begin{theorem}
%\begin{enumerate}
%\item 
%The categories $\mathcal I^{\mathfrak c}_{\mathfrak 0}$ and $\mathcal B_{\mathcal G\mathfrak 0}$ are equivalent.
%\item 
%The categories $\mathcal I^{\mathfrak c}_{\mathfrak 1}$ and $\mathcal B_{\mathcal G\mathfrak 1}$ are equivalent.
%\end{enumerate}
%\end{theorem}
\section*{Acknowledgement} We thank the anonymous reviewer for the careful reading of our manuscript and the many suggestions to improve the presentation of it.

\end{document}